\documentclass[12pt,a4paper]{article}
\usepackage{graphicx}
\usepackage{amsmath}
\usepackage[cp1251]{inputenc}
\usepackage[T2A]{fontenc}
\usepackage{a4}
\usepackage{amsfonts}
\usepackage{amssymb}

\newtheorem{theorem}{Theorem}

\newtheorem{corollary}[theorem]{Corollary}

\newtheorem{lemma}[theorem]{Lemma}

\newtheorem{proposition}[theorem]{Proposition}
\newtheorem{remark}[theorem]{Remark}

\newenvironment{proof}[1][Proof]{\textbf{#1.} }{\ \rule{0.5em}{0.5em}}
\numberwithin{equation}{section}

\begin{document}

\textbf{On some weighted H\"{o}lder spaces  as a possible functional
framework for the thin film equation and other parabolic equations
with a degeneration at the boundary of a domain.}

\bigskip

\bigskip

S.P.Degtyarev

\bigskip

\textbf{Institute for applied mathematics and mechanics of Ukrainian
National academy of sciences, Donetsk, Ukraine, }

\textbf{State Institute for applied mathematics and mechanics,
Donetsk, Ukraine, }

R.Luxemburg str., 74, Donetsk, Ukraine, 83114

\bigskip

E-mail: degtyar@i.ua

\bigskip

\bigskip

\bigskip

\begin{abstract}
 The present paper is devoted to studying of some weighted
H\"{o}lder spaces. These spaces are designed in the way to serve as
a framework for studying different statements for the thin film
equations in weighted classes of smooth functions in the
multidimensional setting. These spaces can serve also for
considering of other equations with the degeneration on the boundary
of the domain of definition.
\end{abstract}

\bigskip

\bigskip



\bigskip

\bigskip

\bigskip

\section{Introduction.} \label{ss1}

 The present paper is devoted to studying of some weighted
H\"{o}lder spaces
$C_{n,\omega\gamma}^{m+\gamma,\frac{m+\gamma}{m}}$. These spaces are
designed in the way to serve as a framework for consideration
 of different statements for the thin film equations in weighted classes
of smooth functions in the multidimensional setting. These spaces
can serve also for considering of other equations with the
degeneration on the boundary of the domain of definition, for
example, in the spirit of \cite{D1}.

The literature on the subject of the thin film equations  is very
numerous but almost all results with sufficient regularity are
devoted to the case of one spatial variable. As a possible target
for
an application of the spaces $C_{n,\omega\gamma}^{m+\gamma,\frac{m+\gamma}{m}%
}$ we only mention the papers \cite{11}- \cite{J1}.

The spaces $C_{n,\omega\gamma}^{m+\gamma,\frac{m+\gamma}{m}}$ arise
at the considering linearised version of the thin film equations.
Let us explain this on the example for the thin film equation in the
case of partial wetting (see, for example, \cite{11} for the
accurate statement). Consider the thin film equation of fourth order
for an unknown function $h(x,t)$ (compare \cite{sg4})

\begin{equation}
\frac{\partial h}{\partial t}+\nabla\left(  h^{n}\nabla\Delta h-\beta\nabla
h\right)  =f(x,t)\ \ \text{in}\ \ \Omega, \label{B1.1}%
\end{equation}
 where $n>0$ is fixed, $\Omega$ is a half space $\Omega
=\{(x,t):x=(x^{\prime},x_{N})\in R^{N},x_{N}>0,t>0\}.$ Consider also
partial wetting conditions at $x_{N}=0$
\begin{equation}
h(x^{\prime},0,t)=0,\ \ \frac{\partial h}{\partial x_{N}}(x^{\prime},0,t)=1
\label{B1.2}%
\end{equation}
and an initial condition
\begin{equation}
h(x,0)=w(x). \label{B1.3}%
\end{equation}
From \eqref{B1.2} it follows that we must have for $w(x)$
\begin{equation}
w(x^{\prime},0)=0,\ \ \frac{\partial w}{\partial x_{N}}(x^{\prime},0)=1.
\label{B1.4}%
\end{equation}
Consequently, we have
\begin{equation}
w(x)\sim x_{N},\ \ x_{N}\rightarrow0. \label{B1.5}%
\end{equation}
The linearization of equation \eqref{B1.1} at the initial datum
$w(x)$ means that we denote in \eqref{B1.1} $h=w+u$ and extract
linear with respect to $u$ part (we also drop lower order terms).
Formally, one can just replace $h^{n}$ by $w^{n}$ in \eqref{B1.1}
and replace $h$ by $u$ in other places of this equation. Taking into
account \eqref{B1.5} and replacing $w$ by just $x_{N}$, we arrive at

\begin{equation}
\frac{\partial u}{\partial t}+\nabla(x_{N}^{n}\nabla\Delta u-\beta\nabla
u)=f(x,t)\ \ \text{in}\ \ \Omega. \label{000001}%
\end{equation}
For second order equations this procedure is described in details
in, for example, \cite{D15}, \cite{D16}, \cite{D1}, and for fourth
order see \cite{J1}, \cite{11}, \cite{G11} formula (13), \cite{G12}
formula (7).

If we are going to consider equations \eqref{000001} (and
correspondingly \eqref{B1.1}) in classes of H\"{o}lder functions we
have to consider $f(x,t)$ in \eqref{000001} from some (may be
weighted) H\"{o}lder class. This leads to the consideration of
$\nabla(x_{N}^{n}\nabla\Delta u)$ from the same weighted H\"{o}lder
class. In our definition below this will be the class
$C_{n,(n/4)\gamma}^{m+\gamma ,\frac{m+\gamma}{m}}$.  In the case of
second order equations such classes were used in fact in
\cite{Dask1}- \cite{BazKrasn}, \cite{D1}, where the papers
\cite{Dask1}- \cite{BazKrasn} are based on the
Carnot-Carath\'{e}odory metric and the paper \cite{D1} is based on
classes $C_{n,\omega\gamma}^{m+\gamma,\frac{m+\gamma }{m}}$. Note
that we consider the framework of classes
$C_{n,\omega\gamma}^{m+\gamma,\frac{m+\gamma}{m}}$ as an alternative
for considering the Carnot-Carath\'{e}odory metric for studying
degenerate equations in classes of smooth functions - \cite{Dask1}-
\cite{BazKrasn}, \cite{J1}.

Note that in the case of elliptic equations more simple weighted
H\"{o}lder classes with unweighted  H\"{o}lder constants   can be
used - \cite{Shima}, \cite{BazDeg}. The reason is that in the
elliptic case no agreement between smoothness in $x$-variables and
$t$- variable is needed.

Let us turn now to exact definitions and to the main results.

Denote $H=\{x=(x^{\prime},x_{N})\in R^{N}:\,x_{N}>0\}$ ,
$Q=\{(x,t):x\in H,-\infty<t<\infty\}$. And we note at once that all
the reasoning and statement below are valid in evident way also for
$Q^{+}=\{(x,t):x\in H,t\geq0\}$ instead of $Q$. Let $m$ be a
positive integer and let $n$ be a positive number, $n<m$. Denote
\[
\omega=n/m<1.
\]
Let $C_{\omega\gamma}^{\gamma}(\overline{H})$, $\gamma\in(0,1)$, be
the weighted H\"{o}lder \ space of continuous functions $u(x)$ with the finite norm%

\begin{equation}
|u|_{\omega\gamma,\overline{H}}^{(\gamma)}\equiv\left\Vert u\right\Vert
_{C_{\omega\gamma}^{\gamma}(\overline{H})}\equiv|u|_{\overline{H}}%
^{(0)}+\left\langle u\right\rangle _{\omega\gamma,\overline{H}}^{(\gamma)},
\label{s1.1}%
\end{equation}
where%

\begin{equation}
|u|_{\overline{H}}^{(0)}=\max_{x\in\overline{H}}|u(x)|,\, \left\langle
u\right\rangle _{\omega\gamma,\overline{H}}^{(\gamma)}=\sup_{x,\overline{x}%
\in\overline{H}}\left(  x_{N}^{\ast}\right)  ^{\omega\gamma}\frac
{|u(x)-u(\overline{x})|}{|x-\overline{x}|^{\gamma}},\,x_{N}^{\ast}=\max
\{x_{N},\overline{x}_{N}\}. \label{s1.2}%
\end{equation}
Thus $\left\langle u\right\rangle
_{\omega\gamma,\overline{H}}^{(\gamma)}$ represents a weighted
H\"{o}lder constant of the function $u(x)$.
We suppose that%
\begin{equation}
n<m,\quad,\text{ if }n\text{ is a noninteger}\quad(1-\omega)\gamma
=\gamma\left(  1-\frac{n}{m}\right)  <\min(\{n\},1-\{n\}), \label{s1.2.1}%
\end{equation}
where for a real number $a$ ,$\{a\}$ is the fractional part of $a$,
$[a]$ is the integer part of $a$. This assumption is technical and
it allows us, for example, to consider the functions $x_{N}^{n-j}$
as elements of $C^{\gamma}_{\omega \gamma}(\overline{H})$ for all
integer $j<n$.

%For the further use we split $\left\langle u\right\rangle _{\omega
%\gamma,\overline{H}}^{(\gamma)}$ into two parts and introduce the following
%notation for $\varepsilon>0$%
%\begin{equation}
%^{(\varepsilon+)}\left\langle u\right\rangle _{\omega\gamma,\overline{H}%
%}^{(\gamma)}\equiv\sup_{|x-\overline{x}|\geq\varepsilon\widehat{x}_{N}}\left(
%x_{N}^{\ast}\right)  ^{\omega\gamma}\frac{|u(x)-u(\overline{x})|}%
%{|x-\overline{x}|^{\gamma}},\,x_{N}^{\ast}=\max\{x_{N},\overline{x}%
%_{N}\},\widehat{x}_{N}=\min\{x_{N},\overline{x}_{N}\}, \label{s1.2.01}%
%\end{equation}
%
%\begin{equation}
%^{(\varepsilon-)}\left\langle u\right\rangle _{\omega\gamma,\overline{H}%
%}^{(\gamma)}=\sup_{|x-\overline{x}|<\varepsilon\widehat{x}_{N}}\left(
%x_{N}^{\ast}\right)  ^{\omega\gamma}\frac{|u(x)-u(\overline{x})|}%
%{|x-\overline{x}|^{\gamma}}. \label{s1.2.02}%
%\end{equation}
%Evidently,%
%\[
%\left\langle u\right\rangle _{\omega\gamma,\overline{H}}^{(\gamma)}%
%\leq\left\langle u\right\rangle _{\omega\gamma,\overline{H}}^{(\gamma
%)(\varepsilon+)}+\left\langle u\right\rangle _{\omega\gamma,\overline{H}%
%}^{(\gamma)(\varepsilon-)}.
%\]

%===========================
\begin{remark}
Note that in terms of the Carnot-Carath\'{e}odory metric seminorm
\eqref{s1.2} is equivalent to

\[
\left\langle u\right\rangle
_{\omega\gamma,\overline{H}}^{(\gamma)}\simeq \sup_{x,\overline{x}%
\in\overline{H}}\frac{|u(x)-u(\overline{x})|}{s(x,\overline{x})^{\gamma}},
\]
where the Carnot-Carath\'{e}odory distance is defined as

\[
s(x,\overline{x})=\frac{|x-\overline{x}|}{|x-\overline{x}|^{\omega
}+x_{N}^{\omega }+\overline{x}_{N}^{\omega }}.
\]
In the case of $m=2$, $n\in (0,1)$ this was proved in \cite{D1} and
the general case is quite similar but one should also take into
account Proposition \ref{Ps1.01} below.
\end{remark}
%===========================

In the similar way we define the H\"{o}lder seminorms with respect
to each variable separately%

\begin{equation}
\left\langle u\right\rangle _{\omega\gamma,x_{i},\overline{H}}^{(\gamma)}%
=\sup_{x,\overline{x}\in\overline{H}}\left(  x_{N}^{\ast}\right)
^{\omega\gamma}\frac{|u(x)-u(\overline{x})|}{h^{\gamma}},\,x_{N}^{\ast}%
=\max\{x_{N},\overline{x}_{N}\},i=\overline{1,N}, \label{s1.2.03}%
\end{equation}
where $x=(x_{1},...x_{i},...,x_{N})$, $\overline{x}=(x_{1},...x_{i}%
+h,...,x_{N})$, $h>0$. %And also%

%\begin{equation}
%\left\langle u\right\rangle _{\omega\gamma,x_{i},\overline{H}}^{(\gamma
%)(\varepsilon+)}=\sup_{h\geq\varepsilon\widehat{x}_{N}}\left(  x_{N}^{\ast
%}\right)  ^{\omega\gamma}\frac{|u(x)-u(\overline{x})|}{h^{\gamma}},\widehat
%{x}_{N}=\min\{x_{N},\overline{x}_{N}\}, \label{s1.2.04}%
%\end{equation}
%
%\begin{equation}
%\left\langle u\right\rangle _{\omega\gamma,x_{i},\overline{H}}^{(\gamma
%)(\varepsilon-)}=\sup_{h\leq\varepsilon\widehat{x}_{N}}\left(  x_{N}^{\ast
%}\right)  ^{\omega\gamma}\frac{|u(x)-u(\overline{x})|}{h^{\gamma}}.
%\label{s1.2.05}%
%\end{equation}
In the standard way we denote by $\left\langle u\right\rangle _{x_{i}%
,\overline{H}}^{(\gamma)}$, $\left\langle u\right\rangle _{x^{\prime
},\overline{H}}^{(\gamma)}$, and $\left\langle u\right\rangle _{x,\overline
{H}}^{(\gamma)}$ usual unweighted H\"{o}lder seminorms
with respect to each variable separately, with respect to $x^{\prime}%
=(x_{1},...,x_{N-1})$ or with respect to all $x$-variables.

Define a weighted H\"{o}lder space
$C_{n,\omega\gamma}^{m+\gamma}(\overline
{H})$ as the space of continuous functions $u(x)$ with the finite norm%
\[
|u|_{n,\omega\gamma,\overline{H}}^{(m+\gamma)}\equiv\left\Vert
u\right\Vert
_{C_{n,\omega\gamma}^{m+\gamma}(\overline{H})}=
\]

\begin{equation}
=|u|_{\overline{H}}^{(0)}+
%--
\sum_{0<|\alpha|<m-n}|D^{\alpha}_{x}u|^{\gamma}_{\overline{H}}
%--
+{\displaystyle\sum\limits_{j=0}^{j\leq n}}
{\displaystyle\sum\limits_{\substack{|\alpha|=m-j,\\\alpha_{N}\neq
m-n}}}
|x_{N}^{n-j}D^{\alpha}u|_{\omega\gamma,\overline{H}}^{(\gamma)}.\label{s1.3}%
\end{equation}
Here $\alpha=(\alpha_{1},...,\alpha_{N})$ is a multiindex, $|\alpha
|=\alpha_{1}+...+\alpha_{N}$, $D^{\alpha}u=D_{x_{1}}^{\alpha_{1}}...D_{x_{N}%
}^{\alpha_{N}}u$.  Note that we do not include in the definition of
the norm the term
$|D_{x_{N}}^{m-n}u|_{\omega\gamma,\overline{H}}^{(\gamma)}$ in the
case of an integer $n$. The reason is that this term is finite only
in the case of the special behaviour of
$x_{N}^{n}D_{x_{N}}^{m}u\rightarrow0$ at $x_{N}\rightarrow0$. This
issue will be explained below. For the spaces with the finite term
$|D_{x_{N}}^{m-n}u|_{\omega\gamma,\overline{H}}^{(\gamma)}$ in the
case of an integer $n$ we use the notation with cap. That is the
space $\widehat {C}_{n,\omega\gamma}^{m+\gamma}(\overline{H})$ is
the space with the finite norm%
\[
\widehat{|u|}_{n,\omega\gamma,\overline{H}}^{(m+\gamma)}\equiv\left\Vert
u\right\Vert _{\widehat{C}_{n,\omega\gamma}^{m+\gamma}(\overline{H}%
)}=
\]

\begin{equation}
=|u|_{\overline{H}}^{(0)}+
\sum_{0<|\alpha|<m-n}|D^{\alpha}_{x}u|^{\gamma}_{\overline{H}}+
{\displaystyle\sum\limits_{j=0}^{j\leq n}}
{\displaystyle\sum\limits_{|\alpha|=m-j}}
|x_{N}^{n-j}D^{\alpha}u|_{\omega\gamma,\overline{H}}^{(\gamma)}.\label{s1.3.1}%
\end{equation}
We will show below that the norm \eqref{s1.3} is equivalent to the
norm
\begin{equation}
\widetilde{|u|}_{n,\omega\gamma,\overline{H}}^{(m+\gamma)}=|u|_{\overline{H}%
}^{(0)}+%
{\displaystyle\sum\limits_{i=1}^{N}} \left\langle
x_{N}^{n}D_{x_{i}}^{m}u\right\rangle _{\omega\gamma
,x_{i},\overline{H}}^{(\gamma)}\label{s1.3.2}%
\end{equation}
and the norm \eqref{s1.3.1} in the case of an integer $n$ is
equivalent to the norm
\begin{equation}
\widehat{\widetilde{|u|}}_{n,\omega\gamma,\overline{H}}^{(m+\gamma
)}=|u|_{\overline{H}}^{(0)}+\left\langle D_{x_{N}}^{m-n}u\right\rangle
_{\omega\gamma,x_{N},\overline{H}}^{(\gamma)}+%
{\displaystyle\sum\limits_{i=1}^{N}} \left\langle
x_{N}^{n}D_{x_{i}}^{m}u\right\rangle _{\omega\gamma
,x_{i},\overline{H}}^{(\gamma)}.\label{s1.3.3}%
\end{equation}

We also consider a space $C_{\omega\gamma}^{\gamma,\frac{\gamma}{m}}%
(\overline{Q})$ of functions $u(x,t)$ with the finite norm%

\begin{equation}
|u|_{\omega\gamma,\overline{Q}}^{(\gamma)}\equiv\left\Vert u\right\Vert
_{C_{\omega\gamma}^{\gamma,\gamma/m}(\overline{Q})}\equiv|u|_{\overline{Q}%
}^{(0)}+\left\langle u\right\rangle _{\omega\gamma,\overline{Q}}%
^{(\gamma,\gamma/m)}, \label{s1.4}%
\end{equation}
where%
\[
\left\langle u\right\rangle _{\omega\gamma,\overline{Q}}^{(\gamma,\gamma
/m)}\equiv\left\langle u\right\rangle _{\omega\gamma,x,\overline{Q}}%
^{(\gamma)}+\left\langle u\right\rangle _{t,\overline{Q}}^{(\gamma/m)},
\]

\begin{equation}
\left\langle u\right\rangle _{\omega\gamma,x,\overline{Q}}^{(\gamma)}%
\equiv\sup_{x,\overline{x}\in\overline{Q}}\left(  x_{N}^{\ast}\right)
^{\omega\gamma}\frac{|u(x,t)-u(\overline{x},t)|}{|x-\overline{x}|^{\gamma}%
},\,x_{N}^{\ast}=\max\{x_{N},\overline{x}_{N}\}, \label{s1.5}%
\end{equation}
and $\left\langle u\right\rangle _{t,\overline{Q}}^{(\gamma/m)}$ is
the usual H\"{o}lder constant of $u$ over $\overline{Q}$ with
respect to $t$ with the exponent $\gamma/m$.
%Analogously to
%\eqref{s1.2.01} and \eqref{s1.2.02} we split $\left\langle
%u\right\rangle _{\omega\gamma ,x,\overline{Q}}^{(\gamma)}$ and
%$\left\langle u\right\rangle _{t,\overline {Q}}^{(\gamma/m)}$ into
%two parts for $\varepsilon>0$%
%
%\begin{equation}
%\left\langle u\right\rangle _{\omega\gamma,x,\overline{Q}}^{(\gamma
%)(\varepsilon+)}\equiv\sup_{|x-\overline{x}|\geq\varepsilon\widehat{x}_{N}%
%}\left(  x_{N}^{\ast}\right)  ^{\omega\gamma}\frac{|u(x,t)-u(\overline{x}%
%,t)|}{|x-\overline{x}|^{\gamma}},\,x_{N}^{\ast}=\max\{x_{N},\overline{x}%
%_{N}\},\widehat{x}_{N}=\min\{x_{N},\overline{x}_{N}\}, \label{s1.5.01}%
%\end{equation}
%
%\begin{equation}
%\left\langle u\right\rangle _{\omega\gamma,x,\overline{Q}}^{(\gamma
%)(\varepsilon-)}\equiv\sup_{|x-\overline{x}|\leq\varepsilon\widehat{x}_{N}%
%}\left(  x_{N}^{\ast}\right)  ^{\omega\gamma}\frac{|u(x,t)-u(\overline{x}%
%,t)|}{|x-\overline{x}|^{\gamma}}, \label{s1.5.02}%
%\end{equation}
%
%\begin{equation}
%\left\langle u\right\rangle _{t,\overline{Q}}^{(\gamma/m)(\varepsilon+)}%
%\equiv\sup_{|t-\overline{t}|\geq\varepsilon x_{N}}\frac{|u(x,t)-u(x,\overline
%{t})|}{|t-\overline{t}|^{\gamma/m}}, \label{s1.5.03}%
%\end{equation}
%
%\begin{equation}
%\left\langle u\right\rangle _{t,\overline{Q}}^{(\gamma/m)(\varepsilon-)}%
%\equiv\sup_{|t-\overline{t}|\leq\varepsilon x_{N}}\frac{|u(x,t)-u(x,\overline
%{t})|}{|t-\overline{t}|^{\gamma/m}}. \label{s1.5.04}%
%\end{equation}
Analogously to \eqref{s1.3}, \eqref{s1.3.1}  we consider the space
$C_{n,\omega\gamma}^{m+\gamma
,\frac{m+\gamma}{m}}(\overline{Q})$ with the finite norm%

\[
|u|_{n,\omega\gamma,\overline{Q}}^{(m+\gamma)}\equiv\left\Vert u\right\Vert
_{C_{n,\omega\gamma}^{m+\gamma,\frac{m+\gamma}{m}}(\overline{Q})}=
\]

\begin{equation}
=|u|_{\overline{Q}}^{(0)}+
\sum_{0<|\alpha|<m-n}|D^{\alpha}_{x}u|^{\gamma}_{\overline{Q}}+
{\displaystyle\sum\limits_{j=0}^{j\leq n}}
{\displaystyle\sum\limits_{\substack{|\alpha|=m-j,\\\alpha_{N}\neq
m-n}}}
|x_{N}^{n-j}D_{x}^{\alpha}u|_{\omega\gamma,\overline{Q}}^{(\gamma)}%
+|D_{t}u|_{\omega\gamma,\overline{Q}}^{(\gamma)},\label{s1.6}%
\end{equation}
and the space $\widehat{C}_{n,\omega\gamma}^{m+\gamma,\frac{m+\gamma}{m}%
}(\overline{Q})$ with the finite norm%
\[
\widehat{|u|}_{n,\omega\gamma,\overline{Q}}^{(m+\gamma)}\equiv\left\Vert
u\right\Vert _{\widehat{C}_{n,\omega\gamma}^{m+\gamma,\frac{m+\gamma}{m}%
}(\overline{Q})}=
\]

\begin{equation}
=|u|_{\overline{Q}}^{(0)}+
\sum_{0<|\alpha|<m-n}|D^{\alpha}_{x}u|^{\gamma}_{\overline{Q}}+
{\displaystyle\sum\limits_{j=0}^{j\leq n}}
{\displaystyle\sum\limits_{|\alpha|=m-j}}
|x_{N}^{n-j}D_{x}^{\alpha}u|_{\omega\gamma,\overline{Q}}^{(\gamma)}%
+|D_{t}u|_{\omega\gamma,\overline{Q}}^{(\gamma)}.\label{s1.6.1a}%
\end{equation}
And again we will show that the norm \eqref{s1.6} is equivalent to
the norm

\begin{equation}
\widetilde{|u|}_{n,\omega\gamma,\overline{Q}}^{(m+\gamma)}=|u|_{\overline{Q}%
}^{(0)}+%
{\displaystyle\sum\limits_{i=1}^{N}} \left\langle
x_{N}^{n}D_{x_{i}}^{m}u\right\rangle _{\omega\gamma
,x_{i},\overline{Q}}^{(\gamma)}+\left\langle D_{t}u\right\rangle
_{t,\overline{Q}}^{(\gamma/m)}\label{s1.6.2.n}%
\end{equation}
and the norm
\eqref{s1.6.1a} in the case of an integer $n$ is equivalent to the norm%

\begin{equation}
\widetilde{\widehat{|u|}}_{n,\omega\gamma,\overline{Q}}^{(m+\gamma
)}=|u|_{\overline{Q}}^{(0)}+\left\langle D_{x_{N}}^{m-n}u\right\rangle
_{\omega\gamma,x_{N},\overline{Q}}^{(\gamma)}+%
{\displaystyle\sum\limits_{i=1}^{N}} \left\langle
x_{N}^{n}D_{x_{i}}^{m}u\right\rangle _{\omega\gamma
,x_{i},\overline{Q}}^{(\gamma)}+\left\langle D_{t}u\right\rangle
_{t,\overline{Q}}^{(\gamma/m)}.\label{s1.6.3.0001}%
\end{equation}

%In fact, not all terms in the definitions of the norm in
%\eqref{s1.6}, \eqref{s1.6.1a} are needed.
Namely, we have the
following estimate which is one of the main results of the present
paper.

Recall that $\left\langle D_{t}u\right\rangle
_{t,\overline{Q}}^{(\gamma/m)}$ is the usual H\"{o}lder constant of
$D_{t}u$ over $\overline{Q}$ with respect
only to $t$ with the exponent $\gamma/m$ and $\left\langle x_{N}%
^{n}D_{x_{i}}^{m}u\right\rangle _{\omega\gamma,x_{i},\overline{Q}}^{(\gamma
)}$ is  the weighted H\"{o}lder constants of the "pure" derivatives $x_{N}%
^{n}D_{x_{i}}^{m}u$ with respect only to the corresponding variables $x_{i}$
with the same index $i$, $i=\overline{1,N}$. That is%

\[
\left\langle x_{N}^{n}D_{x_{i}}^{m}u\right\rangle _{\omega\gamma
,x_{i},\overline{Q}}^{(\gamma)}\equiv\sup_{(x,t),(\overline{x},t)\in
\overline{Q}}\left(  x_{N}^{\ast}\right)  ^{\omega\gamma}\frac
{|u(x,t)-u(\overline{x},t)|}{|x_{i}-\overline{x}_{i}|^{\gamma}},\,x_{N}^{\ast
}=\max\{x_{N},\overline{x}_{N}\},
\]
where $\sup$ is taken over $x=(x_{1},...,x_{i},...x_{N})$,
$\overline {x}=(x_{1},...,\overline{x}_{i},...x_{N})$.

\begin{theorem}\label{Ts1.1}
%Let $u(x,t)\in\widehat{C}_{n,\omega\gamma}^{m+\gamma
%,\frac{m+\gamma}{m}}(\overline{Q})$ or $u(x,t)\in C_{n,\omega\gamma}%
%^{m+\gamma,\frac{m+\gamma}{m}}(\overline{Q})$.
Let $u(x,t)$ be continuous in $\overline{Q}$ and the right hand side in 
\eqref{s1.7} below is finite.  
Then for some $C=C(N,\gamma
,m,n)$%

\[
\left\langle u\right\rangle _{n,\omega\gamma,\overline{Q}}^{(m+\gamma
,\frac{m+\gamma}{m})}\equiv%
{\displaystyle\sum\limits_{j=0}^{j\leq n}}
{\displaystyle\sum\limits_{|\alpha|=m-j}} \left\langle
x_{N}^{n-j}D_{x}^{\alpha}u\right\rangle _{\omega\gamma
,\overline{Q}}^{(\gamma,\gamma/m)}+%
{\displaystyle\sum\limits_{j=0}^{j\leq n}}
{\displaystyle\sum\limits_{|\alpha|=m-j}}
\left\langle x_{N}^{n-j\omega}D_{x}^{\alpha}u\right\rangle _{t,\overline{Q}%
}^{(\frac{\gamma+j}{m})}+
\]

\[
+\left\langle D_{t}u\right\rangle _{\omega
\gamma,\overline{Q}}^{(\gamma,\gamma/m)}+%
{\displaystyle\sum\limits_{j=0}^{j\leq m-n}}
{\displaystyle\sum\limits_{|\alpha|=[m-n+(1-\omega)\gamma]-j}}
\left\langle D_{x^{\prime}}^{\alpha}D_{x_{N}}^{j}u\right\rangle
_{x^{\prime },\overline{Q}}^{(\{m-n+(1-\omega)\gamma\})}+
\]

\[
+{\displaystyle\sum\limits_{j=0}^{j\leq m-n}}
{\displaystyle\sum\limits_{|\alpha|=[m-n+\gamma]-j}} \left\langle
D_{x^{\prime}}^{\alpha}D_{x_{N}}^{j}u\right\rangle _{\omega
\gamma,x^{\prime},\overline{Q}}^{(\{m-n+\gamma\})}+
\]

\begin{equation}
+%
{\displaystyle\sum\limits_{j=1}^{j\leq m-n}}
{\displaystyle\sum\limits_{|\alpha|=j}}
\left\langle D_{x}^{\alpha}u\right\rangle _{t,\overline{Q}}^{(1-\frac{j}%
{m-n}+\frac{\gamma}{m})}\leq C\left(
{\displaystyle\sum\limits_{i=1}^{N}} \left\langle
x_{N}^{n}D_{x_{i}}^{m}u\right\rangle _{\omega\gamma
,x_{i},\overline{Q}}^{(\gamma)}+\left\langle D_{t}u\right\rangle
_{t,\overline{Q}}^{(\gamma/m)}\right)  , \label{s1.7}%
\end{equation}
where, $[a]$ and $\{a\}$ are the integer and the fractional parts of
a real number $a$ correspondingly and in the left hand side of
\eqref{s1.7} included only those terms that are finite.

Moreover,
\begin{equation}
x_{N}^{n-j}D_{x}^{\alpha}u(x,t)\rightarrow0,x_{N}\rightarrow0,\quad0\leq
j<n,\alpha=(\alpha_{1},...,\alpha_{N}),|\alpha|=m-j,\alpha_{N}<m-j.
\label{s1.7.1}%
\end{equation}

If $u(x)$ is continuous in $\overline{H}$ and the right hand side in
\eqref{s1.8} below is finite 
then for some $C=C(N,\gamma,m,n)$%

\[
\left\langle u\right\rangle _{n,\omega\gamma,\overline{H}}^{(m+\gamma)}\equiv%
{\displaystyle\sum\limits_{j=0}^{j\leq n}}
{\displaystyle\sum\limits_{|\alpha|=m-j}} \left\langle
x_{N}^{n-j}D_{x}^{\alpha}u\right\rangle _{\omega\gamma
,\overline{H}}^{(\gamma)}+%
{\displaystyle\sum\limits_{j=0}^{j\leq m-n}}
{\displaystyle\sum\limits_{|\alpha|=[m-n+\gamma]-j}} \left\langle
D_{x^{\prime}}^{\alpha}D_{x_{N}}^{j}u\right\rangle _{\omega
\gamma,x^{\prime},\overline{Q}}^{(\{m-n+\gamma\})}+
\]

\begin{equation}
+%
{\displaystyle\sum\limits_{j=0}^{j\leq m-n}}
{\displaystyle\sum\limits_{|\alpha|=[m-n+(1-\omega)\gamma]-j}}
\left\langle D_{x^{\prime}}^{\alpha}D_{x_{N}}^{j}u\right\rangle
_{x^{\prime
},\overline{H}}^{(\{m-n+(1-\omega)\gamma\})}\leq C%
{\displaystyle\sum\limits_{i=1}^{N}} \left\langle
x_{N}^{n}D_{x_{i}}^{m}u\right\rangle _{\omega\gamma
,x_{i},\overline{H}}^{(\gamma)}.\label{s1.8}%
\end{equation}
and in the left hand side of \eqref{s1.8} included only those terms
that are finite.
Moreover,%

\begin{equation}
x_{N}^{n-j}D_{x}^{\alpha}u(x)\rightarrow0,x_{N}\rightarrow0,\quad0\leq
j<n,\alpha=(\alpha_{1},...,\alpha_{N}),|\alpha|=m-j,\alpha_{N}<m-j.
\label{s1.8.1}%
\end{equation}

\end{theorem}

Note that Theorem \ref{Ts1.1} is an analog for weighted H\"{o}lder
spaces of well known properties of standard H\"{o}lder spaces. We
are going to use these known properties so we formulate them in the
next section.

Let us stress that the assumption that the terms in the left hand side of 
\eqref{s1.7}, \eqref{s1.8} are finite is essential. Consider in $\{(x_{1},x_{2}):x_{2}\geq 0\}$ for $m=2$ the 
function $u(x)=x_{1}^{2}x_{2}^{2-n}$, where $n\in [0,1)$. For this function the right hand side 
of \eqref{s1.8} is zero but the H\"{o}lder seminorms of the mixed derivative $x_{2}^{n}D^{2}_{x_{1}x_{2}}u$ 
in the left hand side are infinite.

The further content of the paper is as follows. In section
\ref{ss2}, we formulate some known results about classical
H\"{o}lder spaces and prove some useful statements about weighted
H\"{o}lder spaces for further using. Section \ref{ss1.3} is devoted
to the proof of Theorem \ref{Ts1.1}. In section \ref{ss1.4}, we
consider properties of mixed and lower order derivatives of
functions from the space
$C_{n,\omega\gamma}^{m+\gamma,\frac{m+\gamma}{m}}(\overline{Q})$. In
section \ref{ss1.5} we study traces of functions from
$C_{n,\omega\gamma}^{m+\gamma,\frac {m+\gamma}{m}}(\overline{Q})$ at
$\{x_{N}=0\}$. Section \ref{ss1.6} contains some interpolations
inequalities for functions from
$C_{n,\omega\gamma}^{m+\gamma,\frac{m+\gamma}{m}}(\overline{Q})$,
$C_{n,\omega\gamma}^{m+\gamma}(\overline{H})$. In section
\ref{ss1.8} we consider the spaces $C_{n,\omega\gamma}^{m+\gamma,\frac{m+\gamma}{m}%
}(\overline{\Omega}_{T})$,
$C_{n,\omega\gamma}^{m+\gamma}(\overline{\Omega})$ in the case of
arbitrary smooth domain. At last, section \ref{ss1.9} devoted to
some properties of functions from $C_{n,\omega\gamma,0}^{m+\gamma,\frac{m+\gamma}{m}%
}(\overline{\Omega}_{T})$, where the last is the closed subspace of $C_{n,\omega\gamma}%
^{m+\gamma,\frac{m+\gamma}{m}}(\overline{\Omega}_{T})$ consisting of
functions $u(x,t)$ with the property $u(x,0)\equiv
u_{t}(x,0)\equiv0$ in $\overline{\Omega}$.

\section{Auxiliary assertions.}
\label{ss2}

Let $M$ be a positive integer. In the space $R^{M}$ we use standard
\bigskip H\"{o}lder spaces $C^{\overline{l}}(R^{M})$, where
$\overline{l}=(l_{1},l_{2},...,l_{M})$, $l_{i}$ are arbitrary
positive non-integers. The norm in such spaces is defined by%

\begin{equation}
\left\Vert u\right\Vert _{C^{\overline{l}}(R^{M})}\equiv\left\vert
u\right\vert _{R^{M}}^{(\overline{l})}=\left\vert u\right\vert _{R^{M}}%
^{(0)}+\sum_{i=1}^{M}\left\langle u\right\rangle _{x_{i},R^{M}}^{(l_{i})},
\label{s1.9}%
\end{equation}

\begin{equation}
\left\langle u\right\rangle _{x_{i},R^{M}}^{(l_{i})}=\sup_{x\in R^{M}%
,h>0}\frac{\left\vert D_{x_{i}}^{[l_{i}]}u(x_{1},...,x_{i}+h,...,x_{M}%
)-D_{x_{i}}^{[l_{i}]}u(x)\right\vert }{h^{l_{i}-[l_{i}]}}, \label{s1.10}%
\end{equation}
where $[l_{i}]$ is the integer part of the number $l_{i}$, $D_{x_{i}}%
^{[l_{i}]}u$ is the derivative of order $[l_{i}]$ with respect to
the variable $x_{i}$ of a function $u$.

\begin{proposition}
\label{Ps1.1}

Seminorm \eqref{s1.10} can be equivalently defined by
(\cite{Triebel},\cite{Sol15},
\cite{Gol18} )%

\begin{equation}
\left\langle u\right\rangle _{x_{i},R^{M}}^{(l_{i})}\simeq\sup_{x\in
R^{M},h>0}\frac{\left\vert \Delta_{h,x_{i}}^{k}u(x)\right\vert }{h^{l_{i}}%
},\quad k>l_{i}, \label{s1.11}%
\end{equation}
where $\Delta_{h,x_{i}}u(x)=u(x_{1},...,x_{i}+h,...,x_{N})-u(x)$ is
the difference from a function $u(x)$ with respect to the variable
$x_{i}$  with a step $h$, $\Delta_{h,x_{i}}^{k}u(x)=$
$\Delta_{h,x_{i}}\left( \Delta_{h,x_{i}}^{k-1}u(x)\right)  =\left(
\Delta_{h,x_{i}}\right)  ^{k}u(x)$ is the difference of power $k$.
\end{proposition}

The same is also valid not only for the whole space $R^{M}$ but also
for it's subsets of the form
$R^{M}\cap\{x_{i_{1}},x_{i_{2}},...,x_{i_{K}}\geq0\}$ with $K\leq
M$. Note that below we prove an analogous statement for weighted
spaces.

It is known that functions from the space $C^{\overline{l}}(R^{M})$ have also
mixed derivatives up to definite orders and all derivatives are H\"{o}lder
continuous with respect to all variables with some exponents in accordance
with ratios between the exponents $l_{i}$. Namely, if $\overline{k}%
=(k_{1},...,k_{M})$ with nonnegative integers $k_{i}$, $k_{i}\leq\lbrack
l_{i}]$, and
\begin{equation}
\omega=1-\sum_{i=1}^{N}\frac{k_{i}}{l_{i}}>0, \label{s1.11.1}%
\end{equation}
then (see for example \cite{Sol15} )
\begin{equation}
D_{x}^{\overline{k}}u(x)\in C^{\overline{d}}(R^{M}),\ \ \ \ \Vert
D_{x}^{\overline{k}}u\Vert_{C^{\overline{d}}(R^{M})}\leq C\Vert u\Vert
_{C^{\overline{l}}(R^{M})}, \label{s1.12}%
\end{equation}
where
\begin{equation}
\overline{d}=(d_{1},...,d_{M}),\, \,d_{i}=\omega l_{i}. \label{s1.12.1}%
\end{equation}
Moreover, relation \eqref{s1.12} is valid not only for $R^{M}$ but
for any domain $\Omega\subset R^{M}$ with sufficiently smooth
boundary and we have

\begin{equation}
\Vert D_{x}^{\overline{k}}u\Vert_{C^{\overline{d}}(\overline{\Omega})}\leq
C\Vert u\Vert_{C^{\overline{l}}(\overline{\Omega})}. \label{s1.12+1}%
\end{equation}
For special domains of the form $\Omega_{+}=$ $R^{M}\cap\{x_{i_{1}},x_{i_{2}%
},...,x_{i_{K}}\geq0\}$ we have even more strong inequality just for seminorms%

\begin{equation}%
{\displaystyle\sum\limits_{\overline{k}}}
{\displaystyle\sum\limits_{i=1}^{M}}
\left\langle D_{x}^{\overline{k}}u\right\rangle _{x_{i},\overline{\Omega}_{+}%
}^{(d_{i})}\leq C\sum_{i=1}^{M}\left\langle u\right\rangle _{x_{i}%
,\overline{\Omega}_{+}}^{(l_{i})}. \label{s1.12+2}%
\end{equation}
Here the sum is taken over all $\overline{k}$ with the property
\eqref{s1.11.1} and $d_{i}$ are defined in \eqref{s1.12.1}.

The analog of this estimate for an arbitrary smooth domain $\Omega$
(including bounded domains) is

\begin{equation}%
{\displaystyle\sum\limits_{\overline{k}}}
{\displaystyle\sum\limits_{i=1}^{M}}
\left\langle D_{x}^{\overline{k}}u\right\rangle _{x_{i},\overline{\Omega}%
}^{(\widehat{d}_{i})}\leq C\left(  \sum_{i=1}^{M}\left\langle u\right\rangle
_{x_{i},\overline{\Omega}}^{(l_{i})}+|u|_{\overline{\Omega}}^{(0)}\right)
\label{s1.12+2.1}%
\end{equation}
with arbitrary $\widehat{d}_{i}\leq d_{i}$. Note that inequalities
\eqref{s1.7} and \eqref{s1.8} are in fact a particular cases of
\eqref{s1.12+2} for weighted spaces.

It turns out that the weighted space $C_{\omega\gamma}^{\gamma}(\overline{H})$
is embedded into the usual space $C^{\gamma-\omega\gamma}(\overline{H})$.
Namely, we have the following assertion.

\begin{proposition}
\label{Ps1.01}
Let a function $u(x)\in
C_{\omega\gamma}^{\gamma}(\overline{H})$. Then $u(x)$
is continuous in $\overline{H}$ and%

\begin{equation}
\left\langle u\right\rangle _{x,\overline{H}}^{(\gamma-\omega\gamma)}\leq
C\left\langle u\right\rangle _{\omega\gamma,x,\overline{H}}^{(\gamma)}.
\label{s1.01}%
\end{equation}
\end{proposition}

\begin{proof}

We consider the H\"{o}lder property with the exponent
$\gamma-\omega\gamma$ of the function $u(x)$ with respect to the
variable $x_{N}$ and with respect to the variables $x^{\prime}$
separately.

Consider the ratio with $h>0$%

\[
A_{h}\equiv\frac{|u(x^{\prime},x_{N}+h)-u(x^{\prime},x_{N})|}{h^{\gamma
-\omega\gamma}}=h^{\omega\gamma}\frac{|u(x^{\prime},x_{N}+h)-u(x^{\prime
},x_{N})|}{h^{\gamma}}\leq
\]

\[
\leq(x_{N}+h)^{\omega\gamma}\frac{|u(x^{\prime},x_{N}+h)-u(x^{\prime},x_{N}%
)|}{h^{\gamma}}\leq\left\langle u\right\rangle _{\omega\gamma,x,\overline{H}%
}^{(\gamma)}.
\]
Thus it is proved that at least on open set $H$%

\begin{equation}
\left\langle u\right\rangle _{x_{N},H}^{(\gamma-\omega\gamma)}\leq\left\langle
u\right\rangle _{\omega\gamma,x_{N},\overline{H}}^{(\gamma)}. \label{s1.02}%
\end{equation}
Let now $\overline{h}=(h_{1},...,h_{N-1})$. Consider the expression%

\[
A_{h}\equiv\frac{|u(x^{\prime}+\overline{h},x_{N})-u(x^{\prime},x_{N}%
)|}{|\overline{h}|^{\gamma-\omega\gamma}}.
\]
If $|\overline{h}|\leq x_{N}/2$ we can write%

\[
A_{h}=|\overline{h}|^{\omega\gamma}\frac{|u(x^{\prime}+\overline{h}%
,x_{N})-u(x^{\prime},x_{N})|}{|\overline{h}|^{\gamma}}\leq
\]

\begin{equation}
\leq Cx_{N}^{\omega\gamma}\frac{|u(x^{\prime}+\overline{h},x_{N})-u(x^{\prime
},x_{N})|}{|\overline{h}|^{\gamma}}\leq C\left\langle u\right\rangle
_{\omega\gamma,x^{\prime},\overline{H}}^{(\gamma)}. \label{s1.03}%
\end{equation}
If now $|\overline{h}|>x_{N}/2$, then we estimate $A_{h}$ as%

\[
A_{h}\leq\frac{|u(x^{\prime}+\overline{h},x_{N})-u(x^{\prime}+\overline
{h},x_{N}+2|\overline{h}|)|}{|\overline{h}|^{\gamma}}+
\]

\[
+\frac{|u(x^{\prime}+\overline{h},x_{N}+2|\overline{h}|)-u(x^{\prime}%
,x_{N}+2|\overline{h}|)|}{|\overline{h}|^{\gamma-\omega\gamma}}+\frac
{|u(x^{\prime},x_{N}+2|\overline{h}|)-u(x^{\prime},x_{N})|}{|\overline
{h}|^{\gamma-\omega\gamma}}\equiv%
{\displaystyle\sum\limits_{i=1}^{3}} I_{i}.
\]
The estimates for $I_{1}$ and $I_{3}$ follow from \eqref{s1.02} and
the estimate for $I_{2}$ follows from \eqref{s1.03} because in this
case $|\overline{h}|\leq(x_{N}+2|\overline{h}|)/2$. Thus in
this case%

\[
A_{h}\leq C(\left\langle u\right\rangle _{\omega\gamma,x_{N},\overline{H}%
}^{(\gamma)}+\left\langle u\right\rangle _{\omega\gamma,x^{\prime}%
,\overline{H}}^{(\gamma)})\leq C\left\langle u\right\rangle _{\omega
\gamma,x,\overline{H}}^{(\gamma)}.
\]
Consequently, it is proved that on open set $H$%

\begin{equation}
\left\langle u\right\rangle _{x^{\prime},H}^{(\gamma-\omega\gamma)}\leq
C\left\langle u\right\rangle _{\omega\gamma,x,\overline{H}}^{(\gamma)}.
\label{s1.04}%
\end{equation}
From \eqref{s1.02} and
\eqref{s1.04} it follows that%

\[
\left\langle u\right\rangle _{x,H}^{(\gamma-\omega\gamma)}\leq C\left\langle
u\right\rangle _{\omega\gamma,x,\overline{H}}^{(\gamma)}.
\]
This means that $u(x)$ has a finite limit as $x_{N}\rightarrow0$ and
consequently can be defined at $x_{N}=0$ as a continuous function
with \eqref{s1.01}. Thus the proposition follows.
\end{proof}

We need also the analog of relation \eqref{s1.1} for weighted
seminorm.

\begin{proposition}
\label{Ps1.1}
Let $l=m+\gamma>0$ be noninteger, $m=[l]$, ,
$\gamma\in(0,1)$, and let a function $u(y)\in
C_{\omega\gamma}^{l}([0,\infty))$, $\omega\in(0,1)$, in the
sense that%

\begin{equation}
\left\langle D_{y}^{m}u\right\rangle _{\omega\gamma,y}^{(\gamma)}=\sup
_{y,h>0}(y+h)^{\omega\gamma}\frac{|D_{y}^{m}u(y+h)-D_{y}^{m}u(y)|}{h^{\gamma}%
}<\infty. \label{s1.15}%
\end{equation}

Then for any integer $k>l$%

\begin{equation}
\left\langle D_{y}^{m}u\right\rangle _{\omega\gamma,y}^{(\gamma)}\leq
C_{k}\sup_{y,h>0}y^{\omega\gamma}\frac{|\Delta_{h}^{k}u(y)|}{h^{l}}\equiv
C_{k}\left\langle \left\langle u\right\rangle \right\rangle _{\omega\gamma
,y}^{(l)(k)}, \label{s1.16}%
\end{equation}
where $\Delta_{h}^{k}u(y)$ is the $k$-th difference with the step
$h$, $\Delta_{h}^{1}u(y)=\Delta_{h}u(y)=u(y+h)-u(y)$, $\Delta_{h}%
^{k}u(y)=\Delta_{h}(\Delta_{h}^{k-1}u(y))$. Note that the inverse
inequality to \eqref{s1.16} is evident because of the mean value
theorem.
\end{proposition}

\begin{proof}
The idea of the proof is taken from \cite{LS} and demonstrates also
the main idea of the proof of Theorem \ref{Ts1.1}. Let
$\varepsilon\in(0,1)$ be fixed and will be chosen later. To prove
\eqref{s1.16} we represent $\left\langle D_{y}^{m}u\right\rangle
_{\omega\gamma,y}^{(\gamma)}$ as%

\[
\left\langle D_{y}^{m}u\right\rangle _{\omega\gamma,y}^{(\gamma)}\leq
\sup_{y,h\geq\varepsilon y}(y+h)^{\omega\gamma}\frac{|D_{y}^{m}u(y+h)-D_{y}%
^{m}u(y)|}{h^{\gamma}}+
\]

\begin{equation}
+\sup_{y,0<h<\varepsilon y}(y+h)^{\omega\gamma}\frac{|D_{y}^{m}u(y+h)-D_{y}%
^{m}u(y)|}{h^{\gamma}}\equiv\left\langle D_{y}^{m}u\right\rangle
_{\omega\gamma,y}^{(\gamma)(\varepsilon+)}+\left\langle D_{y}^{m}%
u\right\rangle _{\omega\gamma,y}^{(\gamma)(\varepsilon-)}. \label{s1.16.1}%
\end{equation}

We are going to consider the two cases for the relation between $\left\langle
D_{y}^{m}u\right\rangle _{\omega\gamma,y}^{(\gamma)(\varepsilon+)}$ and
$\left\langle D_{y}^{m}u\right\rangle _{\omega\gamma,y}^{(\gamma
)(\varepsilon-)}$.

Suppose first that%

\begin{equation}
\left\langle D_{y}^{m}u\right\rangle _{\omega\gamma,y}^{(\gamma)(\varepsilon
-)}\leq\left\langle D_{y}^{m}u\right\rangle _{\omega\gamma,y}^{(\gamma
)(\varepsilon+)}, \label{s1.16.3}%
\end{equation}
and consequently

\begin{equation}
\left\langle D_{y}^{m}u\right\rangle _{\omega\gamma,y}^{(\gamma)(\varepsilon
+)}\leq\left\langle D_{y}^{m}u\right\rangle _{\omega\gamma,y}^{(\gamma)}%
\leq2\left\langle D_{y}^{m}u\right\rangle _{\omega\gamma,y}^{(\gamma
)(\varepsilon+)}. \label{s1.16.4}%
\end{equation}
We prove that in this case%

\begin{equation}
\left\langle D_{y}^{m}u\right\rangle _{\omega\gamma,y}^{(\gamma)(\varepsilon
+)}\leq C_{\varepsilon,k}\sup_{y,h>0}y^{\omega\gamma}\frac{|\Delta_{h}%
^{k}u(y)|}{h^{l}}. \label{s1.16.2}%
\end{equation}
The proof is by contradiction. Suppose that \eqref{s1.16.2} is not
valid. Then for any positive integer $p$ there exists a function
$u_{p}(y)\in C_{\omega\gamma}^{l}([0,\infty))$ with

\begin{equation}
\sup_{y,h\geq\varepsilon y}(y+h)^{\omega\gamma}\frac{|D_{y}^{m}u_{p}%
(y+h)-D_{y}^{m}u_{p}(y)|}{h^{\gamma}}\geq p\sup_{y,h>0}y^{\omega\gamma}%
\frac{|\Delta_{h}^{k}u_{p}(y)|}{h^{l}}. \label{s1.17}%
\end{equation}
Consider the functions
% $w_{p}(y)=u_{p}(y)/\left\langle D_{y}^{m}%
%u_{p}\right\rangle _{\omega\gamma,y}^{(\gamma)}$%

\begin{equation}
w_{p}(y)=\frac{u_{p}(y)}{\left\langle D_{y}^{m}u_{p}\right\rangle
_{\omega\gamma,y}^{(\gamma)}}. \label{s1.18}%
\end{equation}
For such functions we have by the definition and by \eqref{s1.17},
\eqref{s1.16.4}

\begin{equation}
\left\langle D_{y}^{m}w_{p}\right\rangle _{\omega\gamma,y}^{(\gamma)}%
=1,\quad\sup_{y,h\geq\varepsilon y}(y+h)^{\omega\gamma}\frac{|D_{y}^{m}%
w_{p}(y+h)-D_{y}^{m}w_{p}(y)|}{h^{\gamma}}\geq\frac{1}{2}, \label{s1.19}%
\end{equation}

\begin{equation}
\sup_{y,h>0}y^{\omega\gamma}\frac{|\Delta_{h}^{k}w_{p}(y)|}{h^{l}}\leq\frac
{1}{p}. \label{s1.20}%
\end{equation}
It follows from the second relation in \eqref{s1.19} that there
exist sequences
$\{y_{p}\} \subset\lbrack0,\infty)$ and $\{h_{p}\} \subset(0,\infty)$ with%

\begin{equation}
(y_{p}+h_{p})^{\omega\gamma}\frac{|D_{y}^{m}w_{p}(y_{p}+h_{p})-D_{y}^{m}%
w_{p}(y_{p})|}{h_{p}^{\gamma}}\geq\frac{1}{4}. \label{s1.21}%
\end{equation}
Now we apply the scaling arguments. Define the sequence of scaled
functions $\{v_{p}$\bigskip$(z)\}$, $z\in\lbrack0,\infty)$,%

\begin{equation}
v_{p}\bigskip(z)\equiv h_{p}^{-m-(1-\omega)\gamma}w_{p}(zh_{p}). \label{s1.22}%
\end{equation}
It follows from this definition and from \eqref{s1.19}-
\eqref{s1.21} that

\begin{equation}
\left\langle D_{z}^{m}v_{p}\right\rangle _{\omega\gamma,z}^{(\gamma)}%
=1,\quad\sup_{z,h>0}z^{\omega\gamma}\frac{|\Delta_{h}^{k}v_{p}(z)|}{h^{l}}%
\leq\frac{1}{p}, \label{s1.23}%
\end{equation}

\begin{equation}
(z_{n}+1)^{\omega\gamma}|D_{z}^{m}v_{p}(z_{p}+1)-D_{z}^{m}v_{p}(z_{p}%
)|\geq\frac{1}{4}, \label{s1.24}%
\end{equation}
where $z_{p}=y_{p}/h_{p}$. \ Let now $P_{m}^{(p)}(z)$ be the Taylor
polynomial of the degree $m$ for the function $v_{p}(z)$ at the
point, for example, $z=1$. Since $D_{z}^{m}P_{m}^{(p)}(z)=const$ and
$k>m$ in \eqref{s1.23}, we have for the functions
$r_{p}(z)=v_{p}(z)-P_{m}^{(p)}(z)$%

\begin{equation}
\left\langle D_{z}^{m}r_{p}\right\rangle _{\omega\gamma,z}^{(\gamma)}%
=1,\quad\sup_{z,h>0}z^{\omega\gamma}\frac{|\Delta_{h}^{k}r_{p}(z)|}{h^{l}}%
\leq\frac{1}{p}, \label{s1.25}%
\end{equation}

\begin{equation}
(z_{p}+1)^{\omega\gamma}|D_{z}^{m}r_{p}(z_{p}+1)-D_{z}^{m}r_{p}(z_{p}%
)|\geq\frac{1}{4}. \label{s1.26}%
\end{equation}

From Proposition \ref{Ps1.01}, the first relation in \eqref{s1.23},
and from the fact that $D^{i}r_{p}(1)=0$, $i=\overline{0,m}$
it follows that%

\begin{equation}
\left\Vert r_{p}\right\Vert _{C^{m+(1-\omega)\gamma}(K)}\leq
C(K)=CR^{m},
\label{s1.27}%
\end{equation}
where $K$ is a compact set in $[0,\infty)$, $K\subseteq\lbrack0,R]$,
$R>0$. From this and the Arzela theorem we conclude that (at least
for a subsequence) $D^{i}r_{p}(z)$, $i=\overline{0,m}$,
uniformly converge on compact sets $K$ to some function $r(z)$ and it's derivatives%

\begin{equation}
D^{i}r_{p}(z)\rightrightarrows_{K}D^{i}r(z),\quad i=\overline{0,m}.
\label{s1.28}%
\end{equation}
This, together with the first relation in
\eqref{s1.25}, in particular, gives%

\begin{equation}
\left\langle D_{z}^{m}r\right\rangle _{\omega\gamma,z}^{(\gamma)}+\left\langle
D_{z}^{m}r\right\rangle _{z}^{((1-\omega)\gamma)}\leq1. \label{s1.29}%
\end{equation}
 Let now $z,h>0$ be fixed. From
\eqref{s1.25} it follows that%

\[
z^{\omega\gamma}|\Delta_{h}^{k}r_{p}(z)|\leq\frac{1}{p}h^{l}%
\]
and letting $p \rightarrow\infty$ we obtain $\Delta_{h}^{k}r(z)=0$.
As $z$ and $h$ are arbitrary we conclude that%

\[
\Delta_{h}^{k}r(z)\equiv0,\quad z,h>0,
\]
and consequently $r(z)$ is a polynomial of degree not greater than
$k-1$. Moreover $D^{m}r(z)$ is not a constant because of
\eqref{s1.26}. Indeed, consider the sequence $\{z_{p}\}$. Since we
are considering $A_{1}(\varepsilon)$ with the condition
$h\geq\varepsilon y$, we have $0\leq
z_{p}=y_{p}/h_{p}\leq1/\varepsilon$. Therefore for a subsequence
$z_{p}\rightarrow z_{0}$, $n\rightarrow\infty$. Then it follows from
\eqref{s1.26} and \eqref{s1.28} that%

\[
(z_{0}+1)^{\omega\gamma}|D_{z}^{m}r(z_{0}+1)-D_{z}^{m}r(z_{0})|\geq\frac{1}{4}%
\]
that is $D^{m}r(z)$ is not a constant polynomial. But this fact
contradicts to \eqref{s1.29} since a non constant polynomial can not
have finite seminorms as those in \eqref{s1.29}. This contradiction
shows that \eqref{s1.16.2} is valid with some constant
$C_{\varepsilon,k}$ and in this case we have also \eqref{s1.16} with
such $C_{\varepsilon,k}$ by virtue of \eqref{s1.16.4}.

Suppose now that $\left\langle D_{y}^{m}u\right\rangle
_{\omega\gamma ,y}^{(\gamma)(\varepsilon+)}\leq\left\langle
D_{y}^{m}u\right\rangle _{\omega\gamma,y}^{(\gamma)(\varepsilon-)}$.
In this case we have instead of \eqref{s1.16.4}%

\begin{equation}
\left\langle D_{y}^{m}u\right\rangle _{\omega\gamma,y}^{(\gamma)(\varepsilon
-)}\leq\left\langle D_{y}^{m}u\right\rangle _{\omega\gamma,y}^{(\gamma)}%
\leq2\left\langle D_{y}^{m}u\right\rangle _{\omega\gamma,y}^{(\gamma
)(\varepsilon-)}. \label{s1.30}%
\end{equation}
We prove in this case the estimate%

\begin{equation}
\left\langle D_{y}^{m}u\right\rangle _{\omega\gamma,y}^{(\gamma)(\varepsilon
-)}\leq C_{\varepsilon,k}\sup_{y,h>0}y^{\omega\gamma}\frac{|\Delta_{h}%
^{k}u(y)|}{h^{l}}+C_{k}\varepsilon^{1-\gamma}\left\langle D_{y}^{m}%
u\right\rangle _{\omega\gamma,y}^{(\gamma)}, \label{s1.30.1}%
\end{equation}
where $C_{k}$ does not depend on $\varepsilon\in(0,1/(8k))$. We
apply some local considerations around arbitrary point in
$[0,\infty)$ . Let $y_{0}>0$ and $0<h<\varepsilon y_{0}$ be fixed.
Let $B=[y_{0}/4,7y_{0}/4]$ be a ball with center in $y_{0}$ and of
radius $3y_{0}/4.$ Denote by $\eta(y)\in
C^{\infty}([0,\infty))$ a smooth function with the properties%

\begin{equation}
\eta(y)\equiv1,|y-y_{0}|\leq\frac{1}{4}y_{0},\quad\eta(y)\equiv0,|y-y_{0}%
|\geq\frac{1}{2}y_{0},\quad|D_{y}^{s}\eta(y)|\leq C_{s}y_{0}^{-s}. \label{s1.31}%
\end{equation}
Without loss of generality we can assume that%

\begin{equation}
D_{y}^{i}u(y_{0})=0,\quad i=0,1,...,m. \label{s1.32}%
\end{equation}
If it is not the case we can consider $\overline{u}(y)=u(y)-P_{y_{0}}%
^{(m)}(y)$ instead of $u(y)$, where $P_{y_{0}}^{(m)}(y)$ is the
Taylor polynomial of $u(y)$ of power $m$ at the point $y_{0}$. It is
possible because
$\Delta_{h}D_{y}^{m}u(y)\equiv\Delta_{h}D_{y}^{m}\overline{u}(y)$
and $\Delta_{h}^{k}u(y)\equiv\Delta_{h}^{k}\overline{u}(y)$. Denote
also $v(y)=u(y)\eta(y)$. Keeping in mind the definition of
$\left\langle D_{y}^{m}u\right\rangle _{\omega\gamma
,y}^{(\gamma)(\varepsilon-)}$,  we have by virtue of the
properties of $\eta$ in \eqref{s1.31} and $h<\varepsilon y_{0}<y_{0}/4$%

\[
A_{2}^{(y_{0},h)}(\varepsilon)\equiv(y_{0}+h)^{\omega\gamma}\frac{|D_{y}%
^{m}u(y_{0}+h)-D_{y}^{m}u(y_{0})|}{h^{\gamma}}=
\]

\begin{equation}
=(y_{0}+h)^{\omega\gamma}\frac{|D_{y}^{m}v(y_{0}+h)-D_{y}^{m}v(y_{0}%
)|}{h^{\gamma}}\equiv(y_{0}+h)^{\omega\gamma}\cdot A. \label{s1.33}%
\end{equation}
Note that the truncated function $v(y)=u(y)\eta(y)\in C^{m+\gamma}%
([0,\infty))$ , that is to the usual space without a weight. Thus by
\eqref{s1.11} we have ($l=m+\gamma$)%

\begin{equation}
A\leq C\sup_{y,h>0}\frac{|\Delta_{h}^{k}v(y)|}{h^{l}}. \label{s1.34}%
\end{equation}
The ratio in the right hand side of this inequality has the form%

\[
\frac{\Delta_{h}^{k}v(y)}{h^{l}}=\frac{\Delta_{h}^{k}\left(  u(y)\eta
(y)\right)  }{h^{l}}=%
{\displaystyle\sum\limits_{i=0}^{k}}
C_{i}\frac{\Delta_{h}^{i}u(y_{i}^{(u)})\Delta_{h}^{k-i}\eta(y_{i}^{(\eta)}%
)}{h^{l}}=
\]

\begin{equation}
=\frac{\Delta_{h}^{k}u(y)}{h^{l}}\eta(y_{k}^{(\eta)})+%
{\displaystyle\sum\limits_{i=0}^{k-1}}
C_{i}\frac{\Delta_{h}^{i}u(y_{i}^{(u)})\Delta_{h}^{k-i}\eta(y_{i}^{(\eta)}%
)}{h^{l}}\equiv I_{k}+%
{\displaystyle\sum\limits_{i=0}^{k-1}}
I_{i}, \label{s1.35}%
\end{equation}
where $y_{i}^{(u)}=y+n_{i}h$, $y_{i}^{(\eta)}=y+m_{i}h$, and
$n_{i}\leq k$, $m_{i}\leq k$, $C_{i}\leq C(k)$ are some integers.
Evidently, by virtue of \eqref{s1.31}

\begin{equation}
|I_{k}|\leq\sup_{y\in B,h>0}\frac{|\Delta_{h}^{k}u(y)|}{h^{l}}. \label{s1.36}%
\end{equation}
Let us estimate expressions $I_{i}$ in \eqref{s1.35}. First, it
follows from \eqref{s1.31} and the mean value theorem that%

\begin{equation}
|\Delta_{h}^{k-i}\eta(y_{i}^{(\eta)})|\leq C_{k}h^{k-i}y_{0}^{-(k-i)}.
\label{s1.37}%
\end{equation}
Besides, as it follows from \eqref{s1.32},%

\[
|D_{y}^{i}u(y)|\leq C|y-y_{0}|^{m+\gamma-i}\left\langle D_{y}^{m}%
u\right\rangle _{B}^{(\gamma)},\quad y\in B,i=\overline{0,m}.
\]
Since $\varepsilon<1/(8k)$ is sufficiently small and $h<\varepsilon
y_{0}$, it follows from the last inequality and the mean value
theorem that%

\begin{equation}
|\Delta_{h}^{i}u(y_{i}^{(u)})|\leq C_{k}%
\genfrac{\{}{.}{0pt}{0}{h^{i}y_{0}^{m+\gamma-i}\left\langle D_{y}%
^{m}u\right\rangle _{B}^{(\gamma)},\quad i\leq m,}{h^{m+\gamma}\left\langle
D_{y}^{m}u\right\rangle _{B}^{(\gamma)},\quad m<i\leq k-1.}
\label{s1.38}%
\end{equation}
From \eqref{s1.37} and \eqref{s1.38} we have ($h<\varepsilon y_{0}$)%

\[
|I_{i}|\leq C_{k}h^{-l}h^{k-i}y_{0}^{-(k-i)}h^{i}y_{0}^{m+\gamma
-i}\left\langle D_{y}^{m}u\right\rangle _{B}^{(\gamma)}=
\]

\begin{equation}
=C_{k}h^{(k-l)}y_{0}^{-(k-l)}\left\langle D_{y}^{m}u\right\rangle
_{B}^{(\gamma)}\leq C_{k}\varepsilon^{(k-l)}\left\langle D_{y}^{m}%
u\right\rangle _{B}^{(\gamma)},\quad i\leq m, \label{s1.39}%
\end{equation}
and

\[
|I_{i}|\leq C_{k}h^{-l}h^{k-i}y_{0}^{-(k-i)}h^{m+\gamma}\left\langle D_{y}%
^{m}u\right\rangle _{B}^{(\gamma)}=
\]

\begin{equation}
=C_{k}h^{k-i}y_{0}^{-(k-i)}\left\langle D_{y}^{m}u\right\rangle _{B}%
^{(\gamma)}\leq C_{k}\varepsilon^{(k-i)}\left\langle D_{y}^{m}u\right\rangle
_{B}^{(\gamma)},\quad m<i\leq k-1. \label{s1.40}%
\end{equation}
From \eqref{s1.33}- \eqref{s1.36}, \eqref{s1.39}, and \eqref{s1.40}
it follows that the expression
$A_{2}^{(y_{0},h)}(\varepsilon)$ in \eqref{s1.33} is estimated as follows%

\[
A_{2}^{(y_{0},h)}(\varepsilon)\leq C_{k}(y_{0}+h)^{\omega\gamma}\sup_{y\in
B,h>0}\frac{|\Delta_{h}^{k}u(y)|}{h^{l}}+C_{k}\varepsilon^{1-\gamma}%
(y_{0}+h)^{\omega\gamma}\left\langle D_{y}^{m}u\right\rangle _{B}^{(\gamma)}.
\]
Since $h\leq\varepsilon y_{0}$ and on the ball $B$ we have
$y_{0}/4\leq y\leq7y_{0}/4$, we infer

\[
A_{2}^{(y_{0},h)}(\varepsilon)\leq C_{k}\sup_{y\in B,h>0}y^{\omega\gamma}%
\frac{|\Delta_{h}^{k}u(y)|}{h^{l}}+C_{k}\varepsilon^{1-\gamma}\left\langle
D_{y}^{m}u\right\rangle _{\omega\gamma,y,B}^{(\gamma)}.
\]
As the point $y_{0}$ is arbitrary, we obtain \eqref{s1.30.1}.

Combining now estimates \eqref{s1.16.2} and \eqref{s1.30.1}, we
obtain with $\varepsilon<1/8k$

\[
\left\langle D_{y}^{m}u\right\rangle _{\omega\gamma,y}^{(\gamma)}\leq
C_{\varepsilon,k}\sup_{y,h>0}y^{\omega\gamma}\frac{|\Delta_{h}^{k}u(y)|}%
{h^{l}}+C_{k}\varepsilon^{1-\gamma}\left\langle D_{y}^{m}u\right\rangle
_{\omega\gamma,y}^{(\gamma)}.
\]
Choosing now $\varepsilon$ in the last term sufficiently small and
absorbing this term in the left hand side, we arrive at the
assertion of the proposition.

\end{proof}

As a corollary we have the following assertion.

\begin{proposition}
\label{Ps1.2}

Let a function $u(x)$ be defined in $\overline{H}$ and
\[
\left\langle u\right\rangle _{\omega\gamma,x,\overline{H}}^{(\gamma)}%
\equiv\sup_{x,\overline{h}\in\overline{H}}(x_{N}+h_{N})^{\omega\gamma}%
\frac{|u(x+\overline{h})-u(x)|}{\left\vert \overline{h}\right\vert ^{\gamma}%
}<\infty,\quad\gamma\in(0,1).
\]

Then for any integer $k\geq1$ there is a constant $C_{k}^{(i)}=C^{(i)}%
(k,N,\gamma,\omega)$, $i=1,2$ with%

\begin{equation}
\left\langle u\right\rangle _{\omega\gamma,x,\overline{H}}^{(\gamma)}\leq
C_{k}^{(1)}\sup_{x,\overline{h}\in\overline{H}}x_{N}{}^{\omega\gamma}%
\frac{|\Delta_{\overline{h}}^{k}u(x)|}{\left\vert \overline{h}\right\vert
^{\gamma}}\equiv C_{k}^{(1)}\left\langle \left\langle u\right\rangle
\right\rangle _{\omega\gamma,x,\overline{H}}^{(\gamma)(k)} \label{s1.001}%
\end{equation}
and

\begin{equation}
\left\langle \left\langle u\right\rangle \right\rangle _{\omega\gamma
,x,\overline{H}}^{(\gamma)(k)}\leq C_{k}^{(2)}\left\langle u\right\rangle
_{\omega\gamma,x,\overline{H}}^{(\gamma)}. \label{s1.002}%
\end{equation}
\end{proposition}

\begin{proof}

We prove only \eqref{s1.001} because \eqref{s1.002} can be checked
directly.

It is enough to verify the weighted H\"{o}lder property with respect
to $x^{\prime}=(x_{1},...,x_{N-1})$ and $x_{N}$ separately.
Let first $x^{\prime}$ be fixed and $\overline{h}=(0,...,0,h)$, $h>0$. Then%

\[
(x_{N}+h_{N})^{\omega\gamma}\frac{|u(x+\overline{h})-u(x)|}{\left\vert
\overline{h}\right\vert ^{\gamma}}=
\]

\begin{equation}
=(x_{N}+h_{N})^{\omega\gamma}\frac
{|u(x^{\prime},x_{N}+h)-u(x^{\prime},x_{N})|}{h^{\gamma}}\leq C_{k}%
^{(1)}\left\langle \left\langle u\right\rangle \right\rangle _{\omega
\gamma,x,\overline{H}}^{(\gamma)(k)} \label{s1.003}%
\end{equation}
by Proposition \ref{Ps1.1}.  Let now $x_{N}$ be fixed and $\overline{h}=(h_{1}%
,...,h_{N-1},0)=(h^{\prime},0)$ with $h_{N}=0$.  Then

\[
(x_{N}+h_{N})^{\omega\gamma}\frac{|u(x+\overline{h})-u(x)|}{\left\vert
\overline{h}\right\vert ^{\gamma}}=x{}_{N}^{\omega\gamma}\left(
\frac{|u(x^{\prime}+h^{\prime},x_{N})-u(x^{\prime},x_{N})|}{\left\vert
\overline{h}\right\vert ^{\gamma}}\right)  \leq
\]

\begin{equation}
\leq x{}_{N}^{\omega\gamma}\sup_{x^{\prime},h^{\prime}}\frac{|\Delta
_{h^{\prime}}^{k}u(x^{\prime},x_{N})|}{\left\vert h^{\prime}\right\vert
^{\gamma}}=\sup_{x^{\prime},h^{\prime}}x{}_{N}^{\omega\gamma}\frac
{|\Delta_{h^{\prime}}^{k}u(x^{\prime},x_{N})|}{\left\vert h^{\prime
}\right\vert ^{\gamma}}\leq C_{k}^{(1)}\left\langle \left\langle
u\right\rangle \right\rangle _{\omega\gamma,x,\overline{H}}^{(\gamma)(k)}
\label{s1.004}%
\end{equation}
by \eqref{s1.11}.  The assertion of the proposition follows now from
\eqref{s1.003} and \eqref{s1.004}.

\end{proof}

\begin{corollary} \label{Cs1.1}
The seminorms
\[
\left\langle u\right\rangle _{\omega\gamma,\overline{H}}^{(\gamma)}%
=\sup_{x,\overline{x}\in\overline{H}}\left(  x_{N}^{\ast}\right)
^{\omega\gamma}\frac{|u(x)-u(\overline{x})|}{|x-\overline{x}|^{\gamma}%
},\,x_{N}^{\ast}=\max\{x_{N},\overline{x}_{N}\}
\]
with $x_{N}^{\ast}=\max\{x_{N},\overline{x}_{N}\}$ and

\[
\widehat{\left\langle u\right\rangle }_{\omega\gamma,\overline{H}}^{(\gamma
)}=\sup_{x,\overline{x}\in\overline{H}}\left(  \widehat{x}_{N}\right)
^{\omega\gamma}\frac{|u(x)-u(\overline{x})|}{|x-\overline{x}|^{\gamma}},\,
\widehat{x}_{N}=\min\{x_{N},\overline{x}_{N}\}
\]
with $\widehat{x}_{N}=\min\{x_{N},\overline{x}_{N}\}$ are
equivalent.
\end{corollary}

For the proof of this corollary it is enough to choose $k=1$ in
\eqref{s1.001}.

Let now we are given a function $u(x,t)\in
C_{n,\omega\gamma}^{m+\gamma ,\frac{m+\gamma}{m}}(\overline{Q})$. We
are going to construct an analog of Taylor polynomial for such a
function at the point $O=(0,0)$. Under this we mean some
"power-like" function $Q_{u}(x,t)$ with the same asymptotic at
$(x,t)\rightarrow(0,0)$ as that of $u(x,t)$. The simplest situation
for constructing of such a function is when $u(x,t)$ is smooth with
respect to the tangent space variables $x^{\prime}$ and it's
derivatives with respect to these variables need not weights near
$\{x_{N}=0\}$. We are going to achieve this situation by the
smoothing process

\begin{equation}
u_{\varepsilon}(x,t)\equiv%
{\displaystyle\int\limits_{R^{N-1}}}
{\displaystyle\int\limits_{-\infty}^{\infty}}
u(y^{\prime},x_{N},\tau)\omega_{\varepsilon}(x^{\prime}-y^{\prime}%
,t-\tau)dy^{\prime}d\tau, \label{s1.005}%
\end{equation}
where
$\omega_{\varepsilon}(x^{\prime},t)=\varepsilon^{-N}\omega(x^{\prime
}/\varepsilon,t/\varepsilon)$, $\varepsilon>0$,
$\omega(x^{\prime},t)\in C^{\infty}$ is a mollifier with compact
support and with unit total integral. For such more smooth function we have%

\begin{equation}
x_{N}^{n-j}D_{x}^{\alpha}u_{\varepsilon}(x,t)\rightarrow0,x_{N}\rightarrow
0,\quad0\leq j<n,\alpha=(\alpha_{1},...,\alpha_{N}),|\alpha|=m-j,\alpha
_{N}<m-j. \label{s1.006}%
\end{equation}
This means that for the mixed derivatives
$D_{x}^{\alpha}u_{\varepsilon}(x,t)$ of order
$|\alpha|=m-j>\alpha_{N}$ (that is when a derivative $D_{x}^{\alpha
}u_{\varepsilon}(x,t)$ does not coincide with the pure derivative $D_{x_{N}%
}^{m-j}u_{\varepsilon}(x,t)$) we have the property as in
\eqref{s1.006}. It turns out that \eqref{s1.006} is valid in fact
for the function $u(x,t)$ itself without smoothing. But this will be
proved later on.

\begin{lemma}
\label{Ls1.01}
Let $u(x,t)\in C_{n,\omega\gamma}^{m+\gamma,\frac{m+\gamma}{m}}(\overline{Q}%
)$. Then \eqref{s1.006} is valid.
\end{lemma}

\begin{proof}

Show first that for a function $u(x,t)\in C_{n,\omega\gamma
}^{m+\gamma,\frac{m+\gamma}{m}}(\overline{Q})$ for a positive
integer $j\leq m$

\begin{equation}
\left\vert D_{x_{N}}^{m-j}u(x,t)\right\vert \leq F(j;x_{N})\equiv\left\{
\begin{array}
[c]{c}%
Cx_{N}^{-(n-j)},\quad0\leq j<n,\\
C(1+|\ln x_{N}|),\quad j=n,\\
C,\quad n<j\leq m,
\end{array}
\right.  ,\quad0\leq x_{N}\leq2. \label{s1.007.1}%
\end{equation}
Really, for $j=0$ this estimate follows directly from the definition
of the space
$C_{n,\omega\gamma}^{m+\gamma,\frac{m+\gamma}{m}}(\overline{Q}).$
Since the functions from $C_{n,\omega\gamma}^{m+\gamma,\frac
{m+\gamma}{m}}(\overline{Q})$ belong to the standard class
$C^{m+\gamma ,\frac{m+\gamma}{m}}(\overline{Q})$ for $x_{N}>0$, for
$j=1$ we have for $x_{N}\leq2$ (if $1<n$)

\[
D_{x_{N}}^{m-1}u(x,t)=-%
{\displaystyle\int\limits_{x_{N}}^{1}}
\xi^{-n}\left[  \xi^{n}D_{\xi}^{m}u(x^{\prime},\xi,t)\right]  d\xi+D_{x_{N}%
}^{m-1}u(x^{\prime},1,t).
\]
Consequently,

\begin{equation}
\left\vert D_{x_{N}}^{m-1}u(x,t)\right\vert \leq C_{1}%
{\displaystyle\int\limits_{x_{N}}^{1}} \xi^{-n}d\xi+C_{2}\leq\left\{
\begin{array}
[c]{c}%
Cx_{N}^{-(n-1)},\quad n>1,\\
C(1+|\ln x_{N}|),\quad n=1,\\
C,\quad n<1
\end{array}
\right.  =F(1;x_{N}),\quad x_{N}\leq2, \label{s1.008.1}%
\end{equation}
that is \eqref{s1.007} is proved for $j=1$. Now \eqref{s1.007.1} for
$j=2$ follows from \eqref{s1.008.1}  and so on by induction for
$j\leq m$.

Let now $\varepsilon>0$, $j<n$ ,
$\alpha=(\alpha_{1},...,\alpha_{N})$, $|\alpha|=m-j$,
$\alpha_{N}<m-j$. Denoting $\alpha^{\prime}=(\alpha
_{1},...,\alpha_{N-1})$,\ we have

\[
\left\vert x_{N}^{n-j}D_{x}^{\alpha}u_{\varepsilon}(x,t)\right\vert
=\left\vert x_{N}^{n-j}%
{\displaystyle\int\limits_{R^{N-1}}}
{\displaystyle\int\limits_{-\infty}^{\infty}}
\left(  D_{x_{N}}^{\alpha_{N}}u(y^{\prime},x_{N},\tau)\right)  D_{x^{\prime}%
}^{\alpha^{\prime}}\omega_{\varepsilon}(x^{\prime}-y^{\prime},t-\tau
)dy^{\prime}d\tau\right\vert \leq
\]

\[
\leq Cx_{N}^{n-j}F(m-\alpha_{N};x_{N})=C\left\{
\begin{array}
[c]{c}%
x_{N}^{n-j-[n-(m-\alpha_{N})]},m-\alpha_{N}<n,\\
x_{N}^{n-j}(1+\ln x_{N}),m-\alpha_{N}=n,\\
x_{N}^{n-j},m-\alpha_{N}>n
\end{array}
\right.  =
\]

\[
=C\left\{
\begin{array}
[c]{c}%
x_{N}^{(m-j)-\alpha_{N}},m-\alpha_{N}<n,\\
x_{N}^{n-j}(1+|\ln x_{N}|),m-\alpha_{N}=n,\\
x_{N}^{n-j},m-\alpha_{N}>n
\end{array}
\right.  \rightarrow0,\quad x_{N}\rightarrow0.
\]
This proves the lemma.
\end{proof}

\begin{lemma}
\label{Ls1.1}

Let a function $u(x,t)\in C_{n,\omega\gamma}^{m+\gamma,\frac{m+\gamma}{m}%
}(\overline{Q})$ satisfy \eqref{s1.006} without smoothing that is

\begin{equation}
x_{N}^{n-j}D_{x}^{\alpha}u(x,t)\rightarrow0,x_{N}\rightarrow0,\quad0\leq
j<n,\alpha=(\alpha_{1},...,\alpha_{N}),|\alpha|=m-j,\alpha_{N}<m-j.
\label{s1.009.1}%
\end{equation}
Denote

\begin{equation}
a=\lim_{(x,t)\rightarrow(0,0)}x_{N}^{n}D_{x_{N}}^{m}u(x,t). \label{s1.007}%
\end{equation}
and denote

\begin{equation}
\widetilde{Q}_{u}(x_{N})=\left\{
\begin{array}
[c]{c}%
bax_{N}^{m-n},\quad n\quad\text{is a noninteger},\\
ba\ln^{(m-n)}x_{N},\quad n\quad\text{is an integer}.
\end{array}
\right.  \label{s1.008}%
\end{equation}
Here

\begin{equation}
\ln^{(k)}x_{N}\equiv%
{\displaystyle\int\limits_{0}^{x_{N}}}
d\xi_{n}%
{\displaystyle\int\limits_{0}^{\xi k}}
d\xi_{n-1}...%
{\displaystyle\int\limits_{0}^{\xi_{2}}}
\ln\xi_{1}d\xi_{1}, \label{s1.009}%
\end{equation}
$b=b(a,m,n)$ is a constant which is chosen from the condition $a=x_{N}%
^{n}D_{x_{N}}^{m}(\widetilde{Q}_{u}(x_{N}))$.

Then

\begin{equation}
\lim\limits_{(x,t)\rightarrow(0,0)}x_{N}^{n-j}D_{x}^{\alpha}[u(x,t)-\widetilde
{Q}_{u}(x_{N})]=0,\quad|\alpha|=m-j,0\leq j<n, \label{s1.0010}%
\end{equation}

\begin{equation}
x_{N}^{n-j}D_{x}^{\alpha}\widetilde{Q}_{u}(x_{N})\equiv const,\quad
|\alpha|=m-j,0\leq j\leq n,\alpha_{N}<m-n. \label{s1.0010.1}%
\end{equation}

If $n$ is an integer and $\left\langle D_{x_{N}}^{m-n}u\right\rangle
_{\omega\gamma,\overline{Q}}^{(\gamma,\gamma/m)}<\infty$ is finite, then%

\begin{equation}
D_{x_{N}}^{m-n}\widetilde{Q}_{u}(x_{N})\equiv\widetilde{Q}_{u}(x_{N})\equiv0.
\label{s1.0010.2}%
\end{equation}
\end{lemma}

\begin{proof}

Note first that since $m-j$ is an integer and $m-j>m-n>0$, we have
$m-j\geq1$. Now if the derivative $D_{x}^{\alpha}$ contains at least
one differentiation in $x^{\prime}$ then
$D_{x}^{\alpha}[u(x,t)-\widetilde{Q}_{u}(x_{N})]=D_{x}^{\alpha}u(x,t)$
and we have \eqref{s1.0010} by \eqref{s1.009.1}. Let now
$D_{x}^{\alpha}=D_{x_{N}}^{m-j}$. Then by the construction

\begin{equation}
D_{x_{N}}^{m}u(x,t)=\frac{a}{x_{N}^{n}}+o(x_{N}^{-n})=D_{x_{N}}^{m}%
\widetilde{Q}_{u}(x_{N})+o(x_{N}^{-n}),\quad(x,t)\rightarrow(0,0).
\label{s1.0010.3}%
\end{equation}
Integrating this relation with respect to $x_{N}\rightarrow\xi$ on
the interval $[x_{N},1]$ for example, we find for $j<n$%

\begin{equation}
D_{x_{N}}^{m-j}u(x,t)=\frac{b_{j}a}{x_{N}^{n-j}}+o(x_{N}^{-(n-j)})=D_{x_{N}%
}^{m-j}\widetilde{Q}_{u}(x_{N})+o(x_{N}^{-(n-j)}),\,(x,t)\rightarrow(0,0),
\label{s1.0011}%
\end{equation}
where $b_{j}$ are some definite constants and these constants agree
with the condition
$a=x_{N}^{n}D_{x_{N}}^{m}(\widetilde{Q}_{u}(x_{N}))$.  In
particular, if $n$ is an integer

\begin{equation}
D_{x_{N}}^{m-n}u(x,t)=b_{n}a\ln x_{N}+o(|\ln x_{N}|),\,(x,t)\rightarrow(0,0).
\label{s1.0011.1}%
\end{equation}
From \eqref{s1.0011} we obtain \eqref{s1.0010}. Relations
\eqref{s1.0010.1} follows directly from the definition of
$\widetilde{Q}_{u}(x_{N})$ by the construction taking into account
that $\widetilde{Q}_{u}(x_{N})$ depends
on $x_{N}$ only. If now for an integer $n$ we have $\left\langle D_{x_{N}%
}^{m-n}u\right\rangle
_{\omega\gamma,\overline{Q}}^{(\gamma,\gamma/m)}<\infty$ then it
follows from \eqref{s1.0011.1} that
we must have $a=0$ in this relation. But in this case $\widetilde{Q}_{u}%
(x_{N})\equiv0$. This proves \eqref{s1.0010.2}.
\end{proof}

\begin{lemma}
\label{Ls1.2}

Denote%

\begin{equation}
Q_{u}(x,t)=\widetilde{Q}_{u}(x_{N})+%
{\displaystyle\sum\limits_{|\alpha|\leq m-n}}
\frac{a_{\alpha}}{\alpha!}(x-\overline{e})^{\alpha}+a^{(1)}t, \label{s1.0012}%
\end{equation}
where $\widetilde{Q}_{u}(x_{N})$ is defined in \eqref{s1.008},
$\alpha=(\alpha_{1},...,\alpha_{N})$, $\alpha!=\alpha
_{1}!...\alpha_{N}!$, $\overline{e}=(0,...,1)\in R^{N}$,
$(x-\overline{e})^{\alpha}=x_{1}^{\alpha_{1}}...x_{N-1}^{\alpha_{N-1}}%
(x_{N}-1)^{\alpha_{N}}$,

\[
a_{\alpha}=D_{x}^{\alpha}(u-\widetilde{Q}_{u}(x_{N}))|_{x=\overline{e}%
,t=0},\qquad\,a^{(1)}=D_{t}(u-\widetilde{Q}_{u}(x_{N}))|_{x=\overline{e}%
,t=0}.
\]

Then the function $Q_{u}(x,t)$ has the following properties%

\begin{equation}
x_{N}^{n-j}D_{x}^{\alpha}[u(x,t)-Q_{u}(x,t)]|_{(x,t)=(0,0)}=0,\quad
j<n,\,|\alpha|=m-j, \label{s1.0012+1}%
\end{equation}

\begin{equation}
D_{x}^{\alpha}[u(x,t)-Q_{u}(x,t)]|_{(x,t)=(\overline{e},0)}=0,\,|\alpha|\leq
m-n,\hspace{0.05in}D_{t}[u(x,t)-Q_{u}(x,t)]|_{(x,t)=(\overline{e},0)}=0.
\label{s1.0012+2}%
\end{equation}

\begin{equation}
x_{N}^{n-j}D_{x}^{\alpha}Q_{u}(x,t)\equiv const,\quad|\alpha|=m-j,0\leq j\leq
n,\alpha_{N}<m-n,\, \,D_{t}Q_{u}(x,t)\equiv const. \label{s1.0015}%
\end{equation}

If $n$ is an integer and $\left\langle D_{x_{N}}^{m-n}u\right\rangle
_{\omega\gamma,\overline{Q}}^{(\gamma,\gamma/m)}<\infty$ then%

\begin{equation}
D_{x_{N}}^{m-n}Q_{u}(x,t)\equiv const. \label{s1.0016}%
\end{equation}

At last for $j\leq n$ and $|\alpha|=[m-n+(1-\omega)\gamma]-j$%

\begin{equation}
D_{x^{\prime}}^{\alpha}D_{x_{N}}^{j}Q_{u}(x,t)\quad\text{does not depend on
}x^{\prime}\text{ and }t. \label{s1.0017}%
\end{equation}

\end{lemma}

The proof of this lemma follows from Lemma \ref{Ls1.1} directly by
the construction of $Q_{u}(x,t)$ with the
taking into account that for $j<n$ we have ($\beta=(\beta_{1},...,\beta_{N})$)%

\[
D_{x}^{\beta}\left( {\displaystyle\sum\limits_{|\alpha|\leq m-n}}
\frac{a_{\alpha}}{\alpha!}(x-\overline{e})^{\alpha}+a^{(1)}t\right)
\equiv0,\,|\beta|=m-j.
\]

We prove in addition two useful lemmas about H\"{o}lder spaces.
First we prove some lemma that makes the verification of the
H\"{o}lder condition for functions on domains with boundaries more
simple. This is done by restricting the general position of two
different points of a domain to the situation when the two points
are away from a boundary in some sense.

\begin{lemma}
\label{Ls1.3}

Let $\gamma\in(0,1)$ and $\omega\in\lbrack0,1)$ (the case $\omega=0$
corresponds to the nonweighted case).
Let a function $u(y)$, $y\in R^{+}\equiv(0,\infty)$ satisfy the condition%

\begin{equation}
\sup_{0<h\leq\varepsilon y}y^{\omega\gamma}\frac{|u(y+h)-u(y)|}{h^{\gamma}%
}\equiv\left\langle u\right\rangle _{\omega\gamma,R^{+}}^{(\gamma
)(\varepsilon-)}<\infty. \label{s1.0018}%
\end{equation}

Then $u(y)$ is continuous on $[0,\infty)$ and%

\begin{equation}
\left\langle u\right\rangle _{\omega\gamma,\overline{R^{+}}}^{(\gamma)}\leq
C_{\gamma}\varepsilon^{-1}\left\langle u\right\rangle _{\omega\gamma,R^{+}%
}^{(\gamma)(\varepsilon-)}. \label{s1.0019}%
\end{equation}

\end{lemma}

\begin{proof}

Due to Corollary \ref{Cs1.1} and \eqref{s1.0018} it is enough to
verify that

\begin{equation}
\sup_{\varepsilon y\leq h}y^{\omega\gamma}\frac{|u(y+h)-u(y)|}{h^{\gamma}}\leq
C_{\gamma}\varepsilon^{-1}\left\langle u\right\rangle _{\omega\gamma,R^{+}%
}^{(\gamma)(\varepsilon-)}. \label{s1.0020}%
\end{equation}
%where $\sup$ in the left hand side is taken over all $h>0$ without
%the restriction $h\leq\varepsilon y$.

Let $y,h\geq\varepsilon y>0$ be arbitrary, $\Delta_{h}u(y)\equiv
u(y+h)-u(y)$. Denote ($[a]$ is the integer part of $a$)%

\[
M=\left\{
\begin{array}
[c]{c}%
\left[  \log_{(1+\varepsilon)}\left(  \frac{y+h}{y}\right)  \right]  ,\text{
if }\log_{(1+\varepsilon)}\left(  \frac{y+h}{y}\right)  \text{ is a
noninteger,}\\
\log_{(1+\varepsilon)}\left(  \frac{y+h}{y}\right)  -1,\text{ if }%
\log_{(1+\varepsilon)}\left(  \frac{y+h}{y}\right)  \text{ is an integer.}%
\end{array}
\right.
\]
Consider the difference%

\[
u(y+h)-u(y)=\Delta_{h}u(y)=%
{\displaystyle\sum\limits_{i=1}^{M}} (u(y_{i+1})-u(y_{i})),
\]
where

\[
y_{1}=y,\quad y_{i}=y_{i-1}+\varepsilon y_{i-1}=(1+\varepsilon)y_{i-1}%
=(1+\varepsilon)^{i-1}y,i\leq M,\quad y_{M+1}=y+h,
\]
so that $\left(  y_{i+1}-y_{i}\right)  =\varepsilon y_{i}$.  We have

\[
y^{\omega\gamma}\frac{|u(y+h)-u(y)|}{h^{\gamma}}\leq%
{\displaystyle\sum\limits_{i=1}^{M}}
y_{i}^{\omega\gamma}\frac{|u(y_{i+1})-u(y_{i})|}{|y_{i+1}-y_{i}|^{\gamma}%
}\left(  \frac{|y_{i+1}-y_{i}|^{\gamma}}{h^{\gamma}}\right)  \leq
\]

\[
\leq\left\langle u\right\rangle _{\omega\gamma,R^{+}}^{(\gamma)(\varepsilon-)}%
{\displaystyle\sum\limits_{i=1}^{M}} \left(
\frac{|y_{i+1}-y_{i}|}{h}\right)  ^{\gamma}\equiv\left\langle
u\right\rangle _{\omega\gamma,R^{+}}^{(\gamma)(\varepsilon-)}S.
\]
On the other hand%

\[
S\equiv%
{\displaystyle\sum\limits_{i=1}^{M}}
\left(  \frac{|y_{i+1}-y_{i}|}{h}\right)  ^{\gamma}\leq%
{\displaystyle\sum\limits_{i=1}^{M}} \left(
\frac{\varepsilon(1+\varepsilon)^{i-1}y}{h}\right)  ^{\gamma}=
\]

\[
=\varepsilon^{\gamma}\left(  \frac{y}{h}\right)  ^{\gamma}%
{\displaystyle\sum\limits_{i=1}^{M}}
(1+\varepsilon)^{\gamma(i-1)}=\varepsilon^{\gamma}\left(
\frac{y}{h}\right) ^{\gamma}\frac{(1+\varepsilon)^{\gamma
M}-1}{(1+\varepsilon)^{\gamma}-1}.
\]
But according to the definition of the number $M$%

\[
(1+\varepsilon)^{\gamma M}\leq\left(  \frac{y+h}{y}\right)  ^{\gamma},
\]
so that%

\[
S=\left(  \frac{\varepsilon^{\gamma}}{(1+\varepsilon)^{\gamma}-1}\right)
\left[  \left(  \frac{y}{h}\right)  ^{\gamma}\left[  (1+\varepsilon)^{\gamma
M}-1\right]  \right]  \leq C_{\gamma}\varepsilon^{-1+\gamma}\left(  \frac
{y+h}{h}\right)  ^{\gamma}\leq C_{\gamma}\varepsilon^{-1}.
\]
From this \eqref{s1.0020} follows for $y>0$. That is on the open set
$(0,\infty)$

\[
\left\langle u\right\rangle _{\omega\gamma,(0,\infty)}^{(\gamma)}\leq
C_{\gamma}\varepsilon^{-1}\left\langle u\right\rangle _{\omega\gamma,R^{+}%
}^{(\gamma)(\varepsilon-)}.
\]
Then from Corollary \ref{Cs1.1} and the proof of Proposition
\ref{Ps1.01} it follows that

\[
\left\langle u\right\rangle _{(0,\infty)}^{(\gamma-\omega\gamma)}\leq
C_{\gamma}\varepsilon^{-1}\left\langle u\right\rangle _{\omega\gamma,R^{+}%
}^{(\gamma)(\varepsilon-)}.
\]
This means that $u(y)$ has a finite limit as $y\rightarrow0$ and
consequently can be defined at $y=0$ as a continuous function with
\eqref{s1.0019}. Thus the lemma follows.
\end{proof}

\begin{corollary}
\label{Cs1.5}

Let $\gamma\in(0,1)$ and $\omega\in\lbrack0,1)$ (the case $\omega=0$
corresponds to the nonweighted case).
Let a function $u(x)$, $x\in H$ satisfy the condition%

\begin{equation}
\sup_{\overline{h}\in H,|\overline{h}|\leq\varepsilon x_{N}}x_{N}{}%
^{\omega\gamma}\frac{|u(x+\overline{h})-u(x)|}{|\overline{h}|^{\gamma}}%
\equiv\left\langle u\right\rangle _{\omega\gamma,H}^{(\gamma)(\varepsilon
-)}<\infty. \label{s1.0021}%
\end{equation}

Then $u(x)$ is continuous on $\overline{H}$ and%

\begin{equation}
\left\langle u\right\rangle _{\omega\gamma,\overline{H}}^{(\gamma)}\leq
C_{\gamma}\varepsilon^{-1-\gamma}\left\langle u\right\rangle _{\omega\gamma
,H}^{(\gamma)(\varepsilon-)}. \label{s1.0022}%
\end{equation}

\end{corollary}

\begin{proof}

In view of Lemma \ref{Ls1.3} for arbitrary fixed $x^{\prime}$ we have%

\begin{equation}
\left\langle u(x^{\prime},\cdot)\right\rangle _{\omega\gamma,x_{N},[0,\infty
)}^{(\gamma)}\leq C_{\gamma}\varepsilon^{-1}\left\langle u\right\rangle
_{\omega\gamma,H}^{(\gamma)(\varepsilon-)}. \label{s1.0023}%
\end{equation}
Therefore it is enough to consider the H\"{o}lder property of $u(x)$
with respect to the tangent variables $x^{\prime}$ only under the
condition $|\overline {h}|\geq\varepsilon x_{N}$. That is for a
given $\overline{h^{\prime}}=(h_{1},...,h_{N-1})$, and for
$x_{N}>0$\ with $|\overline{h^{\prime}}|\geq\varepsilon x_{N}$ we
must estimate the expression

\[
A(x,\overline{h^{\prime}})\equiv x_{N}{}^{\omega\gamma}\frac{|u(x^{\prime
}+\overline{h^{\prime}},x_{N})-u(x^{\prime},x_{N})|}{|\overline{h^{\prime}%
}|^{\gamma}}.
\]
We estimate $A(x,\overline{h^{\prime}})$ as follows%

\[
A(x,\overline{h^{\prime}})\leq x_{N}{}^{\omega\gamma}\frac{|u(x^{\prime
}+\overline{h^{\prime}},x_{N}+\varepsilon^{-1}|\overline{h^{\prime}%
}|)-u(x^{\prime},x_{N}+\varepsilon^{-1}|\overline{h^{\prime}}|)|}%
{|\overline{h^{\prime}}|^{\gamma}}+
\]

\[
+x_{N}{}^{\omega\gamma}\frac{|u(x^{\prime}+\overline{h^{\prime}}%
,x_{N}+\varepsilon^{-1}|\overline{h^{\prime}}|)-u(x^{\prime}+\overline
{h^{\prime}},x_{N})|}{|\overline{h^{\prime}}|^{\gamma}}+x_{N}{}^{\omega\gamma
}\frac{|u(x^{\prime},x_{N}+\varepsilon^{-1}|\overline{h^{\prime}%
}|)-u(x^{\prime},x_{N})|}{|\overline{h^{\prime}}|^{\gamma}}.
\]
In view of \eqref{s1.0021} and \eqref{s1.0023} the proof of the
corollary is completed now exactly as in Lemma \ref{Ls1.3}.
\end{proof}

Now we prove a simple lemma about compactness and convergence in
weighted H\"{o}lder spaces. This assertion is "almost well known".
But the experience of the author shows that the following simple
fact is not generally known: after convergence of a sequence from a
H\"{o}lder space in a weaker H\"{o}lder space the limit belongs to
the original space. This fact is a very useful tool in applications
because smooth functions are not dense in H\"{o}lder spaces. The
precise statement is as follows.

\begin{proposition}
\label{Ps1.04}

Let $\gamma\in(0,1)$, $\omega\in\lbrack0,1)$.  Let
$K\subset\overline{H}$ be a compact domain in $\overline{H}$ with
smooth boundary. Let $U\subset C_{\omega\gamma}^{\gamma}(K)$ be a
bounded subset in  $C_{\omega\gamma
}^{\gamma}(K)$ that is%

\begin{equation}
u(x)\in U\Rightarrow\left\Vert u\right\Vert _{C_{\omega\gamma}^{\gamma}%
(K)}\leq M \label{s1.0024}%
\end{equation}
for some constant $M>0$.

Then there exists a sequence $\{u_{n}(x)\} \subset U$ and a function
$u_{0}(x)\in$ $C_{\omega\gamma}^{\gamma}(K)$ from the same space
$C_{\omega\gamma}^{\gamma}(K)$ such that for any $\gamma
^{\prime}\in(0,\gamma)$

\begin{equation}
\left\Vert u_{n}-u_{0}\right\Vert _{C_{\omega\gamma^{\prime}}^{\gamma^{\prime
}}(K)}+\left\Vert u_{n}-u_{0}\right\Vert _{C^{(1-\omega)\gamma^{\prime}}%
(K)}\rightarrow_{n\rightarrow\infty}0,\quad\left\Vert u_{0}\right\Vert
_{C_{\omega\gamma}^{\gamma}(K)}\leq M. \label{s1.0025}%
\end{equation}
\end{proposition}

\begin{proof}

From Proposition \ref{Ps1.01} and from \eqref{s1.0024} it follows
that

\[
u(x)\in U\Rightarrow\left\Vert u\right\Vert _{C^{(1-\omega)\gamma}(K)}\leq
CM.
\]
Thus, as it is well known, there exists a sequence $\{u_{n}(x)\}
\subset U$
and a function $u_{0}(x)\in$ $\cap_{\gamma^{\prime}\in(0,\gamma)}%
C^{(1-\omega)\gamma^{\prime}}(K)$ with

\begin{equation}
\left\Vert u_{n}-u_{0}\right\Vert _{C^{(1-\omega)\gamma^{\prime}}%
(K)}\rightarrow_{n\rightarrow\infty}0,\quad\gamma^{\prime}\in(0,\gamma).
\label{s1.0026}%
\end{equation}
Let us show that $u_{0}(x)$ belongs to the original space
$C_{\omega\gamma}^{\gamma}(K)$ and the estimate in \eqref{s1.0025}
is valid. Let $x\in\overline{K}$ and $\overline{h}\neq0\in H$ be
fixed and such that $x+\overline{h}\in\overline{K}$.  Consider the
expression

\begin{equation}
A_{n}(x,\overline{h})=x_{N}^{\omega\gamma}\frac{|u_{n}(x+\overline{h}%
)-u_{n}(x)|}{|\overline{h}|^{\gamma}}\leq M \label{s1.0026.1}%
\end{equation}
and suppose that $x_{N}>0$ because in the case $x_{N}=0$ we have
$A_{n}(x,\overline{h})=0$. From \eqref{s1.0026} it follows that
$u_{n}(x)\rightarrow u_{0}(x)$ uniformly on $K$. Therefore letting
$n\rightarrow\infty$ in \eqref{s1.0026.1}, we obtain

\[
A_{0}(x,\overline{h})=x_{N}^{\omega\gamma}\frac{|u_{0}(x+\overline{h}%
)-u_{0}(x)|}{|\overline{h}|^{\gamma}}\leq M.
\]
Since $x$ and $\overline{h}$ are arbitrary we infer from the last
inequality and from \eqref{s1.0026}

\[
u_{0}(x)\in C_{\omega\gamma}^{\gamma}(K),\quad\left\Vert u_{0}\right\Vert
_{C_{\omega\gamma}^{\gamma}(K)}\leq M.
\]

Let us show now that

\begin{equation}
\left\Vert u_{n}-u_{0}\right\Vert _{C_{\omega\gamma^{\prime}}^{\gamma^{\prime
}}(K)}\rightarrow_{n\rightarrow\infty}0,\quad\gamma^{\prime}\in(0,\gamma).
\label{s1.0027}%
\end{equation}

Let $x\in\overline{K}$ and $\overline{h}\neq0\in H$ be such that
$x+\overline{h}\in\overline{K}$ and let
$\gamma^{\prime}\in(0,\gamma)$. Denote $v_{n}(x)=u_{n}(x)-u_{0}(x)$.
and consider the expression

\[
A_{n}(x,\overline{h})=x_{N}^{\omega\gamma^{\prime}}\frac{|v_{n}(x+\overline
{h})-v_{n}(x)|}{|\overline{h}|^{\gamma^{\prime}}}.
\]
Let we are given an $\varepsilon>0$.  If $x_{N}=0$ then
$A_{n}(x,\overline {h})=0$ therefore we suppose that $x_{N}>0$.\
Denote $R_{K}=\inf\{R>0:K\subset\{0\leq x_{N}\leq R\} \}$\ \ and
consider two cases.
If $|\overline{h}|\leq\varepsilon x_{N}$ then we have%

\[
A_{n}(x,\overline{h})=x_{N}^{\omega\gamma^{\prime}}|\overline{h}%
|^{\gamma-\gamma^{\prime}}\frac{|v_{n}(x+\overline{h})-v_{n}(x)|}%
{|\overline{h}|^{\gamma}}\leq
\]

\[
\leq\varepsilon^{\gamma-\gamma^{\prime}}R_{K}^{(1-\omega)(\gamma
-\gamma^{\prime})}\left(
x_{N}^{\omega\gamma}\frac{|v_{n}(x+\overline
{h})-v_{n}(x)|}{|\overline{h}|^{\gamma}}\right)  \leq
\]

\[
\leq \varepsilon
^{\gamma-\gamma^{\prime}}R_{K}^{(1-\omega)(\gamma-\gamma^{\prime})}\left(
\left\langle u_{n}\right\rangle
_{\omega\gamma,K}^{(\gamma)}+\left\langle u_{0}\right\rangle
_{\omega\gamma,K}^{(\gamma)}\right)  \leq
\]

\begin{equation}
\leq\varepsilon^{\gamma-\gamma^{\prime}}2MR_{K}^{(1-\omega)(\gamma
-\gamma^{\prime})}. \label{s1.0028}%
\end{equation}
If now $|\overline{h}|>\varepsilon x_{N}$ then

\begin{equation}
A_{n}(x,\overline{h})=\left(  \frac{x_{N}}{|\overline{h}|}\right)
^{\omega\gamma^{\prime}}\frac{|v_{n}(x+\overline{h})-v_{n}(x)|}{|\overline
{h}|^{(1-\omega)\gamma^{\prime}}}\leq\varepsilon^{-\omega\gamma^{\prime}%
}\left\langle u_{n}-u_{0}\right\rangle _{K}^{((1-\omega)\gamma^{\prime})}.
\label{s1.0029}%
\end{equation}
Since $x$ and $\overline{h}$ are arbitrary, from \eqref{s1.0028} and
\eqref{s1.0029} it follows that

\[
\left\langle u_{n}-u_{0}\right\rangle _{\omega\gamma^{\prime},K}%
^{(\gamma^{\prime})}\leq\varepsilon^{\gamma-\gamma^{\prime}}C(M,K)+\varepsilon
^{-\omega\gamma^{\prime}}\left\langle u_{n}-u_{0}\right\rangle _{K}%
^{((1-\omega)\gamma^{\prime})}.
\]
Taking into account \eqref{s1.0026}, we have

\[
\left\Vert u_{n}-u_{0}\right\Vert _{C_{\omega\gamma^{\prime}}^{\gamma^{\prime
}}(K)}\leq\varepsilon^{\gamma-\gamma^{\prime}}C(M,K)+\varepsilon
^{-\omega\gamma^{\prime}}\left\Vert u_{n}-u_{0}\right\Vert _{C^{(1-\omega
)\gamma^{\prime}}(K)}.
\]
From this we see that the left hand side can be made arbitrary small
for large $n$ by choosing first $\varepsilon$ sufficiently small and
then $n\geq N(\varepsilon)$ sufficiently large.

This completes the proof of the proposition.
\end{proof}

\section{Proof of Theorem \ref{Ts1.1} }
\label{ss1.3}

We prove only \eqref{s1.7} since \eqref{s1.8} is a consequence of
\eqref{s1.7} for functions without dependance on $t$. We use the
idea of scaling arguments from \cite{LS} and the reasoning by
contradiction exactly as in the proof of Proposition \ref{Ps1.1}.

On the base of Proposition \ref{Ps1.2} we can turn to prof of  the
estimate

\[
\left\langle \left\langle u\right\rangle \right\rangle _{n,\omega
\gamma,\overline{Q}}^{(m+\gamma)(2s)}\equiv%
{\displaystyle\sum\limits_{j=0}^{j\leq n}}
{\displaystyle\sum\limits_{|\alpha|=m-j}} \left\langle \left\langle
x_{N}^{n-j}D_{x}^{\alpha}u\right\rangle
\right\rangle _{\omega\gamma,x,\overline{Q}}^{(\gamma)(2s)}+%
{\displaystyle\sum\limits_{j=0}^{j\leq n}}
{\displaystyle\sum\limits_{|\alpha|=m-j}} \left\langle \left\langle
x_{N}^{n-j}D_{x}^{\alpha}u\right\rangle \right\rangle
_{t,\overline{Q}}^{(\gamma/m)(2s)}+
\]

\[
+%
{\displaystyle\sum\limits_{j=0}^{j\leq n}}
{\displaystyle\sum\limits_{|\alpha|=m-j}} \left\langle \left\langle
x_{N}^{n-j\omega}D_{x}^{\alpha}u\right\rangle \right\rangle
_{t,\overline{Q}}^{(\frac{\gamma+j}{m})(2s)}+\left\langle
\left\langle D_{t}u\right\rangle \right\rangle _{\omega\gamma,x,\overline{Q}%
}^{(\gamma)(4)}+\left\langle \left\langle D_{t}u\right\rangle \right\rangle
_{t,\overline{Q}}^{(\gamma/m)(4)}+
\]

\[
+%
{\displaystyle\sum\limits_{j=0}^{j\leq m-n}}
{\displaystyle\sum\limits_{|\alpha|=[m-n+(1-\omega)\gamma]-j}}
\left\langle \left\langle
D_{x^{\prime}}^{\alpha}D_{x_{N}}^{j}u\right\rangle
\right\rangle _{x^{\prime},\overline{Q}}^{(\{m-n+(1-\omega)\gamma\})(2s)}+%
\]

\[
+{\displaystyle\sum\limits_{j=0}^{j\leq m-n}}
{\displaystyle\sum\limits_{|\alpha|=[m-n+\gamma]-j}} \left\langle
\left\langle D_{x^{\prime}}^{\alpha}D_{x_{N}}^{j}u\right\rangle
\right\rangle
_{\omega\gamma,x^{\prime},\overline{Q}}^{(\{m-n+\gamma\})(2s)}+
\]

\begin{equation}
+%
{\displaystyle\sum\limits_{j=1}^{j\leq m-n}}
{\displaystyle\sum\limits_{|\alpha|=j}} \left\langle \left\langle
D_{x}^{\alpha}u\right\rangle \right\rangle
_{t,\overline{Q}}^{(1-\frac{j}{m-n}+\frac{\gamma}{m})(2s)}\leq
C\left( {\displaystyle\sum\limits_{i=1}^{N}} \left\langle
x_{N}^{n}D_{x_{i}}^{m}u\right\rangle _{\omega\gamma
,x_{i},\overline{Q}}^{(\gamma)}+\left\langle D_{t}u\right\rangle
_{t,\overline{Q}}^{(\gamma/m)}\right)  , \label{s1.41}%
\end{equation}
where $s=m+1$ and for a function $v(x,t)$ we denote
($\varepsilon,\beta \in(0,1)$)

\[
\left\langle \left\langle v\right\rangle \right\rangle _{\omega\gamma
,x,\overline{Q}}^{(\gamma)(k)}\equiv\sup_{(x,t)\in\overline{Q},\overline{h}%
\in\overline{H}}x_{N}^{\omega\gamma}\frac{|\Delta_{\overline{h},x}^{k}%
v(x,t)|}{|\overline{h}|^{\gamma}}\leq\sup_{(x,t)\in\overline{Q},\overline
{h}\in\overline{H},|\overline{h}|\geq\varepsilon x_{N}}x_{N}^{\omega\gamma
}\frac{|\Delta_{\overline{h},x}^{k}v(x,t)|}{|\overline{h}|^{\gamma}}+
\]

\begin{equation}
+\sup_{(x,t)\in\overline{Q},\overline{h}\in\overline{H},|\overline{h}%
|\leq\varepsilon x_{N}}x_{N}^{\omega\gamma}\frac{|\Delta_{\overline{h},x}%
^{k}v(x,t)|}{|\overline{h}|^{\gamma}}\equiv\left\langle \left\langle
v\right\rangle \right\rangle _{\omega\gamma,x,\overline{Q}}^{(\gamma
)(k)(\varepsilon+)}+\left\langle \left\langle v\right\rangle \right\rangle
_{\omega\gamma,x,\overline{Q}}^{(\gamma)(k)(\varepsilon-)}, \label{s1.42}%
\end{equation}

\[
\left\langle \left\langle v\right\rangle \right\rangle _{x^{\prime}%
,\overline{Q}}^{(\gamma)(k)}\equiv\sup_{(x,t)\in\overline{Q},\overline
{h}^{\prime}\in R^{N-1}}\frac{|\Delta_{\overline{h}^{\prime},x^{\prime}}%
^{k}v(x,t)|}{|\overline{h}^{\prime}|^{\gamma}}\leq\sup_{(x,t)\in\overline
{Q},\overline{h}^{\prime}\in R^{N-1},|\overline{h}^{\prime}|\geq\varepsilon
x_{N}}\frac{|\Delta_{\overline{h}^{\prime},x^{\prime}}^{k}v(x,t)|}%
{|\overline{h}^{\prime}|^{\gamma}}+
\]

\begin{equation}
+\sup_{(x,t)\in\overline{Q},\overline{h}^{\prime}\in R^{N-1},|\overline
{h}^{\prime}|\leq\varepsilon x_{N}}\frac{|\Delta_{\overline{h}^{\prime
},x^{\prime}}^{k}v(x,t)|}{|\overline{h}^{\prime}|^{\gamma}}\equiv\left\langle
\left\langle v\right\rangle \right\rangle _{x^{\prime},\overline{Q}}%
^{(\gamma)(k)(\varepsilon+)}+\left\langle \left\langle v\right\rangle
\right\rangle _{x^{\prime},\overline{Q}}^{(\gamma)(k)(\varepsilon-)},
\label{s1.42.1}%
\end{equation}

\[
\left\langle \left\langle v\right\rangle \right\rangle _{t,\overline{Q}%
}^{(\beta)(k)}\equiv\sup_{(x,t)\in\overline{Q},h>0}\frac{|\Delta_{h,t}%
^{k}v(x,t)|}{h^{\beta}}\leq\sup_{(x,t)\in\overline{Q},h\geq\varepsilon x_{N}%
}\frac{|\Delta_{h,t}^{k}v(x,t)|}{h^{\beta}}+
\]

\begin{equation}
+\sup_{(x,t)\in\overline{Q},h\leq\varepsilon x_{N}}\frac{|\Delta_{h,t}%
^{k}v(x,t)|}{h^{\beta}}\equiv\left\langle \left\langle v\right\rangle
\right\rangle _{t,\overline{Q}}^{(\beta)(k)(\varepsilon+)}+\left\langle
\left\langle v\right\rangle \right\rangle _{t,\overline{Q}}^{(\beta
)(k)(\varepsilon-)}, \label{s1.43}%
\end{equation}

\[
\Delta_{\overline{h},x}v(x,t)=\Delta_{\overline{h},x}^{1}v(x,t)=v(x+\overline
{h})-v(x),\Delta_{\overline{h},x}^{k}v(x,t)=\Delta_{\overline{h},x}%
(\Delta_{\overline{h},x}^{k-1}v(x,t)),
\]

\[
\Delta_{h,t}v(x,t)=\Delta_{h,t}^{1}v(x,t)=v(x,t+h)-v(x,t),\Delta_{h,t}%
^{k}v(x,t)=\Delta_{h,t}(\Delta_{h,t}^{k-1}v(x,t)).
\]
We first prove \eqref{s1.41} under the additional restriction
\eqref{s1.009.1} on functions $u(x,t)$, that is we suppose that

\begin{equation}
x_{N}^{n-j}D_{x}^{\alpha}u(x,t)\rightarrow0,x_{N}\rightarrow0,\quad0\leq
j<n,\alpha=(\alpha_{1},...,\alpha_{N}),|\alpha|=m-j,\alpha_{N}<m-j.
\label{s1.43.01}%
\end{equation}
According to the definitions in \eqref{s1.42}, \eqref{s1.43} we
represent left hand side of \eqref{s1.41} as

\begin{equation}
\left\langle \left\langle u\right\rangle \right\rangle _{n,\omega
\gamma,\overline{Q}}^{(m+\gamma)(2s)}\leq\left\langle \left\langle
u\right\rangle \right\rangle _{n,\omega\gamma,\overline{Q}}^{(m+\gamma
)(2s)(\varepsilon+)}+\left\langle \left\langle u\right\rangle \right\rangle
_{n,\omega\gamma,\overline{Q}}^{(m+\gamma)(2s)(\varepsilon-)}, \label{s1.44}%
\end{equation}
where correspondingly

\begin{equation}
\left\langle \left\langle u\right\rangle \right\rangle _{n,\omega
\gamma,\overline{Q}}^{(m+\gamma)(2s)(\varepsilon\pm)}(u)\equiv%
{\displaystyle\sum\limits_{j=0}^{j\leq n}}
{\displaystyle\sum\limits_{|\alpha|=m-j}} \left\langle \left\langle
x_{N}^{n-j}D_{x}^{\alpha}u\right\rangle
\right\rangle _{\omega\gamma,x,\overline{Q}}^{(\gamma)(2s)(\varepsilon\pm)}+%
\label{s1.45}
\end{equation}

\[
+{\displaystyle\sum\limits_{j=0}^{j\leq n}}
{\displaystyle\sum\limits_{|\alpha|=m-j}} \left\langle \left\langle
x_{N}^{n-j}D_{x}^{\alpha}u\right\rangle
\right\rangle _{t,\overline{Q}}^{(\gamma/m)(2s)(\varepsilon\pm)}+ %
\]

\[
+%
{\displaystyle\sum\limits_{j=0}^{j\leq m-n}}
{\displaystyle\sum\limits_{|\alpha|=[m-n+(1-\omega)\gamma]-j}}
\left\langle \left\langle
D_{x^{\prime}}^{\alpha}D_{x_{N}}^{j}u\right\rangle \right\rangle
_{x^{\prime},\overline{Q}}^{(\{m-n+(1-\omega)\gamma
\})(2s)(\varepsilon\pm)}+%
\]

\[
+{\displaystyle\sum\limits_{j=1}^{j\leq m-n}}
{\displaystyle\sum\limits_{|\alpha|=j}} \left\langle \left\langle
D_{x}^{\alpha}u\right\rangle \right\rangle
_{t,\overline{Q}}^{(1-\frac{j}{m-n}+\frac{\gamma}{m})(2s)(\varepsilon\pm)}+
\]

\[
+%
{\displaystyle\sum\limits_{j=0}^{j\leq m-n}}
{\displaystyle\sum\limits_{|\alpha|=[m-n+\gamma]-j}} \left\langle
\left\langle D_{x^{\prime}}^{\alpha}D_{x_{N}}^{j}u\right\rangle
\right\rangle _{\omega\gamma,x^{\prime},\overline{Q}}^{(\{m-n+\gamma
\})(2s)(\varepsilon\pm)}+
\]

\[
+%
{\displaystyle\sum\limits_{j=0}^{j\leq n}}
{\displaystyle\sum\limits_{|\alpha|=m-j}} \left\langle \left\langle
x_{N}^{n-j\omega}D_{x}^{\alpha}u\right\rangle \right\rangle
_{t,\overline{Q}}^{(\frac{\gamma+j}{m})(2s)(\varepsilon\pm
)}+\left\langle \left\langle D_{t}u\right\rangle \right\rangle
_{\omega
\gamma,x,\overline{Q}}^{(\gamma)(4)(\varepsilon\pm)}+\left\langle
\left\langle D_{t}u\right\rangle \right\rangle
_{t,\overline{Q}}^{(\gamma/m)(4)(\varepsilon \pm)}.
\]

Suppose first that

\begin{equation}
\left\langle \left\langle u\right\rangle \right\rangle _{n,\omega
\gamma,\overline{Q}}^{(m+\gamma)(2s)(\varepsilon-)}\leq\left\langle
\left\langle u\right\rangle \right\rangle _{n,\omega\gamma,\overline{Q}%
}^{(m+\gamma)(2s)(\varepsilon+)}, \label{s1.45.1}%
\end{equation}
and consequently%

\begin{equation}
\left\langle \left\langle u\right\rangle \right\rangle _{n,\omega
\gamma,\overline{Q}}^{(m+\gamma)(2s)(\varepsilon+)}\leq\left\langle
\left\langle u\right\rangle \right\rangle _{n,\omega\gamma,\overline{Q}%
}^{(m+\gamma)(2s)}\leq2\left\langle \left\langle u\right\rangle \right\rangle
_{n,\omega\gamma,\overline{Q}}^{(m+\gamma)(2s)(\varepsilon+)}. \label{s1.46}%
\end{equation}
Let us show that on the class of functions $u$ with this condition%

\begin{equation}
\left\langle \left\langle u\right\rangle \right\rangle _{n,\omega
\gamma,\overline{Q}}^{(m+\gamma)(2s)(\varepsilon+)}\leq
C_{\varepsilon}\left( {\displaystyle\sum\limits_{i=1}^{N}}
\left\langle x_{N}^{n}D_{x_{i}}^{m}u\right\rangle _{\omega\gamma
,x_{i},\overline{Q}}^{(\gamma)}+\left\langle D_{t}u\right\rangle
_{t,\overline{Q}}^{(\gamma/m)}\right)  . \label{s1.47}%
\end{equation}
The proof is by contradiction. Suppose that \eqref{s1.47} is not
valid. Then there exists \ a sequence $\{u_{p}(x,t)\} \subset
C_{n,\omega\gamma}^{m+\gamma,\frac{m+\gamma}{m}}(\overline{Q})$,
$p=1,2,...,$ , with the property \eqref{s1.43.01} and with

\begin{equation}
\left\langle \left\langle u_{p}\right\rangle \right\rangle
_{n,\omega \gamma,\overline{Q}}^{(m+\gamma)(2s)(\varepsilon+)}\geq
p\left( {\displaystyle\sum\limits_{i=1}^{N}} \left\langle
x_{N}^{n}D_{x_{i}}^{m}u_{p}\right\rangle _{\omega\gamma
,x_{i},\overline{Q}}^{(\gamma)}+\left\langle D_{t}u_{p}\right\rangle
_{t,\overline{Q}}^{(\gamma/m)}\right)  \label{s1.48}%
\end{equation}
and%

\begin{equation}
\left\langle \left\langle u_{p}\right\rangle \right\rangle _{n,\omega
\gamma,\overline{Q}}^{(m+\gamma)(2s)(\varepsilon+)}\leq\left\langle
\left\langle u_{p}\right\rangle \right\rangle _{n,\omega\gamma,\overline{Q}%
}^{(m+\gamma)(2s)}\leq2\left\langle \left\langle u_{p}\right\rangle
\right\rangle _{n,\omega\gamma,\overline{Q}}^{(m+\gamma)(2s)(\varepsilon+)}.
\label{s1.49}%
\end{equation}
Denote $v_{p}(x,t)\equiv u_{p}(x,t)/\left\langle \left\langle u_{p}%
\right\rangle \right\rangle
_{n,\omega\gamma,\overline{Q}}^{(m+\gamma )(2s)(\varepsilon+)}$. For
the functions $\{v_{p}\}$ we have from \eqref{s1.48}

\[
1=\left\langle \left\langle v_{p}\right\rangle \right\rangle
_{n,\omega \gamma,\overline{Q}}^{(m+\gamma)(2s)(\varepsilon+)}\geq
p\left( {\displaystyle\sum\limits_{i=1}^{N}} \left\langle
x_{N}^{n}D_{x_{i}}^{m}v_{p}\right\rangle _{\omega\gamma
,x_{i},\overline{Q}}^{(\gamma)}+\left\langle D_{t}v_{p}\right\rangle
_{t,\overline{Q}}^{(\gamma/m)}\right)  .
\]
And from the last inequality and from \eqref{s1.48} we infer that%

\[
{\displaystyle\sum\limits_{i=1}^{N}} \left\langle
x_{N}^{n}D_{x_{i}}^{m}v_{p}\right\rangle _{\omega\gamma
,x_{i},\overline{Q}}^{(\gamma)}+\left\langle D_{t}v_{p}\right\rangle
_{t,\overline{Q}}^{(\gamma/m)}\leq\frac{1}{p},
\]

\begin{equation}%
\quad1\leq\left\langle
\left\langle v_{p}\right\rangle \right\rangle _{n,\omega\gamma,\overline{Q}%
}^{(m+\gamma)(2s)}\leq2\left\langle \left\langle v_{p}\right\rangle
\right\rangle _{n,\omega\gamma,\overline{Q}}^{(m+\gamma)(2s)(\varepsilon
+)}\leq2. \label{s1.50}%
\end{equation}
It follows from the second inequality in \eqref{s1.50} that there is
a term in the definition of $\left\langle
\left\langle v_{p}\right\rangle \right\rangle _{n,\omega\gamma,\overline{Q}%
}^{(m+\gamma)(2s)(\varepsilon+)}$,

\begin{equation}
\left\langle \left\langle v_{p}\right\rangle \right\rangle _{n,\omega
\gamma,\overline{Q}}^{(m+\gamma)(2s)(\varepsilon+)}=%
{\displaystyle\sum\limits_{j=0}^{j\leq n}}
{\displaystyle\sum\limits_{|\alpha|=m-j}} \left\langle \left\langle
x_{N}^{n-j}D_{x}^{\alpha}v_{p}\right\rangle
\right\rangle _{\omega\gamma,x,\overline{Q}}^{(\gamma)(2s)(\varepsilon+)}+%
\label{s1.51}
\end{equation}

\[
+{\displaystyle\sum\limits_{j=0}^{j\leq n}}
{\displaystyle\sum\limits_{|\alpha|=m-j}} \left\langle \left\langle
x_{N}^{n-j}D_{x}^{\alpha}v_{p}\right\rangle
\right\rangle _{t,\overline{Q}}^{(\gamma/m)(2s)(\varepsilon+)}+ %
\]

\[
+%
{\displaystyle\sum\limits_{j=0}^{j\leq m-n}}
{\displaystyle\sum\limits_{|\alpha|=[m-n+(1-\omega)\gamma]-j}}
\left\langle \left\langle D_{x^{\prime}}^{\alpha}D_{x_{N}}^{j}v_{p}%
\right\rangle \right\rangle _{x^{\prime},\overline{Q}}^{(\{m-n+(1-\omega
)\gamma\})(2s)(\varepsilon+)}+%
\]

\[
+{\displaystyle\sum\limits_{j=1}^{j\leq m-n}}
{\displaystyle\sum\limits_{|\alpha|=j}} \left\langle \left\langle
D_{x}^{\alpha}v_{p}\right\rangle \right\rangle
_{t,\overline{Q}}^{(1-\frac{j}{m-n}+\frac{\gamma}{m})(2s)(\varepsilon+)}+
\]

\[
+%
{\displaystyle\sum\limits_{j=0}^{j\leq m-n}}
{\displaystyle\sum\limits_{|\alpha|=[m-n+\gamma]-j}}
\left\langle \left\langle D_{x^{\prime}}^{\alpha}D_{x_{N}}^{j}v_{p}%
\right\rangle \right\rangle _{\omega\gamma,x^{\prime},\overline{Q}%
}^{(\{m-n+\gamma\})(2s)(\varepsilon+)}+
\]

\[
+%
{\displaystyle\sum\limits_{j=0}^{j\leq n}}
{\displaystyle\sum\limits_{|\alpha|=m-j}} \left\langle \left\langle
x_{N}^{n-j\omega}D_{x}^{\alpha}v_{p}\right\rangle \right\rangle
_{t,\overline{Q}}^{(\frac{\gamma+j}{m})(2s)(\varepsilon
+)}+\left\langle \left\langle D_{t}v_{p}\right\rangle \right\rangle
_{\omega\gamma,x,\overline{Q}}^{(\gamma)(4)(\varepsilon+)}+\left\langle
\left\langle D_{t}v_{p}\right\rangle \right\rangle _{t,\overline{Q}}%
^{(\gamma/m)(4)(\varepsilon+)}\geq\frac{1}{2},
\]
which is not less than some absolute constant $\nu=\nu(m,n,N)>0$.
This is valid at least for a subsequence of indexes $\{p\}$. It can
not be the sequence of terms $\left\langle \left\langle
D_{t}v_{p}\right\rangle \right\rangle
_{t,\overline{Q}}^{(\gamma/m)(4)(\varepsilon+)}$ because of
\eqref{s1.50}. We suppose, for example, that for some multiindex
$\widehat{\alpha}$, $|\widehat{\alpha}|=m$,

\begin{equation}
\left\langle \left\langle x_{N}^{n}D_{x}^{\widehat{\alpha}}v_{p}\right\rangle
\right\rangle _{\omega\gamma,x,\overline{Q}}^{(\gamma)(2s)(\varepsilon+)}%
\geq\nu>0,\quad p=1,2,..... \label{s1.52}%
\end{equation}
The all reasonings below are completely the same for all other terms
in \eqref{s1.51}.

From \eqref{s1.52}  and from the definition of  $\left\langle
\left\langle x_{N}^{n}D^{\alpha}v_{p}\right\rangle \right\rangle
_{\omega\gamma ,x,\overline{Q}}^{(\gamma)(2s)(\varepsilon+)}$ in
\eqref{s1.42} it follows that there exist sequences of points
$\{(x^{(p)},t^{(p)})\in\overline{Q}\}$ and
vectors $\{ \overline{h}^{(p)}\in\overline{H}\}$ with%

\begin{equation}
h_{p}\equiv|\overline{h}^{(p)}|\geq\varepsilon x_{N}^{(p)},\quad p=1,2,...
\label{s1.53}%
\end{equation}
and with%

\begin{equation}
\left(  x_{N}^{(p)}\right)  ^{\omega\gamma}\frac{|\Delta_{\overline{h}^{(p)}%
}^{2s}\left[  (x_{N}^{(p)})^{n}D_{x}^{\widehat{\alpha}}v_{p}(x^{(p)}%
,t^{(p)})\right]  |}{h_{p}^{\gamma}}\geq\frac{\nu}{2}>0. \label{s1.54}%
\end{equation}
We make in the functions $\{v_{p}\}$ the change of the independent
variables $(x,t)\rightarrow(y,\tau)$%

\begin{equation}
x_{i}=x_{i}^{(p)}+y_{i}h_{p},i=\overline{1,N-1},x_{N}=y_{N}h_{p};\quad
t=t^{(p)}+h_{p}^{m-n}\tau\label{s1.55}%
\end{equation}
and denote%

\begin{equation}
w_{p}(y,\tau)=h_{p}^{-(m-n+(1-\omega)\gamma)}v_{p}(x^{\prime(p)}+y^{\prime
}h_{p},y_{N}h_{p},\tau h_{p}^{m-n}). \label{s1.56}%
\end{equation}
Taking into account that $\omega=n/m$, it can be checked directly
that the rescaled functions $w^{(p)}(y,\tau)$ satisfy%

\begin{equation}
\left\langle \left\langle w_{p}\right\rangle \right\rangle _{n,\omega
\gamma,\overline{Q},y,\tau}^{(m+\gamma)(2s)}\equiv%
{\displaystyle\sum\limits_{j=0}^{j\leq n}}
{\displaystyle\sum\limits_{|\alpha|=m-j}} \left\langle \left\langle
y_{N}^{n-j}D_{y}^{\alpha}w_{p}\right\rangle
\right\rangle _{\omega\gamma,y,\overline{Q}}^{(\gamma)(2s)}+%
\label{s1.56.0003}
\end{equation}

\[
+{\displaystyle\sum\limits_{j=0}^{j\leq n}}
{\displaystyle\sum\limits_{|\alpha|=m-j}} \left\langle \left\langle
y_{N}^{n-j}D_{y}^{\alpha}w_{p}\right\rangle
\right\rangle _{\tau,\overline{Q}}^{(\gamma/m)(2s)}+ %
\]

\[
+%
{\displaystyle\sum\limits_{j=0}^{j\leq m-n}}
{\displaystyle\sum\limits_{|\alpha|=[m-n+(1-\omega)\gamma]-j}}
\left\langle \left\langle D_{y^{\prime}}^{\alpha}D_{y_{N}}^{j}w_{p}%
\right\rangle \right\rangle _{y^{\prime},\overline{Q}}^{(\{m-n+(1-\omega
)\gamma\})(2s)}+%
\]

\[
+{\displaystyle\sum\limits_{j=1}^{j\leq m-n}}
{\displaystyle\sum\limits_{|\alpha|=j}} \left\langle \left\langle
D_{y}^{\alpha}w_{p}\right\rangle \right\rangle
_{\tau,\overline{Q}}^{(1-\frac{j}{m-n}+\frac{\gamma}{m})(2s)}+
\]

\[
+%
{\displaystyle\sum\limits_{j=0}^{j\leq m-n}}
{\displaystyle\sum\limits_{|\alpha|=[m-n+\gamma]-j}}
\left\langle \left\langle D_{y^{\prime}}^{\alpha}D_{y_{N}}^{j}w_{p}%
\right\rangle \right\rangle _{\omega\gamma,y^{\prime},\overline{Q}%
}^{(\{m-n+\gamma\})(2s)}+
\]

\[
+%
{\displaystyle\sum\limits_{j=0}^{j\leq n}}
{\displaystyle\sum\limits_{|\alpha|=m-j}} \left\langle \left\langle
y_{N}^{n-j\omega}D_{y}^{\alpha}w_{p}\right\rangle \right\rangle
_{\tau,\overline{Q}}^{(\frac{\gamma+j}{m})(2s)}+
\]

\[
+\left\langle \left\langle D_{\tau}w_{p}\right\rangle \right\rangle
_{\omega\gamma
,y,\overline{Q}}^{(\gamma)(4)}+\left\langle \left\langle D_{\tau}%
w_{p}\right\rangle \right\rangle _{\tau,\overline{Q}}^{(\gamma/m)(4)}%
=\left\langle \left\langle v_{p}\right\rangle \right\rangle
_{n,\omega \gamma,\overline{Q},x,t}^{(m+\gamma)(2s)}.
\]
And also (see \eqref{s1.7})

\begin{equation}
\left\langle w_{p}\right\rangle _{n,\omega\gamma,\overline{Q}}^{(m+\gamma
,\frac{m+\gamma}{m})}=\left\langle v_{p}\right\rangle _{n,\omega
\gamma,\overline{Q}}^{(m+\gamma,\frac{m+\gamma}{m})}. \label{s1.57}%
\end{equation}
Thus from the second inequality in \eqref{s1.50} and Proposition
\ref{Ps1.2} it follows that

\begin{equation}
\left\langle w_{p}\right\rangle _{n,\omega\gamma,\overline{Q}}^{(m+\gamma
,\frac{m+\gamma}{m})}\leq C\left\langle \left\langle w_{p}\right\rangle
\right\rangle _{n,\omega\gamma,\overline{Q},y,\tau}^{(m+\gamma)(2s)}\leq2C=C.
\label{s1.58}%
\end{equation}
From \eqref{s1.50} and \eqref{s1.57} we have

\begin{equation}%
{\displaystyle\sum\limits_{i=1}^{N}} \left\langle
y_{N}^{n}D_{y_{i}}^{m}w_{p}\right\rangle _{\omega\gamma
,y_{i},\overline{Q}}^{(\gamma)}+\left\langle D_{t}w_{p}\right\rangle
_{\tau,\overline{Q}}^{(\gamma/m)}\leq\frac{1}{p}. \label{s1.59}%
\end{equation}
And from \eqref{s1.54} we obtain%

\begin{equation}
\left(  y_{N}^{(p)}\right)  ^{\omega\gamma}|\Delta_{\overline{e}^{(p)}}%
^{2s}(y_{N}^{(p)})^{n}D_{y}^{\widehat{\alpha}}w_{p}(P^{(p)},0)|\geq\nu/2,
\label{s1.60}%
\end{equation}
where

\bigskip%
\begin{equation}
y_{N}^{(p)}\equiv x_{N}^{(p)}/h_{p},\, \overline{e}^{(p)}\equiv\overline
{h}^{(p)}/h_{p},\,|\overline{e}^{(p)}|=1,\,P^{(p)}\equiv(0^{\prime}%
,y_{N}^{(p)}). \label{s1.60.1}%
\end{equation}
Note also that the functions $w_{p}(y,\tau)$ inherit property \eqref{s1.43.01}%

\begin{equation}
y_{N}^{n-j}D_{x}^{\alpha}w_{p}(y,\tau)\rightarrow0,y_{N}\rightarrow
0,\quad0\leq j<n,\alpha=(\alpha_{1},...,\alpha_{N}),|\alpha|=m-j,\alpha
_{N}<m-j. \label{s1.61}%
\end{equation}
Denote by $Q_{p}(y,\tau)\equiv Q_{w_{p}}(y,\tau)$ the "Taylor"
function $Q_{w_{p}}(y,\tau)$ for the function $w_{p}(y,\tau)$, which
was constructed in Lemma \ref{Ls1.2} and denote $r_{p}(y,\tau)\equiv
w_{p}(y,\tau)-Q_{p}(y,\tau)$.  From Lemma \ref{Ls1.2}
it follows that%

\begin{equation}
y_{N}^{n-j}D_{y}^{\alpha}r_{p}(y,\tau)|_{(y,\tau)=(0,0)}=0,\quad
j<n,\,|\alpha|=m-j, \label{s1.62}%
\end{equation}

\begin{equation}
D_{y}^{\alpha}r_{p}(y,\tau)|_{(y,\tau)=(\overline{e},0)}=0,\,|\alpha|\leq
m-n,\hspace{0.05in}D_{\tau}r_{p}(y,\tau)|_{(y,\tau)=(\overline{e},0)}=0.
\label{s1.63}%
\end{equation}
Recall that

\begin{equation}
y_{N}^{n-j}D_{y}^{\alpha}Q_{p}(y,\tau)\equiv const,\quad|\alpha|=m-j,j\leq
n,\alpha_{N}<m-n,\, \,D_{\tau}Q_{p}(y,\tau)\equiv const, \label{s1.64}%
\end{equation}
and also

\begin{equation}
D_{y_{N}}^{m-n}Q_{p}(y,\tau)\equiv const\text{ \ if the seminorm }\left\langle
\left\langle D_{x_{N}}^{m-n}u\right\rangle \right\rangle _{\omega
\gamma,x,\overline{Q}}^{(\gamma)(2s)}<\infty\text{ is finite} \label{s1.65}%
\end{equation}
and it is included in the left hand side of \eqref{s1.41} and the
seminorm $\left\langle \left\langle D_{y_{N}}
^{m-n}w_{p}\right\rangle \right\rangle _{\omega\gamma,y,\overline{Q}}%
^{(\gamma)(2s)}$ is included in \eqref{s1.56.0003}. Consequently,
from \eqref{s1.64}, \eqref{s1.65} and from the definition of
H\"{o}lder classes in view of \eqref{s1.58} it follows that

\begin{equation}
\left\langle r_{p}\right\rangle _{n,\omega\gamma,\overline{Q}}^{(m+\gamma
,\frac{m+\gamma}{m})}=\left\langle w_{p}-Q_{p}(y,\tau)\right\rangle
_{n,\omega\gamma,\overline{Q}}^{(m+\gamma,\frac{m+\gamma}{m})}\leq C.
\label{s1.66}%
\end{equation}
For the same reason we have from \eqref{s1.59}

\begin{equation}%
{\displaystyle\sum\limits_{i=1}^{N}} \left\langle
y_{N}^{n}D_{y_{i}}^{m}r_{p}\right\rangle _{\omega\gamma
,y_{i},\overline{Q}}^{(\gamma)}+\left\langle
D_{\tau}r_{p}\right\rangle
_{\tau,\overline{Q}}^{(\gamma/m)}\leq\frac{1}{p} \label{s1.67}%
\end{equation}
and from \eqref{s1.60}

\begin{equation}
\left(  y_{N}^{(p)}\right)  ^{\omega\gamma}|\Delta_{\overline{e}^{(p)}}%
^{2s}(y_{N}^{(p)})^{n}D_{y}^{\widehat{\alpha}}r_{p}(P^{(p)},0)|\geq\nu.
\label{s1.68}%
\end{equation}
From \eqref{s1.61}, \eqref{s1.62}, \eqref{s1.63}, and \eqref{s1.66}
it follows that the sequence of functions $\{r_{p}(y,\tau)\}$ is
bounded in $C^{m+\gamma,\frac{m+\gamma}{m}}(K_{\delta})$ for any
compact set $K_{\delta}\subset Q\cap\{ \delta\leq
x_{N}\leq\delta^{-1}\}$, $\delta \in(0,1)$. Therefore there exists a
function $r(y,\tau)\in
C^{m+\gamma,\frac{m+\gamma}{m}}(Q\cap\{x_{N}>0\})$
with (at least for a subsequence)%

\begin{equation}
r_{p}\rightarrow r\text{ in }C^{m+\gamma^{\prime},\frac{m+\gamma^{\prime}}{m}%
}(K_{\delta}),\,p\rightarrow\infty,\, \forall K_{\delta}\subset Q\cap\{
\delta\leq y_{N}\leq\delta^{-1}\},\quad\gamma^{\prime}<\gamma. \label{s1.69}%
\end{equation}
At the same time, since the sequences $\{y_{N}^{(p)}\}$, $\{
\overline{e}^{(p)}\}$, and $\{P^{(p)}\}$ are bounded (recall
that $y_{N}^{(p)}=x_{N}^{(p)}/h_{p}\leq\varepsilon^{-1}$ since $h_{p}%
\geq\varepsilon x_{N}^{(p)}$)%

\begin{equation}
y_{N}^{(p)}\rightarrow y_{N}^{(0)},\quad\overline{e}^{(p)}\rightarrow
\overline{e}^{(0)},\quad P^{(p)}\rightarrow P^{(0)},\quad p\rightarrow\infty,
\label{s1.70}%
\end{equation}
where $y_{N}^{(0)}$ is a nonnegative number, $\overline{e}^{(0)}\in
\overline{H}$ is a unit vector, $P^{(0)}=(0^{\prime},y_{N}^{(0)})\in
\overline{H}$. From \eqref{s1.62} and \eqref{s1.66} (together with
\eqref{s1.01} and the Arzela theorem) it follows that the functions
$y_{N}^{n}D_{y}^{\widehat{\alpha}}r_{p}(y,\tau)$ are uniformly \
convergent (for a subsequence) on any compact set
$K_{R}\subset\overline{Q}\cap\{0\leq
y_{N}\leq R\}$, $R>0$,%

\[
y_{N}^{n}D_{y}^{\widehat{\alpha}}r_{p}(y,\tau)\rightrightarrows y_{N}^{n}%
D_{y}^{\widehat{\alpha}}r(y,\tau),\quad p\rightarrow\infty.
\]
Thus we can choose a compact set $K_{R}$ and take the limit of
relation \eqref{s1.68} on this set. This gives

\begin{equation}
|\Delta_{\overline{e}^{(0)}}^{2s}(y_{N}^{(0)})^{n}D_{y}^{\widehat{\alpha}%
}r(P^{(0)},0)|\geq\nu>0. \label{s1.71}%
\end{equation}
Moreover, from \eqref{s1.01} and \eqref{s1.66} it follows that uniformly in $p$%

\begin{equation}
\left\langle y_{N}^{n}D_{y}^{\widehat{\alpha}}r_{p}\right\rangle
_{y,\overline{Q}}^{(\gamma-\omega\gamma)}+\left\langle y_{N}^{n}%
D_{y}^{\widehat{\alpha}}r_{p}\right\rangle _{\tau,\overline{Q}}^{(\gamma
/m)}\leq C. \label{s1.71.1}%
\end{equation}
Together with \eqref{s1.62} this means that the sequence
$\{y_{N}^{n}D_{y}^{\widehat {\alpha}}r_{p}\}$ is bounded in the
space $C^{\gamma-\omega\gamma,\frac {\gamma}{m}}(K_{R})$ for any
compact set $K_{R}$. Therefore for any
$\gamma^{\prime}<\gamma$ the sequence $\{y_{N}^{n}D_{y}^{\widehat{\alpha}%
}r_{p}\}$ converges to $y_{N}^{n}D_{y}^{\widehat{\alpha}}r$ in
the space $C^{\gamma^{\prime}-\omega\gamma^{\prime},\frac{\gamma^{\prime}}{m}%
}(K_{R})$ and for the limit $y_{N}^{n}D_{y}^{\widehat{\alpha}}r$ we have with
the same exponent $\gamma$%

\begin{equation}
\left\langle y_{N}^{n}D_{y}^{\widehat{\alpha}}r\right\rangle _{y,\overline{Q}%
}^{(\gamma-\omega\gamma)}+\left\langle y_{N}^{n}D_{y}^{\widehat{\alpha}%
}r\right\rangle _{\tau,\overline{Q}}^{(\gamma/m)}\leq C. \label{s1.71.2}%
\end{equation}
Further, from \eqref{s1.67} it follows that

\begin{equation}%
\begin{array}
[c]{c}%
y_{N}^{n}D_{y_{N}}^{m}r(y,\tau)\text{ does not depend on }y_{N},\\
D_{y_{i}}^{m}r(y,\tau)\text{ does not depend on }y_{i},\quad i=\overline
{1,N-1},\\
D_{\tau}r(y,\tau)\text{ does not depend on }\tau.
\end{array}
\label{s1.72}%
\end{equation}
Really, let us prove the first assertion. Let
$y=(y^{\prime},y_{N})$, $y_{N}>0$, $\tau$, and $h>0$ be fixed. Then
we have directly from the definition and from \eqref{s1.67}

\[
y_{N}^{\omega\gamma}\frac{|(y_{N}+h)^{n}D_{y_{N}}^{m}r_{p}(y^{\prime}%
,y_{N}+h,\tau)-y_{N}{}^{n}D_{y_{N}}^{m}r_{p}(y^{\prime},y_{N},\tau
)|}{h^{\gamma}}\leq\left\langle y_{N}^{n}D_{y_{i}}^{m}r_{p}\right\rangle
_{\omega\gamma,y_{i},\overline{Q}}^{(\gamma)}\leq\frac{1}{p}.
\]
Making use of \eqref{s1.69} and taking limit in this inequality as
$p\rightarrow\infty$ we obtain

\[
(y_{N}+h)^{n}D_{y_{N}}^{m}r_{p}(y^{\prime},y_{N}+h,\tau)=y_{N}{}^{n}D_{y_{N}%
}^{m}r_{p}(y^{\prime},y_{N},\tau).
\]
Since $y$, $\tau$, and $h$ are arbitrary this proves the first
assertion in \eqref{s1.72}. Other assertions are completely
analogous.

Now from the first assertion in \eqref{s1.72} we have with some
functions \ $a(y^{\prime},\tau)$

\[
D_{y_{N}}^{m}r(y,\tau)=\frac{a(y^{\prime},\tau)}{y_{N}^{n}}.
\]
Integrating this equality in $y_{N}$, we find

\begin{equation}
r(y,\tau)=\left\{
\begin{array}
[c]{c}%
b_{0}(y^{\prime},\tau)y_{N}^{m-n}+%
{\displaystyle\sum\limits_{i=1}^{m}}
b_{i}(y^{\prime},\tau)y_{N}^{i-1},\quad n\text{ is a noninteger},\\
b_{0}(y^{\prime},\tau)\ln^{(m-n)}y_{N}+%
{\displaystyle\sum\limits_{i=1}^{m}}
b_{i}(y^{\prime},\tau)y_{N}^{i-1},\quad n\text{ is an integer},
\end{array}
\right.  \label{s1.73}%
\end{equation}
where $b_{0}(y^{\prime},\tau)$ and $b_{i}(y^{\prime},\tau)$ are some
functions and $\ln^{(m-n)}y_{N}$ is defined in \eqref{s1.009}.

Making use again of \eqref{s1.72} and taking into account the
independence of all terms in \eqref{s1.73}, we see

\[
D_{y_{1}}^{m+1}...D_{y_{N-1}}^{m+1}D_{t}^{2}b_{i}(y^{\prime},\tau
)\equiv0\text{ in }Q,\quad i=\overline{0,m}\quad-
\]
at least in the sense of distributions. This means, as it is well
known, that the functions $b_{i}(y^{\prime},\tau)$ are polynomials
in $y^{\prime}$ of degree not greater than $m$ and in $t$ of degree
not greater than $2$. Consequently, the function
$y_{N}^{n}D_{y}^{\widehat{\alpha}}r_{p}(y,\tau)$ has the form

\begin{equation}
y_{N}^{n}D_{y}^{\widehat{\alpha}}r(y,\tau)=P_{0}(y^{\prime},\tau)y_{N}^{n}%
\ln^{(k)}y_{N}+%
{\displaystyle\sum\limits_{j=1}^{m+1}}
P_{j}(y^{\prime},\tau)y_{N}^{d_{j}}, \label{s1.74}%
\end{equation}
where  $P_{0}(y^{\prime},\tau)$\ and\ $P_{j}(y^{\prime},\tau)$ are
some polynomials, $k$ is an integer, and for each
$j=\overline{1,m+1}$ $\ $the number \ $\ d_{j}$ is either an integer
or a number of the form $k_{j}+n$ with integer $k_{j}$. Now relation
\eqref{s1.71} means that the function
$y_{N}^{n}D_{y}^{\widehat{\alpha}}r_{p}(y,\tau)$ in \eqref{s1.74} is
not a constant identically. At last, as it can be checked directly,
a nonconstant function of the form \eqref{s1.74} can not have a
finite values of $\left\langle y_{N}^{n}
D_{y}^{\widehat{\alpha}}r\right\rangle _{\omega\gamma,y,\overline{Q}}%
^{(\gamma-\omega\gamma)}$ and $\left\langle y_{N}^{n}D_{y}^{\widehat{\alpha}%
}r\right\rangle _{\tau,\overline{Q}}^{(\gamma/m)}$ under our
assumption $\gamma-\omega\gamma<n$ over unbounded halfspace $Q$ (it
is enough to consider the term in \eqref{s1.74} with the maximal
growth at infinity). This contradict to \eqref{s1.71.2}.

This contradiction proves estimate \eqref{s1.47} on the class of
functions $u(x,t)$ with \eqref{s1.45.1}. Note again that all the
above reasonings for the term $\left\langle \left\langle
x_{N}^{n}D_{x}^{\widehat{\alpha}}v_{p}\right\rangle \right\rangle
_{\omega\gamma,x,\overline{Q}}^{(\gamma)(2s)(\varepsilon+)}$ from
\eqref{s1.51}  with \eqref{s1.52} are completely the same for other
terms in \eqref{s1.51}. For any other term in \eqref{s1.51} we
obtain an analog of relations \eqref{s1.71} and \eqref{s1.74} with
the same contradiction.

We now turn to the estimate of the value of $\left\langle
\left\langle u\right\rangle \right\rangle _{n,\omega\gamma,\overline{Q}%
}^{(m+\gamma)(2s)(\varepsilon-)}$ in \eqref{s1.45}. Our goal is to
obtain the estimate (compare \eqref{s1.30.1})

\begin{equation}
\left\langle \left\langle u\right\rangle \right\rangle _{n,\omega
\gamma,\overline{Q}}^{(m+\gamma)(2s)(\varepsilon-)}\leq
C_{\varepsilon}\left( {\displaystyle\sum\limits_{i=1}^{N}}
\left\langle x_{N}^{n}D_{x_{i}}^{m}u\right\rangle _{\omega\gamma
,x_{i},\overline{Q}}^{(\gamma)}+\left\langle D_{t}u\right\rangle
_{t,\overline{Q}}^{(\gamma/m)}\right)
+C\varepsilon^{\gamma}\left\langle
\left\langle u\right\rangle \right\rangle _{n,\omega\gamma,\overline{Q}%
}^{(m+\gamma)(2s)}. \label{s1.75}%
\end{equation}
All terms in the definition of $\left\langle \left\langle
u\right\rangle \right\rangle
_{n,\omega\gamma,\overline{Q}}^{(m+\gamma)(2s)(\varepsilon-)}$ in
\eqref{s1.45} are estimated completely similarly. We estimate
the most complex term with a degenerate factor%

\[
\left\langle \left\langle x_{N}^{n-j}D_{x}^{\alpha}u\right\rangle
\right\rangle
_{\omega\gamma,x,\overline{Q}}^{(\gamma)(2s)(\varepsilon-)}=
\]

\begin{equation}
=\sup_{(x,t)\in\overline{Q},\overline{h}\in\overline{H},|\overline{h}%
|\leq\varepsilon x_{N}}x_{N}^{\omega\gamma}\frac{|\Delta_{\overline{h},x}%
^{2s}\left(  x_{N}^{n-j}D_{x}^{\alpha}u(x,t)\right)  |}{|\overline{h}%
|^{\gamma}},\quad|\alpha|=m-j,j\leq n. \label{s1.76}%
\end{equation}
It can be checked directly that it is enough to consider the
H\"{o}lder property of $x_{N}^{n-j}D_{x}^{\alpha}u$ with respect to
the tangent variables $x^{\prime}$ and with respect to the variable
$x_{N}$ separately. This corresponds to the obtaining separately the
estimates for two cases of step $\overline{h}$ :
$\overline{h}=(\overline{h}^{\prime },0)=(h_{1},...,h_{N-1},0)$ and
$\overline{h}=(0,...,0,h)$, where $h>0$. We first obtain the
estimate \eqref{s1.75} with respect to the tangent variables, that
is we estimate the expression

\[
\left\langle \left\langle x_{N}^{n-j}D_{x}^{\alpha}u\right\rangle
\right\rangle _{\omega\gamma,x^{\prime},\overline{Q}}^{(\gamma
)(2s)(\varepsilon-)}\equiv\sup_{(x,t)\in\overline{Q},h^{\prime}\in
R^{N-1},|h^{\prime}|\leq\varepsilon x_{N}}x_{N}^{\omega\gamma}\frac
{|\Delta_{\overline{h}^{\prime},x^{\prime}}^{2s}\left(  x_{N}^{n-j}%
D_{x}^{\alpha}u(x,t)\right)  |}{|\overline{h}^{\prime}|^{\gamma}}=
\]

\begin{equation}
=\sup_{(x,t)\in\overline{Q},h^{\prime}\in R^{N-1},|h^{\prime}|\leq\varepsilon
x_{N}}x_{N}^{\omega\gamma}\frac{|x_{N}^{n-j}\Delta_{\overline{h}^{\prime
},x^{\prime}}^{s}D_{x}^{\alpha}v(x,t)|}{|\overline{h}^{\prime}|^{\gamma}},
\label{s1.77}%
\end{equation}
where $v(x,t)=\Delta_{h^{\prime},x^{\prime}}^{s}u(x,t)$. Let a point
$(x,t)=(x_{0},t_{0})=(x_{0}^{\prime},x_{N}^{0},t_{0})$ be fixed and
fix also a vector $\overline{h}^{\prime}\in R^{N-1}$,
$|\overline{h}^{\prime}|\leq\varepsilon x_{N}^{0}$. Suppose that
$\varepsilon\in(0,1/32m)$. Consider the expression

\begin{equation}
A\equiv\left(  x_{N}^{0}\right)  ^{\omega\gamma}\frac{|\left(  x_{N}%
^{0}\right)  ^{n-j}\Delta_{\overline{h}^{\prime},x^{\prime}}^{s}D_{x}^{\alpha
}v(x_{0},t_{0})|}{|\overline{h}^{\prime}|^{\gamma}},\quad|\alpha|=m-j,j\leq n.
\label{s1.78}%
\end{equation}
Make in the functions $u(x,t)$ \ and $v(x,t)$ the change of
variables $(x,t)\rightarrow(y,\tau)$, $v(x,t)\rightarrow v(y,\tau)$%

\begin{equation}
x^{\prime}=x_{0}^{\prime}+\left(  x_{N}^{0}\right)  y^{\prime},\quad
x_{N}=\left(  x_{N}^{0}\right)  y_{N},\quad t=t_{0}+\left(  x_{N}^{0}\right)
^{m-n}\tau\label{s1.79}%
\end{equation}
and denote $\overline{d}=\overline{h}^{\prime}/x_{N}^{0}$,
$\left\vert \overline{d}\right\vert \leq\varepsilon<1/32m$,
$P_{1}\equiv(y_{0},\tau _{0})\equiv(0^{\prime},1,0)$, that is
$(x_{0},t_{0})\rightarrow(y_{0},\tau _{0})$. In the new variables
the expression $A$ takes the form%

\begin{equation}
A=\left(  x_{N}^{0}\right)  ^{\omega\gamma+n-m-\gamma}\frac{|\Delta
_{\overline{d},y^{\prime}}^{s}D_{y}^{\alpha}v(0^{\prime},1,0)|}{|\overline
{d}|^{\gamma}}. \label{s1.80}%
\end{equation}
Denote for $\rho<1$%

\[
Q_{\rho}\equiv\{(y,\tau)\in Q:|y^{\prime}|\leq\rho,|y_{N}-1|\leq\rho
,|\tau|\leq\left(  \rho\right)  ^{m-n}\}
\]
and consider the function $v(y,\tau)$ on this cylinder. Note first,
that since $y_{N}\geq1/4$ on $Q_{3/4}$, the function $v(y,\tau)$
belongs to the usual smooth class $C^{m+\gamma,1+\gamma
/m}(\overline{Q}_{3/4})$. Considering this function on $\overline{Q}%
_{1/4}\subset\overline{Q}_{3/4}$ and applying \eqref{s1.12+2.1}, we
obtain

\[
\frac{|\Delta_{\overline{d},y^{\prime}}^{s}D_{y}^{\alpha}v(0^{\prime}%
,1,0)|}{|\overline{d}|^{\gamma}}\leq C\left\langle D_{y}^{\alpha}%
v(y,\tau)\right\rangle _{y^{\prime},\overline{Q}_{1/4}}^{(\gamma)}\leq
\]

\[
\leq C\left( {\displaystyle\sum\limits_{i=1}^{N}}
\left\langle D_{y_{i}}^{m}v\right\rangle _{y_{i},\overline{Q}_{1/4}}%
^{(\gamma)}+|v|_{\overline{Q}_{1/4}}^{(0)}\right) \leq
\]

\begin{equation}
 \leq C\left(
{\displaystyle\sum\limits_{i=1}^{N-1}}
\left\langle D_{y_{i}}^{m}u\right\rangle _{y_{i},\overline{Q}_{3/4}}%
^{(\gamma)}+\left\langle D_{y_{N}}^{m}v\right\rangle _{y_{N},\overline
{Q}_{1/4}}^{(\gamma)}+|v|_{\overline{Q}_{1/4}}^{(0)}\right)  . \label{s1.81}%
\end{equation}
Note that we drop the term $\left\langle D_{\tau}v\right\rangle
_{\tau,\overline{Q}_{1/4}}^{(\gamma/m)}$ in the right hand side of
\eqref{s1.81} because this estimate can be obtained at a fixed
$\tau$ with respect to the variables $y$ only. Now we go back to the
variables $(x,t)$ in the last estimate and obtain ($|\alpha|=m-j$)%

\begin{equation}
\left(  x_{N}^{0}\right)  ^{\gamma+m-j}\frac{|\Delta_{\overline{h}^{\prime
},x^{\prime}}^{s}D_{x}^{\alpha}v(x_{0},t_{0})|}{|\overline{h}^{\prime
}|^{\gamma}}\leq C\left(  \left(  x_{N}^{0}\right)  ^{m+\gamma}%
{\displaystyle\sum\limits_{i=1}^{N-1}}
\left\langle D_{x_{i}}^{m}u\right\rangle _{x_{i},\overline{Q}_{(3/4)x_{N}^{0}%
}}^{(\gamma)}+\right.  \label{s1.82}%
\end{equation}

\[
\left.  +\left(  x_{N}^{0}\right)  ^{m+\gamma}\left\langle D_{x_{N}}%
^{m}v\right\rangle _{x_{N},\overline{Q}_{(1/4)x_{N}^{0}}}^{(\gamma
)}+|v|_{\overline{Q}_{(1/4)x_{N}^{0}}}^{(0)}\right)  ,
\]
where for $\rho\in(0,1)$,

\[
Q_{\rho x_{N}^{0}}\equiv\{(x,t)\in Q:|x^{\prime}|\leq\rho x_{N}^{0}%
,|x_{N}-x_{N}^{0}|\leq\rho x_{N}^{0},|t-t_{0}|\leq\left(  \rho x_{N}%
^{0}\right)  ^{m-n}\}.
\]
Before proceeding further with the estimate of the expression $A$ in
\eqref{s1.78}, note that since $x_{N}\sim x_{N}^{0}$ on the set
$\overline {Q}_{(1/4)x_{N}^{0}}$ we have just from the definition of
the H\"{o}lder constants

\begin{equation}
\left(  x_{N}^{0}\right)  ^{n}\left\langle D_{x_{N}}^{m}v\right\rangle
_{x_{N},\overline{Q}_{(1/4)x_{N}^{0}}}^{(\gamma)}\leq C\left(  \left\langle
x_{N}^{n}D_{x_{N}}^{m}v\right\rangle _{x_{N},\overline{Q}_{(1/4)x_{N}^{0}}%
}^{(\gamma)}+\left(  x_{N}^{0}\right)  ^{n-\gamma}|D_{x_{N}}^{m}%
v|_{\overline{Q}_{(1/4)x_{N}^{0}}}^{(0)}\right)  . \label{s1.82.1}%
\end{equation}
Substituting this estimate in \eqref{s1.82}, dividing both parts of
obtained inequality by $\left( x_{N}^{0}\right)
^{\gamma+m-n-\omega\gamma}$ , and taking into account that
$v(x,t)=\Delta_{h^{\prime},x^{\prime}}^{s}u(x,t)$, we obtain%

\[
A\leq C\left(  \left(  x_{N}^{0}\right)  ^{n+\omega\gamma}%
{\displaystyle\sum\limits_{i=1}^{N-1}}
\left\langle D_{x_{i}}^{m}u\right\rangle _{x_{i},\overline{Q}_{(3/4)x_{N}^{0}%
}}^{(\gamma)}+\left(  x_{N}^{0}\right)  ^{\omega\gamma}\left\langle x_{N}%
^{n}D_{x_{N}}^{m}v\right\rangle _{x_{N},\overline{Q}_{(1/4)x_{N}^{0}}%
}^{(\gamma)}\right)  +
\]

\[
+C\left(  \left(  x_{N}^{0}\right)  ^{\omega\gamma-\gamma}|\Delta_{h^{\prime
},x^{\prime}}^{s}\left(  x_{N}^{n}D_{x_{N}}^{m}u\right)  |_{\overline
{Q}_{(1/4)x_{N}^{0}}}^{(0)}+\left(  x_{N}^{0}\right)  ^{-\gamma-m+n+\omega
\gamma}|\Delta_{h^{\prime},x^{\prime}}^{s}u(x,t)|_{\overline{Q}_{(1/4)x_{N}%
^{0}}}^{(0)}\right)  \leq
\]

\[
\leq C%
{\displaystyle\sum\limits_{i=1}^{N}} \left\langle
x_{N}^{n}D_{x_{i}}^{m}u\right\rangle _{\omega\gamma
,x_{i},\overline{Q}_{(3/4)x_{N}^{0}}}^{(\gamma)}+
\]

\[
+C\left(  \left(  x_{N}^{0}\right)  ^{-\gamma}|x_{N}^{\omega\gamma}%
\Delta_{h^{\prime},x^{\prime}}^{s}\left(
x_{N}^{n}D_{x_{N}}^{m}u\right)
|_{\overline{Q}_{(1/4)x_{N}^{0}}}^{(0)}+\right.
\]

\begin{equation}
+\left. \left(  x_{N}^{0}\right) ^{-(\gamma+m)}\left(
x_{N}^{0}\right) ^{\omega\gamma}|\Delta_{h^{\prime
},x^{\prime}}^{s}\left(  x_{N}^{n}u(x,t)\right)  |_{\overline{Q}%
_{(1/4)x_{N}^{0}}}^{(0)}\right)  . \label{s1.83}%
\end{equation}
At the same time for the last two terms in the right hand side of
\eqref{s1.83} we have

\[
\left(  x_{N}^{0}\right)
^{-\gamma}|x_{N}^{\omega\gamma}\Delta_{h^{\prime
},x^{\prime}}^{s}\left(  x_{N}^{n}D_{x_{N}}^{m}u\right) |_{\overline
{Q}_{(1/4)x_{N}^{0}}}^{(0)}=
\]

\[
=\left(  \frac{|h^{\prime}|}{x_{N}^{0}}\right) ^{\gamma}\left\{
\left(  x_{N}^{0}\right)  ^{\omega\gamma}\left\vert
\frac{\Delta_{h^{\prime},x^{\prime}}^{s}\left(  x_{N}^{n}D_{x_{N}}%
^{m}u\right)  }{|h^{\prime}|^{\gamma}}\right\vert _{\overline{Q}%
_{(1/4)x_{N}^{0}}}^{(0)}\right\}  \leq
\]

\[
\leq C\varepsilon^{\gamma}\left\langle x_{N}^{n}D_{x_{N}}^{m}u\right\rangle
_{\omega\gamma,x^{\prime},\overline{Q}}^{(\gamma)}\leq C\varepsilon^{\gamma
}\left\langle \left\langle u\right\rangle \right\rangle _{n,\omega
\gamma,\overline{Q}}^{(m+\gamma)(2s)},
\]

\[
\left(  x_{N}^{0}\right)  ^{-(\gamma+m)}|x_{N}^{\omega\gamma}\Delta
_{h^{\prime},x^{\prime}}^{s}\left(  x_{N}^{n}u(x,t)\right)
|_{\overline {Q}_{(1/4)x_{N}^{0}}}^{(0)}=
\]

\[
=\left(  \frac{|h^{\prime}|}{x_{N}^{0}}\right) ^{m+\gamma}\left\{
\left(  x_{N}^{0}\right)  ^{\omega\gamma}\left\vert
\frac{\Delta_{h^{\prime},x^{\prime}}^{s}\left(  x_{N}^{n}u\right)
}{|h^{\prime}|^{\gamma}}\right\vert _{\overline{Q}_{(1/4)x_{N}^{0}}}%
^{(0)}\right\}  \leq
\]

\[
\leq C\varepsilon^{m+\gamma}\left(
{\displaystyle\sum\limits_{|\beta|=m}} \left\langle
x_{N}^{n}D_{x^{\prime}}^{\beta}u\right\rangle _{\omega
\gamma,x^{\prime},\overline{Q}}^{(\gamma)}\right)  \leq
C\varepsilon^{\gamma }\left\langle \left\langle u\right\rangle
\right\rangle _{n,\omega \gamma,\overline{Q}}^{(m+\gamma)(2s)},
\]
where we made use of the mean value theorem and of Proposition
\ref{Ps1.2}. Substituting these two inequalities in \eqref{s1.83}
and taking into account the definition of the expression $A$ in
\eqref{s1.78}, we get

\[
\left\langle \left\langle x_{N}^{n-j}D_{x}^{\alpha}u\right\rangle
\right\rangle _{\omega\gamma,x^{\prime},\overline{Q}}^{(\gamma
)(2s)(\varepsilon-)}\leq
\]

\begin{equation}
\leq C_{\varepsilon}\left( {\displaystyle\sum\limits_{i=1}^{N}}
\left\langle x_{N}^{n}D_{x_{i}}^{m}u\right\rangle _{\omega\gamma
,x_{i},\overline{Q}}^{(\gamma)}+\left\langle D_{t}u\right\rangle
_{t,\overline{Q}}^{(\gamma/m)}\right)
+C\varepsilon^{\gamma}\left\langle
\left\langle u\right\rangle \right\rangle _{n,\omega\gamma,\overline{Q}%
}^{(m+\gamma)(2s)}. \label{s1.84}%
\end{equation}
We turn now to the obtaining the same estimate for $\left\langle
\left\langle x_{N}^{n-j}D_{x}^{\alpha}u\right\rangle \right\rangle
_{\omega\gamma,x_{N},\overline{Q}}^{(\gamma)(2s)(\varepsilon-)}$
with respect to the variable $x_{N}$, $|\alpha|=m-j$%

\[
\left\langle \left\langle x_{N}^{n-j}D_{x}^{\alpha}u\right\rangle
\right\rangle
_{\omega\gamma,x_{N},\overline{Q}}^{(\gamma)(2s)(\varepsilon -)}\leq
\]

\begin{equation}
\leq C_{\varepsilon}\left( {\displaystyle\sum\limits_{i=1}^{N}}
\left\langle x_{N}^{n}D_{x_{i}}^{m}u\right\rangle _{\omega\gamma
,x_{i},\overline{Q}}^{(\gamma)}+\left\langle D_{t}u\right\rangle
_{t,\overline{Q}}^{(\gamma/m)}\right)
+C\varepsilon^{\gamma}\left\langle
\left\langle u\right\rangle \right\rangle _{n,\omega\gamma,\overline{Q}%
}^{(m+\gamma)(2s)}. \label{s1.85}%
\end{equation}
We consider only the case $j<n$ because for an integer $n$ in the
case $j=n$ the function $x_{N}^{n-j}D_{x}^{\alpha}u=D_{x}^{\alpha}u$
has no a degeneration and all the estimates below are completely the
same and become simpler. The schema of the reasonings is quite
similar to the proof of \eqref{s1.84} above.

Let $Q_{u}(x,t)$ be the polynomial from \eqref{s1.0012} with the
properties \eqref{s1.0015}, \eqref{s1.0016} for the function
$u(x,t)$ under the consideration. If we consider the function
$v(x,t)=u(x,t)-Q_{u}(x,t)$ instead of the function $u(x,t)$ itself,
we see that all terms in both sides of \eqref{s1.85} remain
unchanged because of \eqref{s1.0015}- \eqref{s1.0017}. Therefore it
is enough to prove that

\begin{equation}
\left\langle \left\langle x_{N}^{n-j}D_{x}^{\alpha}v\right\rangle
\right\rangle
_{\omega\gamma,x_{N},\overline{Q}}^{(\gamma)(2s)(\varepsilon -)}\leq
C_{\varepsilon}\left( {\displaystyle\sum\limits_{i=1}^{N}}
\left\langle x_{N}^{n}D_{x_{i}}^{m}v\right\rangle _{\omega\gamma
,x_{i},\overline{Q}}^{(\gamma)}+\left\langle D_{t}v\right\rangle
_{t,\overline{Q}}^{(\gamma/m)}\right)
+C\varepsilon^{\gamma}\left\langle
\left\langle v\right\rangle \right\rangle _{n,\omega\gamma,\overline{Q}%
}^{(m+\gamma)(2s)}. \label{s1.86}%
\end{equation}
It is important for us that the function $v(x,t)$ possess the property%

\begin{equation}
x_{N}^{\omega\gamma}x_{N}^{n-j-\gamma}|D_{x}^{\alpha}v(x,t)|\leq\left\langle
x_{N}^{n-j}D_{x}^{\alpha}v\right\rangle _{\omega\gamma,x_{N},\overline{Q}%
}^{(\gamma)},\quad(x,t)\in\overline{Q}. \label{s1.87}%
\end{equation}
Really, from \eqref{s1.0012+1} it follows that
$x_{N}^{n-j}D_{x}^{\alpha}v(x,t)|_{x_{N} =0}=0$ and we obtain

\[
x_{N}^{\omega\gamma}x_{N}^{n-j-\gamma}|D_{x}^{\alpha}v(x,t)|=
\]

\[
=x_{N}%
^{\omega\gamma}\frac{|x_{N}^{n-j}D_{x}^{\alpha}v(x,t)-\left[  x_{N}^{n-j}%
D_{x}^{\alpha}v(x,t)\right]  |_{x_{N}=0}|}{x_{N}^{\gamma}}\leq\left\langle
x_{N}^{n-j}D_{x}^{\alpha}v\right\rangle _{\omega\gamma,x_{N},\overline{Q}%
}^{(\gamma)}.
\]
As above, let a point
$(x_{0},t_{0})=(x_{0}^{\prime},x_{N}^{0},t_{0})$ \ and
$0<h<\varepsilon x_{N}$\ be fixed, $0<\varepsilon<1/(32m)$. Consider the expression%

\begin{equation}
A\equiv\left(  x_{N}^{0}\right)  ^{\omega\gamma}\frac{|\Delta_{h,x_{N}}%
^{2s}\left[  \left(  x_{N}^{0}\right)  ^{n-j}D_{x}^{\alpha}v(x_{0}%
,t_{0})\right]  |}{h^{\gamma}}\equiv\frac{\left(  x_{N}^{0}\right)
^{\omega\gamma}}{h^{\gamma}}B. \label{s1.88}%
\end{equation}
We have

\[
B=\Delta_{h,x_{N}}^{2s}\left[  \left(  x_{N}^{0}\right)
^{n-j}D_{x}^{\alpha }v(x_{0},t_{0})\right] =
\]

\begin{equation}
 =%
{\displaystyle\sum\limits_{i=1}^{2s}}
C_{i}\Delta_{h,x_{N}}^{i}\left[  \left(  x_{N}^{0}+h_{\theta}\right)
^{n-j}\right]  \Delta_{h,x_{N}}^{2s-i}D_{x}^{\alpha}v(x_{0}^{\prime},x_{N}%
^{0}+h_{\theta},t_{0})+ \label{s1.89}%
\end{equation}

\[
+\left(  x_{N}^{0}+h_{\theta}\right)  ^{n-j}\Delta_{h,x_{N}}^{2s}D_{x}%
^{\alpha}v(x_{0}^{\prime},x_{N}^{0},t_{0})\equiv%
{\displaystyle\sum\limits_{i=1}^{2s}} B_{i}+B_{0},
\]
where $C_{i}$ are some constants, and by $h_{\theta}$  here and
below we denote all possible expressions of the form
$h_{\theta}=C\cdot h$ \ with $0\leq C\leq C(m)$.  Consider  $B_{i}$
with $i\geq1$. Making use of the mean value theorem to estimate
$\Delta_{h,x_{N}}^{i}\left[  \left(  x_{N}^{0}+h_{\theta}\right)
^{n-j}\right]  $ and keeping in mind the assumption
$h\leq\varepsilon x_{N}^{0}$, we have

\[
B_{i}\leq%
{\displaystyle\sum\limits_{h_{\theta}}}
C\left(  x_{N}^{0}+h_{\theta}\right)  ^{n-j-i}h^{i}|D_{x}^{\alpha}%
v(x_{0}^{\prime},x_{N}^{0}+h_{\theta},t_{0})|\leq
\]

\[
\leq\varepsilon^{i}%
{\displaystyle\sum\limits_{h_{\theta}}}
C\left(  x_{N}^{0}+h_{\theta}\right)  ^{n-j-\gamma}|D_{x}^{\alpha}%
v(x_{0}^{\prime},x_{N}^{0}+h_{\theta},t_{0})|h^{\gamma}.
\]
Therefore, in view of the definition of the expression $A$ in
\eqref{s1.88},

\[
A_{i}\equiv\left(  x_{N}^{0}\right)  ^{\omega\gamma}\frac{B_{i}}{h^{\gamma}%
}\leq
\]

\begin{equation}
\leq\varepsilon^{i}%
{\displaystyle\sum\limits_{h_{\theta}}}
C\left(  x_{N}^{0}+h_{\theta}\right)  ^{n-j-\gamma+\omega\gamma}|D_{x}%
^{\alpha}v(x_{0}^{\prime},x_{N}^{0}+h_{\theta},t_{0})|\leq C\varepsilon
^{i}\left\langle x_{N}^{n-j}D_{x}^{\alpha}v\right\rangle _{\omega\gamma
,x_{N},\overline{Q}}^{(\gamma)}. \label{s1.90}%
\end{equation}
Consider now the expression $B_{0}$ in \eqref{s1.89}. The
considerations in this case are similar to the previous case of the
variables $x^{\prime}$. Denote

\[
w(x,t)\equiv\Delta_{h,x_{N}}^{s}v(x^{\prime},x_{N},t)\text{ so that }%
\Delta_{h,x_{N}}^{2s}D_{x}^{\alpha}v(x_{0}^{\prime},x_{N}^{0},t_{0}%
)=\Delta_{h,x_{N}}^{s}D_{x}^{\alpha}w(x_{0}^{\prime},x_{N}^{0},t_{0})\text{ }%
\]
and consider the expression%

\[
A_{0}\equiv\frac{\left(  x_{N}^{0}\right)  ^{\omega\gamma}}{h^{\gamma}}B_{0}.
\]
As above, make in the functions $v(x,t)$  and $w(x,t)$ the change of
variables \eqref{s1.79} and denote $d=h/x_{N}^{0}$,
$d\leq\varepsilon<1/32m$, $d_{\theta}=h_{\theta }/x_{N}^{0}\leq
C(m)\varepsilon$,\ $P_{1}\equiv(y_{0},\tau_{0})\equiv
(0^{\prime},1,0)$, that is
$(x_{0},t_{0})\rightarrow(y_{0},\tau_{0})$. In the
new variables the expression $A_{0}$ takes the form%

\[
A_{0}=\left(  x_{N}^{0}\right)  ^{\omega\gamma+n-m-\gamma}(1+d_{\theta}%
)^{n-j}\frac{|\Delta_{d,y_{N}}^{s}D_{y}^{\alpha}w(0^{\prime},1,0)|}{d^{\gamma
}}\leq
\]

\begin{equation}
\leq C\left(  x_{N}^{0}\right)  ^{\omega\gamma+n-m-\gamma}\frac
{|\Delta_{d,y_{N}}^{s}D_{y}^{\alpha}w(0^{\prime},1,0)|}{d^{\gamma}}.
\label{s1.91}%
\end{equation}
Denote for $\rho<1$%

\[
Q_{\rho}\equiv\{(y,\tau)\in Q:|y^{\prime}|\leq\rho,|y_{N}-1|\leq\rho
,|\tau|\leq\left(  \rho\right)  ^{m-n}\}
\]
and consider the function $w(y,\tau)$ on this cylinder. As above,
since $y_{N}\geq1/4$ on $Q_{3/4}$, the function $w(y,\tau)$ belongs
to the usual smooth class $C^{m+\gamma,1+\gamma
/m}(\overline{Q}_{3/4})$. \ Considering, as above, this function on
$\overline{Q}_{1/4}\subset\overline{Q}_{3/4}$ and applying
\eqref{s1.12+2.1}, we obtain

\[
\frac{|\Delta_{d,y_{N}}^{s}D_{y}^{\alpha}w(0^{\prime},1,0)|}{d^{\gamma}}\leq
C\left\langle D_{y}^{\alpha}w(y,\tau)\right\rangle _{y_{N},\overline{Q}_{1/4}%
}^{(\gamma)}\leq
\]

\[
\leq C\left( {\displaystyle\sum\limits_{i=1}^{N}}
\left\langle D_{y_{i}}^{m}w\right\rangle _{y_{i},\overline{Q}_{1/4}}%
^{(\gamma)}+|w|_{\overline{Q}_{1/4}}^{(0)}\right) \leq
\]

\begin{equation}
 \leq C\left(
{\displaystyle\sum\limits_{i=1}^{N-1}}
\left\langle D_{y_{i}}^{m}v\right\rangle _{y_{i},\overline{Q}_{3/4}}%
^{(\gamma)}+\left\langle D_{y_{N}}^{m}v\right\rangle _{y_{N},\overline
{Q}_{3/4}}^{(\gamma)}+|w|_{\overline{Q}_{1/4}}^{(0)}\right)  . \label{s1.92}%
\end{equation}
Note that we again drop the term $\left\langle D_{\tau
}v\right\rangle _{\tau,\overline{Q}_{1/4}}^{(\gamma/m)}$ in the
right hand side of \eqref{s1.92} because this estimate can be
obtained at a fixed $\tau$ with respect to the variables $y$ only.
Now we go back to the variables $(x,t)$ in the last estimate and
obtain ($|\alpha|=m-j$)

\[
\frac{|\Delta_{d,y_{N}}^{s}D_{y}^{\alpha}w(0^{\prime},1,0)|}{d^{\gamma}}%
\leq\left(  x_{N}^{0}\right)  ^{\gamma+m-j}\frac{|\Delta_{h,x_{N}}^{s}%
D_{x}^{\alpha}w(x_{0},t_{0})|}{h^{\gamma}}\leq
\]

\begin{equation}
\leq C\left(  \left(  x_{N}%
^{0}\right)  ^{m+\gamma}%
{\displaystyle\sum\limits_{i=1}^{N-1}}
\left\langle D_{x_{i}}^{m}v\right\rangle _{x_{i},\overline{Q}_{(3/4)x_{N}^{0}%
}}^{(\gamma)}+\right.  \label{s1.93}%
\end{equation}

\[
\left.  +\left(  x_{N}^{0}\right)  ^{m+\gamma}\left\langle D_{x_{N}}%
^{m}v\right\rangle _{x_{N},\overline{Q}_{(3/4)x_{N}^{0}}}^{(\gamma
)}+|w|_{\overline{Q}_{(1/4)x_{N}^{0}}}^{(0)}\right)  ,
\]
where again for $\rho\in(0,1)$

\[
Q_{\rho x_{N}^{0}}\equiv\{(x,t)\in Q:|x^{\prime}|\leq\rho x_{N}^{0}%
,|x_{N}-x_{N}^{0}|\leq\rho x_{N}^{0},|t-t_{0}|\leq\left(  \rho x_{N}%
^{0}\right)  ^{m-n}\}.
\]
Substituting this estimate in \eqref{s1.91}, we obtain%

\[
A_{0}\leq C\left(  \left(  x_{N}^{0}\right)  ^{\omega\gamma+n}%
{\displaystyle\sum\limits_{i=1}^{N-1}}
\left\langle D_{x_{i}}^{m}v\right\rangle _{x_{i},\overline{Q}_{(3/4)x_{N}^{0}%
}}^{(\gamma)}+\right.
\]

\begin{equation}
\left.+\left(  x_{N}^{0}\right)  ^{\omega\gamma+n}\left\langle
D_{x_{N}}^{m}v\right\rangle
_{x_{N},\overline{Q}_{(3/4)x_{N}^{0}}}^{(\gamma )}+\left(
x_{N}^{0}\right)  ^{\omega\gamma+n-m-\gamma}|w|_{\overline
{Q}_{(1/4)x_{N}^{0}}}^{(0)}\right)  . \label{s1.94}%
\end{equation}
Again, since $x_{N}\sim x_{N}^{0}$ on the set
$\overline{Q}_{(3/4)x_{N}^{0}}$, for $i=\overline{1,N-1}$ in the
first term of \eqref{s1.94}

\begin{equation}
\left(  x_{N}^{0}\right)  ^{\omega\gamma+n}\left\langle D_{x_{i}}%
^{m}v\right\rangle _{x_{i},\overline{Q}_{(3/4)x_{N}^{0}}}^{(\gamma)}\leq
C\left(  x_{N}^{0}\right)  ^{\omega\gamma}\left\langle x_{N}^{n}D_{x_{i}}%
^{m}v\right\rangle _{x_{i},\overline{Q}_{(3/4)x_{N}^{0}}}^{(\gamma)}\leq
C\left\langle x_{N}^{n}D_{x_{i}}^{m}v\right\rangle _{\omega\gamma
,x_{i},\overline{Q}}^{(\gamma)}. \label{s1.95}%
\end{equation}
Making further use of \eqref{s1.82.1} and then \eqref{s1.87} we have
on $\overline{Q}_{(3/4)x_{N}^{0}}$ for the second term

\[
\left(  x_{N}^{0}\right)  ^{\omega\gamma+n}\left\langle D_{x_{N}}%
^{m}v\right\rangle _{x_{N},\overline{Q}_{(3/4)x_{N}^{0}}}^{(\gamma)}\leq
C\left(  x_{N}^{0}\right)  ^{\omega\gamma}\left(  x_{N}^{n}\left\langle
D_{x_{N}}^{m}v\right\rangle _{x_{N},\overline{Q}_{(3/4)x_{N}^{0}}}^{(\gamma
)}\right)  \leq
\]

\[
\leq C\left(  x_{N}^{0}\right)  ^{\omega\gamma}\left(  \left\langle x_{N}%
^{n}D_{x_{N}}^{m}v\right\rangle _{x_{N},\overline{Q}_{(3/4)x_{N}^{0}}%
}^{(\gamma)}+\left\vert x_{N}^{n-\gamma}|D_{x_{N}}^{m}v|\right\vert
_{\overline{Q}_{(3/4)x_{N}^{0}}}^{(0)}\right)  \leq
\]

\begin{equation}
\leq C\left(  x_{N}^{0}\right)  ^{\omega\gamma}\left\langle x_{N}^{n}D_{x_{N}%
}^{m}v\right\rangle _{x_{N},\overline{Q}_{(3/4)x_{N}^{0}}}^{(\gamma)}\leq
C\left\langle x_{N}^{n}D_{x_{i}}^{m}v\right\rangle _{\omega\gamma
,x_{N},\overline{Q}}^{(\gamma)}. \label{s1.96}%
\end{equation}
The third term in \eqref{s1.94} we estimate as follows ($s=m+1$)

\[
\left(  x_{N}^{0}\right)  ^{\omega\gamma+n-m-\gamma}|w|_{\overline
{Q}_{(1/4)x_{N}^{0}}}^{(0)}=\left(  \frac{h}{x_{N}^{0}}\right)  ^{m+\gamma
}\left\vert \left(  x_{N}^{0}\right)  ^{\omega\gamma+n}\frac{\Delta_{h,x_{N}%
}^{s}v(x^{\prime},x_{N},t)}{h^{m+\gamma}}\right\vert _{\overline
{Q}_{(1/4)x_{N}^{0}}}^{(0)}\leq
\]

\begin{equation}
\leq C\varepsilon^{m+\gamma}\left\vert \left(  x_{N}^{0}\right)
^{\omega\gamma+n}\frac{\Delta_{h,x_{N}}D_{x_{N}}^{m}v(x^{\prime}%
,x_{N}+h_{\theta},t)}{h^{\gamma}}\right\vert _{\overline{Q}_{(1/4)x_{N}^{0}}%
}^{(0)}, \label{s1.97}%
\end{equation}
where the mean value theorem was used. Making again use of
\eqref{s1.82.1} and \eqref{s1.87} we have on
$\overline{Q}_{(3/4)x_{N}^{0}}$ as above

\begin{equation}
\left\vert \left(  x_{N}^{0}\right)  ^{\omega\gamma+n}\frac{\Delta_{h,x_{N}%
}D_{x_{N}}^{m}v(x^{\prime},x_{N}+h_{\theta},t)}{h^{\gamma}}\right\vert
_{\overline{Q}_{(1/4)x_{N}^{0}}}^{(0)}\leq\left(  x_{N}^{0}\right)
^{\omega\gamma+n}\left\langle D_{x_{N}}^{m}v\right\rangle _{x_{N},\overline
{Q}_{(3/4)x_{N}^{0}}}^{(\gamma)}\leq\label{s1.98}%
\end{equation}

\[
\leq C\left(  x_{N}^{0}\right)  ^{\omega\gamma}\left\langle x_{N}^{n}D_{x_{N}%
}^{m}v\right\rangle _{x_{N},\overline{Q}_{(3/4)x_{N}^{0}}}^{(\gamma)}\leq
C\left\langle x_{N}^{n}D_{x_{N}}^{m}v\right\rangle _{\omega\gamma
,x_{N},\overline{Q}}^{(\gamma)}.
\]
From \eqref{s1.95}- \eqref{s1.98} it follows that

\[
A_{0}\leq C%
{\displaystyle\sum\limits_{i=1}^{N}} \left\langle
x_{N}^{n}D_{x_{i}}^{m}v\right\rangle _{\omega\gamma
,x_{i},\overline{Q}}^{(\gamma)}%
\]
and from this and from \eqref{s1.90}, \eqref{s1.88} it follows that

\[
\left(  x_{N}^{0}\right)
^{\omega\gamma}\frac{|\Delta_{h,x_{N}}^{2s}\left[ \left(
x_{N}^{0}\right)  ^{n-j}D_{x}^{\alpha}v(x_{0},t_{0})\right]
|}{h^{\gamma}}\leq
\]

\begin{equation}
\leq C%
{\displaystyle\sum\limits_{i=1}^{N}} \left\langle
x_{N}^{n}D_{x_{i}}^{m}v\right\rangle _{\omega\gamma
,x_{i},\overline{Q}}^{(\gamma)}+C\varepsilon\left\langle
\left\langle x_{N}^{n-j}D_{x}^{\alpha}v\right\rangle \right\rangle
_{\omega\gamma
,x_{N},\overline{Q}}^{(\gamma)}. \label{s1.99}%
\end{equation}
Since $(x_{0},t_{0})$ and $h\leq\varepsilon x_{N}$ are arbitrary,
this means \eqref{s1.86} and therefore \eqref{s1.85}.

Other terms in the definition of $\left\langle \left\langle
u\right\rangle \right\rangle
_{n,\omega\gamma,\overline{Q}}^{(m+\gamma)(2s)(\varepsilon-)}$ in
\eqref{s1.45} are estimated completely similarly. The smoothness
with respect to the $t$ - variable is estimated identically to the
estimates of the smoothness with respect to $x^{\prime}$ with the
taking into account the relation between the dimensions of $x$ and
$t$ : $x\sim x_{N}^{0}$, $t\sim\left(  x_{N}^{0}\right)  ^{m-n}$ or
$x\sim h$, $t\sim h^{m-n}$. \ Note that all terms in the definition
of $\left\langle \left\langle u\right\rangle \right\rangle
_{n,\omega\gamma,\overline{Q}}^{(m+\gamma)(2s)(\varepsilon-)}$ have
the same total dimension with respect to this relation of the
dimensions. Now from the alternative \eqref{s1.46}  and from
\eqref{s1.47} with \eqref{s1.75} it follows that for
arbitrary function $u(x,t)\in C_{n,\omega\gamma}^{m+\gamma,\frac{m+\gamma}{m}%
}(\overline{Q})$ we have

\[
\left\langle \left\langle u\right\rangle \right\rangle _{n,\omega
\gamma,\overline{Q}}^{(m+\gamma)(2s)}\leq C_{\varepsilon}\left(
{\displaystyle\sum\limits_{i=1}^{N}} \left\langle
x_{N}^{n}D_{x_{i}}^{m}u\right\rangle _{\omega\gamma
,x_{i},\overline{Q}}^{(\gamma)}+\left\langle D_{t}u\right\rangle
_{t,\overline{Q}}^{(\gamma/m)}\right)
+C\varepsilon^{\gamma}\left\langle
\left\langle u\right\rangle \right\rangle _{n,\omega\gamma,\overline{Q}%
}^{(m+\gamma)(2s)}.
\]

Finally, choosing $\varepsilon$ in this estimate sufficiently small
and absorbing the last term in the left hand side, we arrive at
\eqref{s1.41}. This completes the proof of Theorem \ref{Ts1.1} under
assumption \eqref{s1.43.01}.

We now remove this assumption. Let $u(x,t)\in
C_{n,\omega\gamma}^{m+\gamma ,\frac{m+\gamma}{m}}(\overline{Q})$ and
let $u_{\varepsilon}(x,t)$ be defined by \eqref{s1.005} so it
satisfies \eqref{s1.006}, \eqref{s1.43.01}. By what was proved above%

\[
\left\langle u_{\varepsilon}\right\rangle _{n,\omega\gamma,\overline{Q}%
}^{(m+\gamma)}\leq C\left( {\displaystyle\sum\limits_{i=1}^{N}}
\left\langle x_{N}^{n}D_{x_{i}}^{m}u_{\varepsilon}\right\rangle
_{\omega \gamma,x_{i},\overline{Q}}^{(\gamma)}+\left\langle
D_{t}u_{\varepsilon }\right\rangle
_{t,\overline{Q}}^{(\gamma/m)}\right)  ,
\]
where the constant $C$ does not depend on $\varepsilon$.  It can be
directly verified that

\[
\left\langle x_{N}^{n}D_{x_{i}}^{m}u_{\varepsilon}\right\rangle _{\omega
\gamma,x_{i},\overline{Q}}^{(\gamma)}\leq C\left\langle x_{N}^{n}D_{x_{i}}%
^{m}u\right\rangle _{\omega\gamma,x_{i},\overline{Q}}^{(\gamma)}%
,\quad\left\langle D_{t}u_{\varepsilon}\right\rangle _{t,\overline{Q}%
}^{(\gamma/m)}\leq\left\langle D_{t}u\right\rangle _{t,\overline{Q}}%
^{(\gamma/m)}%
\]
and therefore uniformly in $\varepsilon$%

\begin{equation}
\left\langle u_{\varepsilon}\right\rangle _{n,\omega\gamma,\overline{Q}%
}^{(m+\gamma)}\leq C\left( {\displaystyle\sum\limits_{i=1}^{N}}
\left\langle x_{N}^{n}D_{x_{i}}^{m}u\right\rangle _{\omega\gamma
,x_{i},\overline{Q}}^{(\gamma)}+\left\langle D_{t}u\right\rangle
_{t,\overline{Q}}^{(\gamma/m)}\right)  . \label{s1.100}%
\end{equation}
Let $\delta\in(0,1)$ and let $Q_{\delta}=\{(x,t):\delta\leq x_{N}%
\leq\delta^{-1},|x^{\prime}|\leq\delta^{-1},|t|\leq\delta^{-1}\}$.
It is well known that for each $\delta\in(0,1)$ the sequence
$\{u_{\varepsilon}\}$ is bounded in the standard space
$C^{m+\gamma,\frac{m+\gamma}{m}}(Q_{\delta})$. Therefore for a subsequence%

\begin{equation}
u_{\varepsilon}\rightrightarrows u\text{ on each }Q_{\delta}\text{ in the
space }C^{m+\gamma^{\prime},\frac{m+\gamma^{\prime}}{m}}(Q_{\delta})
\label{s1.101}%
\end{equation}
with $\gamma^{\prime}<\gamma$. Besides, all weighted terms $x_{N}^{n-j}%
D_{x}^{\alpha}u_{\varepsilon}$ and
$x_{N}^{n-j\omega}D_{x}^{\alpha}u_{\varepsilon}$\ in the definition
of $\left\langle u_{\varepsilon}\right\rangle _{n,\omega\gamma,\overline{Q}%
}^{(m+\gamma)}$ , $j\leq n$, $|\alpha|=m-j$, converge to
$x_{N}^{n-j} D_{x}^{\alpha}u$ and $x_{N}^{n-j\omega}D_{x}^{\alpha}u$
in the space
$C^{(1-\omega)\gamma^{\prime},\frac{\gamma^{\prime}}{m}}(K_{\delta})$
for each $K_{\delta}=\{(x,t):0\leq
x_{N}\leq\delta^{-1},|x^{\prime}|\leq\delta
^{-1},|t|\leq\delta^{-1}\}$

\begin{equation}
x_{N}^{n-j}D_{x}^{\alpha}u_{\varepsilon}\rightarrow_{C^{(1-\omega
)\gamma^{\prime},\frac{\gamma^{\prime}}{m}}(K_{\delta})}x_{N}^{n-j}%
D_{x}^{\alpha}u,\quad x_{N}^{n-j\omega}D_{x}^{\alpha}u_{\varepsilon
}\rightarrow_{C^{(1-\omega)\gamma^{\prime},\frac{\gamma^{\prime}}{m}%
}(K_{\delta})}x_{N}^{n-j\omega}D_{x}^{\alpha}u. \label{s1.101.1}%
\end{equation}
And the same is valid for the term $D_{t}u_{\varepsilon}$. Exactly
the same reasonings as in the proof of Proposition \ref{Ps1.04} show
that

\[%
{\displaystyle\sum\limits_{j=0}^{j\leq n}}
{\displaystyle\sum\limits_{|\alpha|=m-j}} \left\langle
x_{N}^{n-j}D_{x}^{\alpha}u\right\rangle _{\omega\gamma
,\overline{Q}}^{(\gamma,\gamma/m)}+%
{\displaystyle\sum\limits_{j=0}^{j\leq n}}
{\displaystyle\sum\limits_{|\alpha|=m-j}}
\left\langle x_{N}^{n-j\omega}D_{x}^{\alpha}u\right\rangle _{t,\overline{Q}%
}^{(\frac{\gamma+j}{m})}+\left\langle D_{t}u\right\rangle _{\omega
\gamma,\overline{Q}}^{(\gamma,\gamma/m)}\leq
\]

\[
\leq C\limsup_{\varepsilon\rightarrow0}\left(
{\displaystyle\sum\limits_{j=0}^{j\leq n}}
{\displaystyle\sum\limits_{|\alpha|=m-j}} \left\langle
x_{N}^{n-j}D_{x}^{\alpha}u_{\varepsilon}\right\rangle
_{\omega\gamma,\overline{Q}}^{(\gamma,\gamma/m)}+\right.
\]

\[
\left. +{\displaystyle\sum\limits_{j=0}^{j\leq n}}
{\displaystyle\sum\limits_{|\alpha|=m-j}} \left\langle
x_{N}^{n-j\omega}D_{x}^{\alpha}u_{\varepsilon}\right\rangle
_{t,\overline{Q}}^{(\frac{\gamma+j}{m})}+\left\langle
D_{t}u_{\varepsilon }\right\rangle
_{\omega\gamma,\overline{Q}}^{(\gamma,\gamma/m)}\right)  \leq
\]

\[
\leq C\left( {\displaystyle\sum\limits_{i=1}^{N}} \left\langle
x_{N}^{n}D_{x_{i}}^{m}u\right\rangle _{\omega\gamma
,x_{i},\overline{Q}}^{(\gamma)}+\left\langle D_{t}u\right\rangle
_{t,\overline{Q}}^{(\gamma/m)}\right)  .
\]
Note that this is valid for the pure derivative $D_{x_{N}}^{m-n}u$
in the case of an integer $n$ only if this derivative and it's
seminorms are included in the definition of the space. Besides,
since the functions $u_{\varepsilon}$ satisfy \eqref{s1.006},
\eqref{s1.43.01}, it follows from \eqref{s1.101.1} that the function
$u(x,t)$ itself satisfies \eqref{s1.7.1}. The same reasoning holds
true also for the term ${\displaystyle\sum\limits_{j=1}^{j\leq m-n}}
{\displaystyle\sum\limits_{|\alpha|=j}}
\left\langle D_{x}^{\alpha}u_{\varepsilon}\right\rangle _{t,\overline{Q}%
}^{(1-\frac{j}{m-n}+\frac{\gamma}{m})}$ in the definition of
$\left\langle u_{\varepsilon}\right\rangle
_{n,\omega\gamma,\overline{Q}}^{(m+\gamma)}$ and therefore

\[
{\displaystyle\sum\limits_{j=1}^{j\leq m-n}}
{\displaystyle\sum\limits_{|\alpha|=j}}
\left\langle D_{x}^{\alpha}u\right\rangle _{t,\overline{Q}}^{(1-\frac{j}%
{m-n}+\frac{\gamma}{m})}\leq C\limsup_{\varepsilon\rightarrow0}%
{\displaystyle\sum\limits_{j=1}^{j\leq m-n}}
{\displaystyle\sum\limits_{|\alpha|=j}}
\left\langle D_{x}^{\alpha}u_{\varepsilon}\right\rangle _{t,\overline{Q}%
}^{(1-\frac{j}{m-n}+\frac{\gamma}{m})}\leq
\]

\[
\leq C\left( {\displaystyle\sum\limits_{i=1}^{N}} \left\langle
x_{N}^{n}D_{x_{i}}^{m}u\right\rangle _{\omega\gamma
,x_{i},\overline{Q}}^{(\gamma)}+\left\langle D_{t}u\right\rangle
_{t,\overline{Q}}^{(\gamma/m)}\right)  .
\]
Note again that this is valid for the pure derivative
$D_{x_{N}}^{m-n}u$ in the case of an integer $n$ only if this
derivative and it's seminorms are included in the definition of the
space.

At last, since $u_{\varepsilon}(x,t)$ is smooth with respect to
$x^{\prime}$ and $t$ for any $x_{N}>0$, we have in the open domain
$Q=\overline{Q}\cap\{x_{N}>0\}$ by the same reasonins as above
uniformly in $x_{N}$

\[
{\displaystyle\sum\limits_{j=0}^{j\leq m-n}}
{\displaystyle\sum\limits_{|\alpha|=[m-n+(1-\omega)\gamma]-j}}
\left\langle D_{x^{\prime}}^{\alpha}D_{x_{N}}^{j}u\right\rangle
_{x^{\prime
},Q}^{(\{m-n+(1-\omega)\gamma\})}+%
\]

\[
+{\displaystyle\sum\limits_{j=0}^{j\leq m-n}}
{\displaystyle\sum\limits_{|\alpha|=[m-n+\gamma]-j}} \left\langle
D_{x^{\prime}}^{\alpha}D_{x_{N}}^{j}u\right\rangle _{\omega
\gamma,x^{\prime},Q}^{(\{m-n+\gamma\})}\leq
\]

\[
\leq C\limsup_{\varepsilon\rightarrow0}\left(
{\displaystyle\sum\limits_{j=0}^{j\leq m-n}}
{\displaystyle\sum\limits_{|\alpha|=[m-n+(1-\omega)\gamma]-j}}
\left\langle
D_{x^{\prime}}^{\alpha}D_{x_{N}}^{j}u_{\varepsilon}\right\rangle
_{x^{\prime},\overline{Q}}^{(\{m-n+(1-\omega)\gamma\})}+\right.%
\]

\[
\left.+{\displaystyle\sum\limits_{j=0}^{j\leq m-n}}
{\displaystyle\sum\limits_{|\alpha|=[m-n+\gamma]-j}} \left\langle
D_{x^{\prime}}^{\alpha}D_{x_{N}}^{j}u_{\varepsilon}\right\rangle
_{\omega\gamma,x^{\prime},Q}^{(\{m-n+\gamma\})}\right)  \leq
\]

\begin{equation}
\leq C\left( {\displaystyle\sum\limits_{i=1}^{N}} \left\langle
x_{N}^{n}D_{x_{i}}^{m}u\right\rangle _{\omega\gamma
,x_{i},\overline{Q}}^{(\gamma)}+\left\langle D_{t}u\right\rangle
_{t,\overline{Q}}^{(\gamma/m)}\right)  . \label{s1.102}%
\end{equation}
But from estimates \eqref{s1.007.1} it follows that the functions
$D_{x_{N}}^{j}u$, $j<m-n$,\ are continuous at $x_{N}\rightarrow0$.
Then from the above estimate it follows that at a fixed $t_{0}>0$
the sequence of the functions
$\{D_{x_{N}}^{j}u(\cdot,x_{N},t_{0})\}$ with $x_{N}$ as a parameter
is bounded in $C^{m-n+(1-\omega)\gamma-j}(R^{N-1})$. Therefore, in
view of Proposition \ref{Ps1.04} this sequence converges (at least
for a subsequence) on compact sets $K$ in $R^{N-1}$ as
$x_{N}\rightarrow0$ in the space
$C^{m-n+(1-\omega)\gamma^{\prime}-j}(K)$, $\gamma^{\prime}<\gamma$,\
to some function $v(x^{\prime},t_{0})\in$
$C^{m-n+(1-\omega)\gamma-j}(R^{N-1})$ with the same estimate of the
highest seminorm $\left\langle D_{x^{\prime}}^{\alpha}v\right\rangle
_{x^{\prime},R^{N-1}}^{(\{m-n+(1-\omega)\gamma\})}$. This means that
at a fixed $t_{0}>0$ the functions $D_{x_{N}}^{j}u$, $j<m-n$, have
traces at $x_{N}=0$ from the space
$C^{m-n+(1-\omega)\gamma-j}(R^{N-1})$ and estimate \eqref{s1.102} is
valid also at $x_{N}=0$ that is
$D_{x^{\prime}}^{\alpha}D_{x_{N}}^{j}u$, $j<m-n$,
$|\alpha|=[m-n+(1-\omega)\gamma]-j$, exist in the usual classical
sense at $x_{N}=0$ and

\[
{\displaystyle\sum\limits_{j=0}^{j\leq m-n}}
{\displaystyle\sum\limits_{|\alpha|=[m-n+(1-\omega)\gamma]-j}}
\left\langle D_{x^{\prime}}^{\alpha}D_{x_{N}}^{j}u\right\rangle
_{x^{\prime
},\overline{Q}}^{(\{m-n+(1-\omega)\gamma\})}+%
\]

\[
+{\displaystyle\sum\limits_{j=0}^{j\leq m-n}}
{\displaystyle\sum\limits_{|\alpha|=[m-n+\gamma]-j}} \left\langle
D_{x^{\prime}}^{\alpha}D_{x_{N}}^{j}u\right\rangle _{\omega
\gamma,x^{\prime},Q}^{(\{m-n+\gamma\})}\leq
\]

\[
\leq C\left( {\displaystyle\sum\limits_{i=1}^{N}} \left\langle
x_{N}^{n}D_{x_{i}}^{m}u\right\rangle _{\omega\gamma
,x_{i},\overline{Q}}^{(\gamma)}+\left\langle D_{t}u\right\rangle
_{t,\overline{Q}}^{(\gamma/m)}\right).
\]
This completes the proof of \eqref{s1.7} and of the Theorem
\ref{Ts1.1}.

\section{ Mixed and lower order derivatives of functions from
$C_{n,\omega\gamma}^{m+\gamma,\frac{m+\gamma}{m}}(\overline{Q})$.}
\label{ss1.4}

In this section we consider the mixed derivatives  $x_{N}%
^{n-(m-|\alpha|)}D_{x}^{\alpha}u$ of order $m-n<|\alpha|\leq m$ and
also the lower order derivatives $D_{x}^{\alpha}u$ of order
$|\alpha|\leq m-n$. The last lower order derivatives do not require
a weight in general but the situation in the case of an integer $n$
differs from that in the case of a noninteger $n$ . Therefore we
consider these two cases separately. Besides, we concentrate on the
local behaviour of functions near the singular boundary
$\{x_{N}=0\}$ and assume in this section that all functions under
consideration have compact support in $\overline{Q}$ or in
$\overline{H}$. This permits us to avoid the consideration of
possible behaviour of functions at infinity. Particularly, we will
show that for functions with compact support of fixed dimensions
norms \eqref{s1.3}, \eqref{s1.6}  and \eqref{s1.3.2},
\eqref{s1.6.3.0001} are equivalent.

Below we need the following lemma which is valid for both cases of an integer
and a noninteger $n$.

\begin{lemma}
\label{Ls1.s4.1}

Let $n\in(0,m)$ be arbitrary and a function $u(x,t)\in
C_{n,\omega\gamma }^{m+\gamma,\frac{m+\gamma}{m}}(\overline{Q})$ in the sense 
of \eqref{s1.6.2.n} has
compact support.
Then

\begin{equation}
{\displaystyle\sum\limits_{j=0}^{j<n}} \left\langle
x_{N}^{n-j}D_{x_{N}}^{m-j}u\right\rangle _{\omega\gamma
,x_{N},\overline{Q}}^{(\gamma)}\leq C\left\langle x_{N}^{n}D_{x_{N}}%
^{m}u\right\rangle _{\omega\gamma,x_{N},\overline{Q}}^{(\gamma)}.
\label{s1.s4.1}%
\end{equation}

\end{lemma}

\begin{proof}
The proof is by induction. For $j=0$ there is nothing to prove. Let
us prove that for $j<n$%

\begin{equation}
\left\langle x_{N}^{n-j}D_{x_{N}}^{m-j}u\right\rangle _{\omega\gamma
,x_{N},\overline{Q}}^{(\gamma)}\leq C\left\langle x_{N}^{n-j+1}D_{x_{N}%
}^{m-j+1}u\right\rangle _{\omega\gamma,x_{N},\overline{Q}}^{(\gamma)}.
\label{s1.s4.1.01}%
\end{equation}
Since the function $u(x,t)$ has a compact support, we have

\begin{equation}
x_{N}^{n-j}D_{x_{N}}^{m-j}u(x,t)=-x_{N}^{n-j}%
{\displaystyle\int\limits_{x_{N}}^{\infty}} \xi_{N}^{-n+j-1}\left[
\xi_{N}^{n-j+1}D_{\xi_{N}}^{m-j+1}u(x^{\prime},\xi
_{N},t)\right]  d\xi_{N}. \label{s1.s4.1.02}%
\end{equation}
Denote for brevity

\[
f(x,t)\equiv x_{N}^{n-j}D_{x_{N}}^{m-j}u(x,t),\quad-\left[  \xi_{N}%
^{n-j+1}D_{\xi_{N}}^{m-j+1}u(x^{\prime},\xi_{N},t)\right]  \equiv a(x^{\prime
},\xi_{N},t)
\]
and transform the representation for $f(x,t)$ making the change of
the variable $\xi_{N}=\eta x_{N}$ in the corresponding integral

\[
f(x,t)=%
{\displaystyle\int\limits_{1}^{\infty}}
\eta^{-n+j-1}a(x^{\prime},\eta x_{N},t)d\eta.
\]
Let $x_{N}$, $\overline{x}_{N}$ be fixed,
$0<x_{N}<\overline{x}_{N}$. We have

\[
\left\vert x_{N}^{\omega\gamma}\left[  f(x^{\prime},\overline{x}%
_{N},t)-f(x^{\prime},x_{N},t)\right]  \right\vert =
\]

\[
=\left\vert {\displaystyle\int\limits_{1}^{\infty}}
\eta^{-n+j-1-\omega\gamma}\left(  \eta x_{N}\right)
^{\omega\gamma}\left[
a(x^{\prime},\eta\overline{x}_{N},t)-a(x^{\prime},\eta
x_{N},t)\right] d\eta\right\vert \leq
\]

\[
\leq\left\langle a\right\rangle _{\omega\gamma,x_{N},\overline{Q}}^{(\gamma
)}(\overline{x}_{N}-x_{N})^{\gamma}%
{\displaystyle\int\limits_{1}^{\infty}}
\eta^{-n+j-1+(1-\omega)\gamma}d\eta\leq C\left\langle a\right\rangle
_{\omega\gamma,x_{N},\overline{Q}}^{(\gamma)}(\overline{x}_{N}-x_{N})^{\gamma
},
\]
where $-n+j-1+(1-\omega)\gamma<-1$ in view of \eqref{s1.2.1} and the
assumption $j<n$. This proves \eqref{s1.s4.1.01} and the lemma
follows by induction.
\end{proof}

\subsection{The case of an integer $n$.}
\label{sss1.4.1}

We suppose in this subsection that $n$ is an integer, $0<n<m$.
Consider first the derivatives $x_{N}^{n-j}D_{x}^{\alpha}u(x,t)$ of
order $|\alpha|=m-j$, $j\leq n$. In this case we have two different
situations for the derivatives $D_{x}^{\alpha}u(x,t)$ of the
particular order $|\alpha|=m-n$ with $\alpha_{N}<m-n$ and for the
derivative $D_{x_{N}}^{m-n}u(x,t)$.

\begin{proposition}
\label{Ps1.s4.1} Let $n\in(0,m)$ be an integer. Let a function
$u(x,t)\in
C_{n,\omega\gamma}^{m+\gamma,\frac{m+\gamma}{m}}(\overline{Q})$ in the sense
of \eqref{s1.6.2.n} has
compact support .  Then all seminorm of the function $u(x,t)$ in the
left hand side of \eqref{s1.7} except for, may be, seminorms of
$D_{x_{N}}^{m-n}u(x,t)$ are finite and

\[
\left\langle u\right\rangle _{n,\omega\gamma,\overline{Q}}^{(m+\gamma
,\frac{m+\gamma}{m})}\equiv%
{\displaystyle\sum\limits_{j=0}^{j\leq n}}
{\displaystyle\sum\limits_{\substack{|\alpha|=m-j,\\\alpha\neq(0,...,m-n)}}}
\left\langle x_{N}^{n-j}D_{x}^{\alpha}u\right\rangle _{\omega\gamma
,\overline{Q}}^{(\gamma,\gamma/m)}+%
\]

\[
+{\displaystyle\sum\limits_{j=0}^{j\leq n}}
{\displaystyle\sum\limits_{\substack{|\alpha|=m-j,\\\alpha\neq(0,...,m-n)}}}
\left\langle x_{N}^{n-j\omega}D_{x}^{\alpha}u\right\rangle _{t,\overline{Q}%
}^{(\frac{\gamma+j}{m})}+\left\langle D_{t}u\right\rangle _{\omega
\gamma,\overline{Q}}^{(\gamma,\gamma/m)}+
\]

\[
+ {\displaystyle\sum\limits_{j=0}^{j<m-n}}
{\displaystyle\sum\limits_{|\alpha|=m-n-j}} \left\langle
D_{x^{\prime}}^{\alpha}D_{x_{N}}^{j}u\right\rangle _{x^{\prime
},\overline{Q}}^{((1-\omega)\gamma)}+%
{\displaystyle\sum\limits_{j=0}^{j<m-n}}
{\displaystyle\sum\limits_{|\alpha|=m-n-j}} \left\langle
D_{x^{\prime}}^{\alpha}D_{x_{N}}^{j}u\right\rangle _{\omega
\gamma,x^{\prime},\overline{Q}}^{(\gamma)}+
\]

\begin{equation}
+%
{\displaystyle\sum\limits_{j=1}^{j\leq m-n}}
{\displaystyle\sum\limits_{\substack{|\alpha|=j,\\\alpha\neq(0,...,m-n)}}}
\left\langle D_{x}^{\alpha}u\right\rangle _{t,\overline{Q}}^{(1-\frac{j}%
{m-n}+\frac{\gamma}{m})}\leq C\left(
{\displaystyle\sum\limits_{i=1}^{N}} \left\langle
x_{N}^{n}D_{x_{i}}^{m}u\right\rangle _{\omega\gamma
,x_{i},\overline{Q}}^{(\gamma)}+\left\langle D_{t}u\right\rangle
_{t,\overline{Q}}^{(\gamma/m)}\right)  . \label{s1.s4.2}%
\end{equation}

\end{proposition}

\begin{proof}
The proof essentially follows from the proof of Theorem \ref{Ts1.1}.
Consider the smoothed  function $u_{\varepsilon}(x,t)$ as in
\eqref{s1.005} and in the end part of the proof of Theorem
\ref{Ts1.1}

\[
u_{\varepsilon}(x,t)\equiv%
{\displaystyle\int\limits_{R^{N-1}}}
{\displaystyle\int\limits_{-\infty}^{\infty}}
u(y^{\prime},x_{N},\tau)\omega_{\varepsilon}(x^{\prime}-y^{\prime}%
,t-\tau)dy^{\prime}d\tau.
\]
From Lemma \ref{Ls1.s4.1}, from \eqref{s1.007.1} with $j=n$, and
from the way of the construction of $u_{\varepsilon}(x,t)$ it
follows that for this function all seminorms in the left hand side
of \eqref{s1.s4.2} are finite (including the seminorm $\left\langle
x_{N}^{n-n\omega}D_{x_{N}}^{m-n}u_{\varepsilon}\right\rangle _{t,\overline{Q}%
}^{(\frac{\gamma+j}{m})}$) and thus this function satisfies
\eqref{s1.s4.2}. The rest of the proof coincides with the end part
of the proof of Theorem \ref{Ts1.1}.

\end{proof}

Consider now the derivative $D_{x_{N}}^{m-n}u(x,t)$.

\begin{lemma}
\label{Ls1.s4.2} Let $n\in(0,m)$ be an integer and a function
$u(x,t)\in C_{n,\omega\gamma
}^{m+\gamma,\frac{m+\gamma}{m}}(\overline{Q})$ in the sense
of \eqref{s1.6.2.n} has compact support.
Then $D_{x_{N}}^{m-n}u(x,t)$ is bounded and

\begin{equation}
\left\langle D_{x_{N}}^{m-n}u\right\rangle _{\omega\gamma,x_{N},\overline{Q}%
}^{(\gamma)}\leq C\left\langle x_{N}^{n}D_{x_{N}}^{m}u\right\rangle
_{\omega\gamma,x_{N},\overline{Q}}^{(\gamma)} \label{s1.s4.3}%
\end{equation}
if and only if

\begin{equation}
\left[  x_{N}^{n}D_{x_{N}}^{m}u(x,t)\right]  |_{x_{N}=0}\equiv0. \label{s1.s4.4}%
\end{equation}

\end{lemma}

\begin{proof}
From \eqref{s1.0010.3} and \eqref{s1.0011} it follows that condition
\eqref{s1.s4.4} is equivalent to the condition

\[
\left[  x_{N}D_{x_{N}}^{m-n+1}u(x,t)\right]  |_{x_{N}=0}\equiv0.
\]

Denote $a(x,t)\equiv x_{N}D_{x_{N}}^{m-n+1}u(x,t)$ and note that by
Lemma \ref{Ls1.s4.1} $\left\langle a(x,t)\right\rangle
_{\omega\gamma
,x_{N},\overline{Q}}^{(\gamma)}\leq C\left\langle x_{N}^{n}D_{x_{N}}%
^{m}u\right\rangle _{\omega\gamma,x_{N},\overline{Q}}^{(\gamma)}$.
We have the following representation

\begin{equation}
D_{x_{N}}^{m-n}u(x,t)=-%
{\displaystyle\int\limits_{x_{N}}^{\infty}}
D_{x_{N}}^{m-n+1}u(x^{\prime},\xi_{N},t)d\xi_{N}=-%
{\displaystyle\int\limits_{x_{N}}^{\infty}}
\frac{1}{\xi_{N}}a(x^{\prime},\xi_{N},t)d\xi_{N}. \label{s1.s4.5}%
\end{equation}
From this it directly follows that $D_{x_{N}}^{m-n}u(x,t)$ is
bounded at $x_{N}=0$ if and only if $a(x^{\prime},0,t)\equiv0$ that
is if and only if \eqref{s1.s4.4} is valid. Further, suppose that
$a(x^{\prime},0,t)\equiv0$. Let $x_{N}$, $\overline{x}_{N}$ be
fixed, $0<x_{N}<\overline{x}_{N}$. We have

\[
\left\vert x_{N}^{\omega\gamma}\left[  D_{x_{N}}^{m-n}(x^{\prime},\overline
{x}_{N},t)-D_{x_{N}}^{m-n}(x^{\prime},x_{N},t)\right]  \right\vert =
\]

\[
=\left\vert {\displaystyle\int\limits_{x_{N}}^{\overline{x}_{N}}}
\frac{1}{\xi_{N}}\left[
x_{N}^{\omega\gamma}a(x^{\prime},\xi_{N},t)\right]
d\xi_{N}\right\vert \leq%
{\displaystyle\int\limits_{x_{N}}^{\overline{x}_{N}}}
\frac{1}{\xi_{N}}\left\vert \xi_{N}^{\omega\gamma}\left[
a(x^{\prime},\xi _{N},t)-a(x^{\prime},0,t)\right]  \right\vert
d\xi_{N}\leq
\]

\[
\leq\left\langle a(x,t)\right\rangle _{\omega\gamma,x_{N},\overline{Q}%
}^{(\gamma)}%
{\displaystyle\int\limits_{x_{N}}^{\overline{x}_{N}}}
\frac{\xi_{N}^{\gamma}}{\xi_{N}}d\xi_{N}\leq C\left\langle
a(x,t)\right\rangle
_{\omega\gamma,x_{N},\overline{Q}}^{(\gamma)}(\overline{x}_{N}-x_{N})^{\gamma
}.
\]
In view of the definition of $a(x,t)$ this means that%

\[
\left\langle D_{x_{N}}^{m-n}u\right\rangle _{\omega\gamma,x_{N},\overline{Q}%
}^{(\gamma)}\leq C\left\langle x_{N}D_{x_{N}}^{m-n+1}u\right\rangle
_{\omega\gamma,x_{N},\overline{Q}}^{(\gamma)}\leq C\left\langle x_{N}%
^{n}D_{x_{N}}^{m}u\right\rangle _{\omega\gamma,x_{N},\overline{Q}}^{(\gamma)}%
\]
hence, the lemma.
\end{proof}

\begin{corollary}
\label{Cs1.s4.1} Let \ $n\in(0,m)$ be an integer. Let a function
$u(x,t)\in C_{n,\omega\gamma
}^{m+\gamma,\frac{m+\gamma}{m}}(\overline{Q})$ in the sense
of \eqref{s1.6.2.n} has compact support
and condition \eqref{s1.s4.4} is fulfilled.  Then all seminorm of
the function $u(x,t)$ in the left hand side of \eqref{s1.7} are
finite and

\[
\left\langle u\right\rangle _{n,\omega\gamma,\overline{Q}}^{(m+\gamma
,\frac{m+\gamma}{m})}\equiv%
{\displaystyle\sum\limits_{j=0}^{j\leq n}}
{\displaystyle\sum\limits_{|\alpha|=m-j}} \left\langle
x_{N}^{n-j}D_{x}^{\alpha}u\right\rangle _{\omega\gamma
,\overline{Q}}^{(\gamma,\gamma/m)}+%
{\displaystyle\sum\limits_{j=0}^{j\leq n}}
{\displaystyle\sum\limits_{|\alpha|=m-j}}
\left\langle x_{N}^{n-j\omega}D_{x}^{\alpha}u\right\rangle _{t,\overline{Q}%
}^{(\frac{\gamma+j}{m})}+\left\langle D_{t}u\right\rangle _{\omega
\gamma,\overline{Q}}^{(\gamma,\gamma/m)}+
\]

\[
+%
{\displaystyle\sum\limits_{j=0}^{j\leq m-n}}
{\displaystyle\sum\limits_{|\alpha|=m-n-j}} \left\langle
D_{x^{\prime}}^{\alpha}D_{x_{N}}^{j}u\right\rangle _{x^{\prime
},\overline{Q}}^{((1-\omega)\gamma)}+%
{\displaystyle\sum\limits_{j=0}^{j\leq m-n}}
{\displaystyle\sum\limits_{|\alpha|=m-n-j}} \left\langle
D_{x^{\prime}}^{\alpha}D_{x_{N}}^{j}u\right\rangle _{\omega
\gamma,x^{\prime},\overline{Q}}^{(\gamma)}%
\]

\begin{equation}
+%
{\displaystyle\sum\limits_{j=1}^{j\leq m-n}}
{\displaystyle\sum\limits_{|\alpha|=j}}
\left\langle D_{x}^{\alpha}u\right\rangle _{t,\overline{Q}}^{(1-\frac{j}%
{m-n}+\frac{\gamma}{m})}\leq C\left(
{\displaystyle\sum\limits_{i=1}^{N}} \left\langle
x_{N}^{n}D_{x_{i}}^{m}u\right\rangle _{\omega\gamma
,x_{i},\overline{Q}}^{(\gamma)}+\left\langle D_{t}u\right\rangle
_{t,\overline{Q}}^{(\gamma/m)}\right)  . \label{s1.s4.6}%
\end{equation}

\end{corollary}
The corollary follows from Proposition \ref{Ps1.s4.1} and Lemma
\ref{Ls1.s4.2}.

Consider now derivatives $D_{x}^{\alpha}u$ of order $|\alpha|<m-n$.
If condition \eqref{s1.s4.4} is fulfilled then all such derivatives
are characterized by the following statement.

\begin{corollary}
\label{Cs1.s4.2} Let $n\in(0,m)$ be an integer. Let a function
$u(x,t)\in
C_{n,\omega\gamma}^{m+\gamma,\frac{m+\gamma}{m}}(\overline{Q})$ in the sense
of \eqref{s1.6.2.n} has
compact support and condition \eqref{s1.s4.4} is fulfilled. For
multiindex $\alpha$ with $|\alpha|=j<m-n$ denote
$v_{\alpha}=D_{x}^{\alpha}u.$ The following estimate is valid

\begin{equation}%
{\displaystyle\sum\limits_{i=1}^{N}}
\left\langle D_{x_{i}}v_{\alpha}\right\rangle _{\omega\gamma,\overline{Q}%
}^{(\gamma,\gamma/m)}+\left\langle v_{\alpha}\right\rangle _{t,\overline{Q}%
}^{(1-\frac{j}{m-n}+\frac{\gamma}{m})}\leq C\left(
{\displaystyle\sum\limits_{i=1}^{N}} \left\langle
x_{N}^{n}D_{x_{i}}^{m}u\right\rangle _{\omega\gamma
,x_{i},\overline{Q}}^{(\gamma)}+\left\langle D_{t}u\right\rangle
_{t,\overline{Q}}^{(\gamma/m)}\right)  . \label{s1.s4.7}%
\end{equation}

\end{corollary}

This corollary directly follows from \eqref{s1.s4.6}.

In general case we have a weaker statement about smoothness of the
derivatives $D_{x}^{\alpha}u$ of order $|\alpha|<m-n$. In
particular, the derivative $D_{x_{N}}^{m-n-1}u$ belongs only to a
kind of Zygmund space $Z^{1}$ - see \eqref{s1.s4.9.1} below.

\begin{proposition}
\label{Ps1.s4.2}
Let $n\in(0,m)$ be an integer. Let a function
$u(x,t)\in C_{n,\omega\gamma
}^{m+\gamma,\frac{m+\gamma}{m}}(\overline{Q})$ has compact support.
For $v_{\alpha}=D_{x}^{\alpha}u$ with $|\alpha|=j<m-n-1$ we have

\begin{equation}%
{\displaystyle\sum\limits_{i=1}^{N}}
\left\langle D_{x_{i}}v_{\alpha}\right\rangle _{\omega\gamma,\overline{Q}%
}^{(\gamma,\gamma/m)}+\left\langle v_{\alpha}\right\rangle _{t,\overline{Q}%
}^{(1-\frac{j}{m-n}+\frac{\gamma}{m})}\leq C\left(
{\displaystyle\sum\limits_{i=1}^{N}} \left\langle
x_{N}^{n}D_{x_{i}}^{m}u\right\rangle _{\omega\gamma
,x_{i},\overline{Q}}^{(\gamma)}+\left\langle D_{t}u\right\rangle
_{t,\overline{Q}}^{(\gamma/m)}\right)  . \label{s1.s4.8}%
\end{equation}

For $v_{\alpha}=D_{x}^{\alpha}u$ with $|\alpha|=m-n-1>0$ we have

\begin{equation}
{\displaystyle\sum\limits_{\substack{|\alpha|=m-n-1,\\\alpha_{N}<m-n-1}}}
{\displaystyle\sum\limits_{i=1}^{N}}
\left\langle D_{x_{i}}v_{\alpha}\right\rangle _{\omega\gamma,\overline{Q}%
}^{(\gamma,\gamma/m)}+%
{\displaystyle\sum\limits_{i=1}^{N-1}} \left\langle
D_{x_{i}}D_{x_{N}}^{m-n-1}u\right\rangle _{\omega\gamma
,\overline{Q}}^{(\gamma,\gamma/m)}+%
{\displaystyle\sum\limits_{|\alpha|=m-n-1}}
\left\langle v_{\alpha}\right\rangle _{t,\overline{Q}}^{(1-\frac{m-n-1}%
{m-n}+\frac{\gamma}{m})}\leq\label{s1.s4.9}%
\end{equation}

\[
\leq C\left( {\displaystyle\sum\limits_{i=1}^{N}} \left\langle
x_{N}^{n}D_{x_{i}}^{m}u\right\rangle _{\omega\gamma
,x_{i},\overline{Q}}^{(\gamma)}+\left\langle D_{t}u\right\rangle
_{t,\overline{Q}}^{(\gamma/m)}\right)
\]
and also

\begin{equation}
\lbrack D_{x_{N}}^{m-n-1}u]_{x_{N},x^{\prime}}^{(1,(1-\omega)\gamma
)}+[D_{x_{N}}^{m-n-1}u]_{x_{N},t}^{(1,\frac{\gamma}{m})}\leq\label{s1.s4.9.1}%
\end{equation}

\[
\leq C\left( {\displaystyle\sum\limits_{i=1}^{N}} \left\langle
x_{N}^{n}D_{x_{i}}^{m}u\right\rangle _{\omega\gamma
,x_{i},\overline{Q}}^{(\gamma)}+\left\langle D_{t}u\right\rangle
_{t,\overline{Q}}^{(\gamma/m)}\right)  .
\]
where

\begin{equation}
\lbrack v]_{x_{N},x^{\prime},\overline{Q}}^{(1,(1-\omega)\gamma)}%
=\sup_{\substack{\theta>0,\overline{h}\in R^{N-1}\\(x,t)\in\overline{Q}}%
}\frac{|\Delta_{\theta,x_{N}}^{2}\Delta_{\overline{h},x^{\prime}}%
v(x,t)|}{\theta|\overline{h}|^{(1-\omega)\gamma}}, \label{s1.s4.10}%
\end{equation}

\begin{equation}
\lbrack v]_{x_{N},t,\overline{Q}}^{(1,\frac{\gamma}{m})}=\sup
_{\substack{\theta>0,\tau>0\\(x,t)\in\overline{Q}}}\frac{|\Delta_{\theta
,x_{N}}^{2}\Delta_{\tau,t}v(x,t)|}{\theta\tau^{\frac{\gamma}{m}}}.
\label{s1.s4.11}%
\end{equation}

\end{proposition}

\begin{proof}
The proof of \eqref{s1.s4.8} and \eqref{s1.s4.9} is completely
similar to the proof of Proposition \ref{Ps1.s4.1} in view of
\eqref{s1.007.1}. Let us prove \eqref{s1.s4.9.1}. Since estimates
for both terms in \eqref{s1.s4.9.1} are completely similar, we
estimate only the term
$[D_{x_{N}}^{m-n-1}u]_{x_{N},x^{\prime}}^{(1,(1-\omega)\gamma)}$.
Denote $a(x,t)=x_{N}D_{x_{N}}^{m-n+1}u$. We have from
\eqref{s1.s4.6} and from Proposition \ref{Ps1.01}

\begin{equation}
\left\langle a(x,t)\right\rangle
_{\omega\gamma,\overline{Q}}^{(\gamma
,\frac{\gamma}{m})}+\left\langle a(x,t)\right\rangle _{(1-\omega
)\gamma,x,\overline{Q}}^{(\gamma)}\leq C\left(
{\displaystyle\sum\limits_{i=1}^{N}} \left\langle
x_{N}^{n}D_{x_{i}}^{m}u\right\rangle _{\omega\gamma
,x_{i},\overline{Q}}^{(\gamma)}+\left\langle D_{t}u\right\rangle
_{t,\overline{Q}}^{(\gamma/m)}\right)  . \label{s1.s4.12}%
\end{equation}

From the definition of $a(x,t)$ we have%

\begin{equation}
D_{x_{N}}^{m-n-1}u=%
{\displaystyle\int\limits_{x_{N}}^{\infty}}
d\xi%
{\displaystyle\int\limits_{\xi}^{\infty}}
D_{x_{N}}^{m-n+1}u(x^{\prime},\eta,t)d\eta=%
{\displaystyle\int\limits_{x_{N}}^{\infty}}
d\xi%
{\displaystyle\int\limits_{\xi}^{\infty}}
\frac{1}{\eta}a(x^{\prime},\eta,t)d\eta. \label{s1.s4.12+1}%
\end{equation}
Let $\overline{h}\in R^{N-1}$ and $\theta>0$ be fixed. Denote
$b(x^{\prime
},\eta,t)=a(x^{\prime},\eta+\overline{h},t)-a(x^{\prime},\eta,t)$.
Then by simple direct calculations

\[
\left\vert \Delta_{\theta,x_{N}}^{2}\Delta_{\overline{h},x^{\prime}}D_{x_{N}%
}^{m-n-1}u(x,t)\right\vert =\left\vert \Delta_{\theta,x_{N}}^{2}%
{\displaystyle\int\limits_{x_{N}}^{\infty}}
d\xi%
{\displaystyle\int\limits_{\xi}^{\infty}}
\frac{1}{\eta}b(x^{\prime},\eta,t)d\eta\right\vert =
\]

\[
=\left\vert {\displaystyle\int\limits_{x_{N}}^{x_{N}+\theta}}
d\xi%
{\displaystyle\int\limits_{\xi}^{\xi+\theta}}
\frac{1}{\eta}b(x^{\prime},\eta,t)d\eta\right\vert \leq|b|_{\overline{Q}%
}^{(0)}%
{\displaystyle\int\limits_{x_{N}}^{x_{N}+\theta}}
d\xi%
{\displaystyle\int\limits_{\xi}^{\xi+\theta}} \frac{1}{\eta}d\eta=
\]

\[
=|b|_{\overline{Q}}^{(0)}%
{\displaystyle\int\limits_{x_{N}}^{x_{N}+\theta}}
(\ln(\xi+\theta)-\ln\xi)d\xi\leq C|b|_{\overline{Q}}^{(0)}\theta\leq
C\left\langle x_{N}D_{x_{N}}^{m-n+1}u\right\rangle
_{(1-\omega)\gamma
,x^{\prime},\overline{Q}}^{(\gamma)}|\overline{h}|^{(1-\omega)\gamma}\theta.
\]
Since $\overline{h}$ and $\theta$ are arbitrary, this proves
estimate \eqref{s1.s4.9.1} for the first term. The estimate of the
second term is completely analogous. This completes the proof of the
proposition.

\end{proof}

\subsection{The case of a noninteger $n$.}
\label{sss1.4.2}

In this case we have the following propositions analogous to
Propositions \ref{Ps1.s4.1}, \ref{Ps1.s4.2}.

\begin{proposition}
\label{Ps1.s4.3}

Let $n\in(0,m)$ be a noninteger. Let a function $u(x,t)\in C_{n,\omega\gamma
}^{m+\gamma,\frac{m+\gamma}{m}}(\overline{Q})$ in the sense
of \eqref{s1.6.2.n} has compact support.

Then

\[ \left\langle u\right\rangle
_{n,\omega\gamma,\overline{Q}}^{(m+\gamma
,\frac{m+\gamma}{m})}\equiv%
{\displaystyle\sum\limits_{j=0}^{j<n}}
{\displaystyle\sum\limits_{|\alpha|=m-j}} \left\langle
x_{N}^{n-j}D_{x}^{\alpha}u\right\rangle _{\omega\gamma
,\overline{Q}}^{(\gamma,\gamma/m)}+%
\]

\[
+{\displaystyle\sum\limits_{j=0}^{j<n}}
{\displaystyle\sum\limits_{|\alpha|=m-j}}
\left\langle x_{N}^{n-j\omega}D_{x}^{\alpha}u\right\rangle _{t,\overline{Q}%
}^{(\frac{\gamma+j}{m})}+\left\langle D_{t}u\right\rangle _{\omega
\gamma,\overline{Q}}^{(\gamma,\gamma/m)}+
\]

\[
+%
{\displaystyle\sum\limits_{j=0}^{j<m-n}}
{\displaystyle\sum\limits_{|\alpha|=[m-n+(1-\omega)\gamma]-j}}
\left\langle D_{x^{\prime}}^{\alpha}D_{x_{N}}^{j}u\right\rangle
_{x^{\prime
},\overline{Q}}^{(\{m-n+(1-\omega)\gamma\})}+%
\]

\[
+{\displaystyle\sum\limits_{j=0}^{j<m-n}}
{\displaystyle\sum\limits_{|\alpha|=[m-n+\gamma]-j}} \left\langle
D_{x^{\prime}}^{\alpha}D_{x_{N}}^{j}u\right\rangle _{\omega
\gamma,x^{\prime},\overline{Q}}^{(\{m-n+\gamma\})}+
\]

\begin{equation}
+%
{\displaystyle\sum\limits_{j=1}^{j<m-n}}
{\displaystyle\sum\limits_{|\alpha|=j}}
\left\langle D_{x}^{\alpha}u\right\rangle _{t,\overline{Q}}^{(1-\frac{j}%
{m-n}+\frac{\gamma}{m})}\leq C\left(
{\displaystyle\sum\limits_{i=1}^{N}} \left\langle
x_{N}^{n}D_{x_{i}}^{m}u\right\rangle _{\omega\gamma
,x_{i},\overline{Q}}^{(\gamma)}+\left\langle D_{t}u\right\rangle
_{t,\overline{Q}}^{(\gamma/m)}\right)  . \label{s1.s4.12+2}%
\end{equation}

\end{proposition}

The proof of this proposition is completely analogous to the proof
of Proposition \ref{Ps1.s4.1}.

In addition we have the following proposition.

\begin{proposition}
\label{Ps1.s4.4}

Let $n\in(0,m)$ be a noninteger. Let a function $u(x,t)\in
C_{n,\omega\gamma }^{m+\gamma,\frac{m+\gamma}{m}}(\overline{Q})$ in the sense
of \eqref{s1.6.2.n} has
compact support in a set $\{x_{N}\leq R\}$, $R>0$. Then

\[
\left\langle D_{x_{N}}^{[m-n]}u\right\rangle _{x_{N},\overline{Q}}%
^{(1-\{n\})}\leq
C|x_{N}^{\{n\}}D_{x_{N}}^{[m-n]+1}u|_{\overline{Q}}^{(0)}\leq
C(R)\left\langle x_{N}^{\{n\}}D_{x_{N}}^{[m-n]+1}u\right\rangle
_{\omega \gamma,\overline{Q}}^{(\gamma)}\leq
\]

\begin{equation}
\leq C(R)\left\langle x_{N}^{n}D_{x_{N}}%
^{m}u\right\rangle _{\omega\gamma,\overline{Q}}^{(\gamma)}. \label{s1.s4.15}%
\end{equation}

\end{proposition}

The proof of this proposition directly follows from the
Newton-Leibnitz formula and from \eqref{s1.s4.12+2}.

\section{Traces of functions from $C_{n,\omega\gamma}^{m+\gamma,\frac
{m+\gamma}{m}}(\overline{Q})$ at $\{x_{N}=0\}$.} \label{ss1.5}

As it was proved in Theorem \ref{Ts1.1}, for $u(x,t)\in C_{n,\omega\gamma}^{m+\gamma,\frac{m+\gamma}{m}%
}(\overline{Q})$ we have

\[%
{\displaystyle\sum\limits_{j=0}^{j\leq m-n}}
{\displaystyle\sum\limits_{|\alpha|=[m-n+(1-\omega)\gamma]-j}}
\left\langle D_{x^{\prime}}^{\alpha}D_{x_{N}}^{j}u\right\rangle
_{x^{\prime
},\overline{Q}}^{(\{m-n+(1-\omega)\gamma\})}+%
{\displaystyle\sum\limits_{j=1}^{j\leq m-n}}
{\displaystyle\sum\limits_{|\alpha|=j}}
\left\langle D_{x}^{\alpha}u\right\rangle _{t,\overline{Q}}^{(1-\frac{j}%
{m-n}+\frac{\gamma}{m})}\leq
\]

\begin{equation}
\leq C\left( {\displaystyle\sum\limits_{i=1}^{N}} \left\langle
x_{N}^{n}D_{x_{i}}^{m}u\right\rangle _{\omega\gamma
,x_{i},\overline{Q}}^{(\gamma)}+\left\langle D_{t}u\right\rangle
_{t,\overline{Q}}^{(\gamma/m)}\right)  , \label{s1.s5.1}%
\end{equation}
where the terms for $j=m-n$ are included if $n$ is an integer and\
if $D_{x_{N}}^{m-n}u$ is bounded. From this estimate it follows that
for $j\leq m-n$ and for a fixed $x_{N}>0$ the function
$D_{x_{N}}^{j}u(x^{\prime},x_{N},t)$ belongs to the space
$C_{x^{\prime},t}^{m-n+(1-\omega)\gamma-j,1-\frac{j}
{m-n}+\frac{\gamma}{m}}(R^{N-1}\times R^{1})$. This means that we
have the following statement.

\begin{proposition}
\label{s1.s5.1}

A function $u(x,t)\in
C_{n,\omega\gamma}^{m+\gamma,\frac{m+\gamma}{m} }(\overline{Q})$
with compact support and it's derivatives $D_{x_{N}}^{j}u$, $j\leq
m-n$, have traces at $x_{N}=0$ from the spaces
$D_{x_{N}}^{j}u(x^{\prime},0,t)\in$
 $ C_{x^{\prime},t}^{m-n+(1-\omega
)\gamma-j,1-\frac{j}{m-n}+\frac{\gamma}{m}}(R^{N-1}\times R^{1})$
and

\[
\left\Vert D_{x_{N}}^{j}u(x^{\prime},0,t)\right\Vert _{C_{x^{\prime}%
,t}^{m-n+(1-\omega)\gamma-j,1-\frac{j}{m-n}+\frac{\gamma}{m}}(R^{N-1}\times
R^{1})}\leq
\]

\begin{equation}
\leq C\left( {\displaystyle\sum\limits_{i=1}^{N}} \left\langle
x_{N}^{n}D_{x_{i}}^{m}u\right\rangle _{\omega\gamma
,x_{i},\overline{Q}}^{(\gamma)}+\left\langle D_{t}u\right\rangle
_{t,\overline{Q}}^{(\gamma/m)}\right)  . \label{s1.s5.2}%
\end{equation}

\end{proposition}

Now we consider the question of  the extension of functions
$v(x^{\prime},t)$ from the class
$C_{x^{\prime},t}^{m-n+(1-\omega)\gamma,1+\frac{\gamma}{m}}
(R^{N-1}\times R^{1})$ to the region $\overline{Q}$.

\begin{theorem}
\label{Ts1.s5.1}

There exists an operator $E:C_{x^{\prime},t}^{m-n+(1-\omega)\gamma
,1+\frac{\gamma}{m}}(R^{N-1}\times R^{1})\rightarrow
C_{n,\omega\gamma }^{m+\gamma,\frac{m+\gamma}{m}}(\overline{Q})$
defined on functions with compact supports in
$B_{R}=\{|x^{\prime}|\leq R,|t|\leq R\}$ with the property:

for a given function $v(x^{\prime},t)\in C_{x^{\prime},t}^{m-n+(1-\omega
)\gamma,1+\frac{\gamma}{m}}(R^{N-1}\times R^{1})$ with compact
support in $B_{R}$ the function $w(x,t)=Ev\in C_{n,\omega\gamma}%
^{m+\gamma,\frac{m+\gamma}{m}}(\overline{Q})$ has compact support and satisfies%

\begin{equation}
w(x^{\prime},0,t)=v(x^{\prime},t),\quad\left\Vert w\right\Vert _{C_{n,\omega
\gamma}^{m+\gamma,\frac{m+\gamma}{m}}(\overline{Q})}\leq C\left\Vert
v\right\Vert _{C_{x^{\prime},t}^{m-n+(1-\omega)\gamma,1+\frac{\gamma}{m}%
}(R^{N-1}\times R^{1})}, \label{s1.s5.3}%
\end{equation}
where the constant $C$ does not depend on $v$.

\end{theorem}

\begin{proof}

The proof is similar to the proof of corresponding Lemma 2.4 from
\cite{D1}. Let we are given a function $v(x^{\prime},t)\in C_{x^{\prime},t}%
^{m-n+(1-\omega)\gamma,1+\frac{\gamma}{m}}(R^{N-1}\times R^{1})$
with compact support.  Consider the following boundary problem with
$t$ as a parameter for a unknown function $u(x,t)$%

\begin{equation}
\Delta u(x,t)=\frac{\partial^{2}u}{\partial x_{1}^{2}}+...+\frac{\partial
^{2}u}{\partial x_{N}^{2}}=0,\quad x\in H,t\in R^{1}, \label{s1.s5.4}%
\end{equation}

\begin{equation}
u(x^{\prime},0,t)=v(x^{\prime},t),\quad x^{\prime}\in R^{N-1},t\in R^{1}.
\label{s1.s5.5}%
\end{equation}
It is well known (see \cite{19}, \cite{20}, for example ) that for a
fixed $t>0$ problem \eqref{s1.s5.4}, \eqref{s1.s5.5} has the unique
bounded solution $u(x,t)$ with

\begin{equation}
\left\Vert u(\cdot,t)\right\Vert _{C_{x}^{m-n+(1-\omega)\gamma}(\overline{H}%
)}\leq C\left\Vert v(\cdot,t)\right\Vert _{C_{x^{\prime}}^{m-n+(1-\omega
)\gamma}(R^{N-1})}. \label{s1.s5.6}%
\end{equation}
In Lemma 2.4 from \cite{D1} also proved that

\begin{equation}
\left\langle \frac{\partial u(x,t)}{\partial t}\right\rangle _{t,\overline{Q}%
}^{(\frac{\gamma}{m})}\leq C\left\langle \frac{\partial v(x^{\prime}%
,t)}{\partial t}\right\rangle _{t,R^{N-1}\times R^{1}}^{(\frac{\gamma}{m})}.
\label{s1.s5.7}%
\end{equation}
Therefore it is enough to consider the properties of $u(x,t)$ with
respect to the variables $x$. For this we will use the following
inequality (see \cite{Stein}, Chapter 5.4)

\begin{equation}
|D_{x}^{\alpha}u(x,t)|\leq C_{\alpha}x_{N}^{-|\alpha|+l}||v(\cdot
,t)||_{C_{x}^{l}(R^{N-1})},\quad|\alpha|\geq l, \label{s1.s5.8}%
\end{equation}
We now prove the estimate%

\begin{equation}%
{\displaystyle\sum\limits_{i=1}^{N}} \left\langle
x_{N}^{n}D_{x_{i}}^{m}u(\cdot,t)\right\rangle _{\omega
\gamma,x_{i},\overline{H}}^{(\gamma)}\leq C\left\Vert
v(\cdot,t)\right\Vert
_{C_{x^{\prime}}^{m-n+(1-\omega)\gamma}(R^{N-1})}. \label{s1.s5.9}%
\end{equation}

Since it is important to prove \eqref{s1.s5.9} for $x_{N}<1$ only
(for $x_{N}>1$, such the estimate follows from the local estimates
and is well- known), we consider only the case $x_{N}<1$. We also
use the well-known interpolation inequality
\begin{equation}
\left\langle v(x)\right\rangle _{x,\overline{\Omega}}^{(\gamma)}\leq C\left(
|v|_{\overline{\Omega}}^{(0)}\right)  ^{1-\gamma}\left(  \langle
v\rangle_{\overline{\Omega}}^{(1)}\right)  ^{\gamma}, \label{s1.s5.10}%
\end{equation}
which is valid for functions $v(x)\in C^{1}(\overline{\Omega})$, where
$\Omega$ is a domain (possibly unbounded) with sufficiently smooth boundary
(see, for example, \cite{18.1}, Ch.1 ). It is important that the constant $C$
does not depend on the size of the domain $\Omega$ under scaling. Consider
first the tangent variables $x_{i}$, $i=1,\overline{N-1}$.

Let $x_{N}$ be fixed. Then by \eqref{s1.s5.9} and \eqref{s1.s5.10},

\[
x_{N}^{\omega\gamma}\left\langle x_{N}^{n}D_{x_{i}}^{m}u(\cdot,x_{N}%
)\right\rangle _{x_{i},R^{N-1}}^{(\gamma)}\leq
\]%

\[
\leq Cx_{N}^{\omega\gamma}\left(  |x_{N}^{n}D_{x_{i}}^{m}u(\cdot
,x_{N})|_{,R^{N-1}}^{(0)}\right)  ^{1-\gamma}\left(  |x_{N}^{n}D_{x_{i}}%
^{m}u(\cdot,x_{N})|_{x,R^{N-1}}^{(1)}\right)  ^{\gamma}\leq
\]%

\[
\leq
C||u||_{C_{x}^{m-n+(1-\omega)\gamma}(R_{T}^{N-1})}x_{N}^{\omega\gamma
}\left(  x_{N}^{n}x_{N}^{-m+(m-n+(1-\omega)\gamma)}\right)
^{1-\gamma}\left(
x_{N}^{n}x_{N}^{-m-1+(m-n+(1-\omega)\gamma)}\right)  ^{\gamma}\leq
\]

\begin{equation}
\leq
C||v||_{C_{x}^{m-n+(1-\omega)\gamma}(R_{T}^{N-1})}. \label{s1.s5.11}%
\end{equation}
By the definition, this means that

\[
\left\langle x_{N}^{n}D_{x_{i}}^{m}u(\cdot,t)\right\rangle _{\omega
\gamma,x_{i},\overline{H}}^{(\gamma)}\leq C\left\Vert v(\cdot,t)\right\Vert
_{C_{x^{\prime}}^{m-n+(1-\omega)\gamma}(R^{N-1})},\quad i=\overline{1,N-1}.
\]

We consider now $\left\langle
x_{N}^{n}D_{x_{N}}^{m}u(\cdot,t)\right\rangle
_{\omega\gamma,x_{N},\overline{H}}^{(\gamma)}$ .  Let $x_{N}$ and
$\overline{x}_{N}$ be fixed.  We fix some
$\varepsilon_{0}\in(0,1/16)$ and consider the two cases, assuming
without loss of generality that $\overline {x}_{N}\leq x_{N}$. Let
first

\[
|x_{N}-\ \overline{x}_{N}|=(x_{N}-\ \overline{x}_{N})\geq\varepsilon_{0}%
x_{N}.
\]
Then
\[
x_{N}^{\omega\gamma}\frac{|x_{N}^{n}D_{x_{N}}^{m}u(x,t)-\overline{x}_{N}%
^{n}D_{x_{N}}^{m}u(\overline{x},t)|}{|x_{N}-\
\overline{x}_{N}|^{\gamma}}\leq
\]

\begin{equation}
\leq
C\left(  |x_{N}^{n-(1-\omega)\gamma}D_{x_{N}}^{m}u(x,t)|+|\overline{x}%
_{N}^{n-(1-\omega)\gamma}D_{x_{N}}^{m}u(\overline{x},t)|\right)  .
\label{s1.s5.12}%
\end{equation}
In this case, as above

\begin{equation}
|x_{N}^{n-(1-\omega)\gamma}D_{x_{N}}^{m}u(x,t)|\leq C||v||_{C_{x}%
^{m-n+(1-\omega)\gamma}(R_{T}^{N-1})}x_{N}^{n-(1-\omega)\gamma}x_{N}%
^{-m+(m-n+(1-\omega)\gamma)}= \label{s1.s5.12+1}%
\end{equation}

\[
=C|u|_{C_{x}^{m-n+(1-\omega)\gamma}(R_{T}^{N-1})}\leq C||v||_{C_{x}%
^{m-n+(1-\omega)\gamma}(R_{T}^{N-1})},
\]
and similarly for $|\overline{x}_{N}^{n-(1-\omega)\gamma}D_{x_{N}}%
^{m}u(\overline{x},t)|$.

Let now

\begin{equation}
0<(x_{N}-\overline{x}_{N})\leq\varepsilon_{0}x_{N}, \label{s1.s5.12+2}%
\end{equation}
and let also

\begin{equation} \Pi(x_{N})=\left\{  y\in
R_{+}^{N}:x_{N}-2\varepsilon_{0}x_{N}\leq y_{N}\leq
x_{N}+2\varepsilon_{0}x_{N}\right\}  , \label{s1.s5.15}%
\end{equation}
Then, taking into account that on $\Pi(x_{N})$ we have $y_{N}\sim x_{N}$, as
in the previous case

\[
x_{N}^{\omega\gamma}\frac{|x_{N}^{n}D_{x_{N}}^{m}u(x,t)-\overline{x}_{N}%
^{n}D_{x_{N}}^{m}u(\overline{x},t)|}{|x_{N}-\ \overline{x}_{N}|^{\gamma}}\leq
x_{N}^{\omega\gamma}\left\langle y_{N}^{n}D_{y_{N}}^{m}u(y,t)\right\rangle
_{y,\Pi(x_{N})}^{(\gamma)}\leq
\]%

\begin{equation}
\leq C\left(  x_{N}^{\omega\gamma}|y_{N}^{n-\gamma}D^{2}u|_{\Pi(x_{N})}%
^{(0)}+x_{N}^{\omega\gamma}x_{N}^{n}\left\langle D_{y_{N}}^{m}%
u(y,t)\right\rangle _{y,\Pi(x_{N})}^{(\gamma)}\right)  \equiv A_{1}+A_{2}.
\label{s1.s5.16}%
\end{equation}
Here $A_{1}$ is estimated in the same way as in \eqref{s1.s5.12+1},
and $A_{2}$ - in similar way after the estimate

\[
\left\langle D_{y_{N}}^{m}u(y,t)\right\rangle _{y,\Pi(x_{N})}^{(\gamma)}\leq
C\left(  |D_{y_{N}}^{m}u(y,t)|_{\overline{\Omega}}^{(0)}\right)  ^{1-\gamma
}\left(  |D_{y_{N}}^{m+1}u(y,t)|_{\overline{\Omega}}^{(1)}\right)  ^{\gamma}.
\]
This gives

\[
\left\langle x_{N}^{n}D_{x_{N}}^{m}u(\cdot,t)\right\rangle _{\omega
\gamma,x_{N},\overline{H}}^{(\gamma)}\leq C\left\Vert v(\cdot,t)\right\Vert
_{C_{x^{\prime}}^{m-n+(1-\omega)\gamma}(R^{N-1})}%
\]
Now fix $\eta(x,t)\in C^{\infty}(\overline{Q})$ with compact support and with
$\eta(x^{\prime},0,t)\equiv1$ on $B_{R}$. Then we can define $w(x,t)\equiv
Ev(x,t)\equiv u(x,t)\eta(x,t)$ . This completes the proof of the theorem.

\end{proof}

\section{Some interpolations inequalities for functions from
$C_{n,\omega\gamma}^{m+\gamma,\frac{m+\gamma}{m}}(\overline{Q})$,
$C_{n,\omega\gamma}^{m+\gamma}(\overline{H})$.} \label{ss1.6}

In this section we prove some interpolation inequalities for
functions from the spaces
$C_{n,\omega\gamma}^{m+\gamma,\frac{m+\gamma}{m}}(\overline{Q})$,
$C_{n,\omega\gamma}^{m+\gamma}(\overline{H})$. These inequalities
are consequences of \eqref{s1.7}, \eqref{s1.8} and they are useful
in applications.

\begin{theorem}
\label{Ts5.1}

Let a function $u(x)\in$
$C_{n,\omega\gamma}^{m+\gamma}(\overline{H})$ and
$\alpha=(\alpha_{1},...,\alpha_{N})$, $|\alpha|=m$, be a multiindex,
$k\in\{1,2,...,N\}$. Then for any $\varepsilon>0$ ($\varepsilon$ may
be chosen big or small)

\[
\left\langle x_{N}^{n}D_{x}^{\alpha}u\right\rangle _{\omega\gamma
,x_{k},,\overline{H}}^{(\gamma)}\leq C\varepsilon^{-\alpha_{k}-\gamma}%
{\displaystyle\sum\limits_{i=1,i\neq k}^{N}} \left\langle
x_{N}^{n}D_{x_{i}}^{m}u\right\rangle _{\omega\gamma
,x_{i},\overline{H}}^{(\gamma)}+
\]

\begin{equation}
+C\varepsilon^{m-\alpha_{k}}\left\langle
x_{N}^{n}D_{x_{k}}^{m}u\right\rangle _{\omega\gamma,x_{k},\overline{H}%
}^{(\gamma)}, \ \ \ k<N, \label{s1.6.1}%
\end{equation}

\[
\left\langle x_{N}^{n}D_{x}^{\alpha}u\right\rangle _{\omega\gamma
,x_{N},,\overline{H}}^{(\gamma)}\leq
C\varepsilon^{-\alpha_{k}-(1-\omega
)\gamma}%
{\displaystyle\sum\limits_{i=1}^{N-1}} \left\langle
x_{N}^{n}D_{x_{i}}^{m}u\right\rangle _{\omega\gamma
,x_{i},\overline{H}}^{(\gamma)}+
\]

\begin{equation}
+C\varepsilon^{m-\alpha_{k}}\left\langle
x_{N}^{n}D_{x_{N}}^{m}u\right\rangle _{\omega\gamma,x_{N},\overline{H}%
}^{(\gamma)}, \ \ \ k=N, \label{s1.6.1.1}%
\end{equation}
where the constants $C$ does not depend on $\varepsilon$, $u$.

\end{theorem}

\begin{proof}

Let $\varepsilon>0$ be fixed. Consider the function $v_{\varepsilon
}(y)=u(y_{1},y_{2},...,\varepsilon y_{k},...y_{N-1},y_{N})\in
C_{n,\omega \gamma}^{m+\gamma}(\overline{H})$.  Then from
\eqref{s1.8} we have

\begin{equation}
\left\langle y_{N}^{n}D_{y}^{\alpha}v_{\varepsilon}\right\rangle
_{\omega\gamma,y_{k},,\overline{H}}^{(\gamma)}\leq C%
{\displaystyle\sum\limits_{i=1}^{N}} \left\langle
y_{N}^{n}D_{y_{i}}^{m}v_{\varepsilon}\right\rangle _{\omega
\gamma,x_{i},\overline{H}}^{(\gamma)}, \label{s1.6.2}%
\end{equation}
where the constant $C$ does not depend on $\varepsilon$.  Now make
in \eqref{s1.6.2} the change of the variables

\[
y=e(x):\quad y_{i}=x_{i},i\neq k,\quad y_{k}=\varepsilon^{-1}x_{k}%
\]

and take into account that $v_{\varepsilon}(y)\circ e(x)=u(x)$. This
gives \eqref{s1.6.1}, \eqref{s1.6.1.1} and completes the proof.

\end{proof}

\begin{theorem}
\label{Ts6.2}

Let a function $u(x)\in$ $C_{n,\omega\gamma}^{m+\gamma,\frac{m+\gamma}{m}%
}(\overline{Q})$ and $\alpha=(\alpha_{1},...,\alpha_{N})$, $|\alpha|=m$, be a
multiindex, $k\in\{1,2,...,N\}$. Then for any $\varepsilon>0$%

\[
\left\langle x_{N}^{n}D_{x}^{\alpha}u\right\rangle _{\omega\gamma
,x_{k,}\overline{Q}}^{(\gamma)}\leq C\varepsilon^{-\alpha_{k}-\gamma}%
{\displaystyle\sum\limits_{i=1,i\neq k}^{N}} \left\langle
x_{N}^{n}D_{x_{i}}^{m}u\right\rangle _{\omega\gamma
,x_{i},\overline{Q}}^{(\gamma)}+
\]

\begin{equation}
+C\varepsilon^{m-\alpha_{k}}\left\langle
x_{N}^{n}D_{x_{k}}^{m}u\right\rangle _{\omega\gamma,x_{k},\overline{Q}%
}^{(\gamma)}, \ \ \ k<N, \label{s1.6.3}%
\end{equation}

\[
\left\langle x_{N}^{n}D_{x}^{\alpha}u\right\rangle _{\omega\gamma
,x_{N},,\overline{H}}^{(\gamma)}\leq
C\varepsilon^{-\alpha_{k}-(1-\omega
)\gamma}%
{\displaystyle\sum\limits_{i=1}^{N-1}} \left\langle
x_{N}^{n}D_{x_{i}}^{m}u\right\rangle _{\omega\gamma
,x_{i},\overline{H}}^{(\gamma)}+
\]

\begin{equation}
+C\varepsilon^{m-\alpha_{k}}\left\langle
x_{N}^{n}D_{x_{N}}^{m}u\right\rangle _{\omega\gamma,x_{N},\overline{H}%
}^{(\gamma)}, \ \ \ k=N, \label{s1.6.4}%
\end{equation}

\begin{equation}
\left\langle x_{N}^{n}D_{x}^{\alpha}u\right\rangle _{t,\overline{Q}}%
^{(\gamma/m)}\leq\varepsilon^{-\gamma/m}C%
{\displaystyle\sum\limits_{i=1}^{N}} \left\langle
x_{N}^{n}D_{x_{i}}^{m}u\right\rangle _{\omega\gamma
,x_{i},\overline{Q}}^{(\gamma)}+C\varepsilon\left\langle
D_{t}u\right\rangle
_{t,\overline{Q}}^{(\gamma/m)}, \label{s1.6.5}%
\end{equation}

\[
\left\langle D_{t}u\right\rangle
_{\omega\gamma,x_{k},\overline{Q}}^{(\gamma
)}\leq C\varepsilon^{-\gamma}%
{\displaystyle\sum\limits_{i=1,i\neq k}^{N}} \left\langle
x_{N}^{n}D_{x_{i}}^{m}u\right\rangle _{\omega\gamma
,x_{i},\overline{Q}}^{(\gamma)}+\varepsilon^{-(1-\omega)\gamma}C\left\langle
D_{t}u\right\rangle _{t,\overline{Q}}^{(\gamma/m)}+
\]

\begin{equation}
+\varepsilon^{m}%
C\left\langle x_{N}^{n}D_{x_{k}}^{m}u\right\rangle _{\omega\gamma
,x_{k},\overline{Q}}^{(\gamma)}, \ \ \ k<N, \label{s1.6.6}%
\end{equation}

\[
\left\langle D_{t}u\right\rangle
_{\omega\gamma,x_{N},\overline{Q}}^{(\gamma
)}\leq C\varepsilon^{-n-(1-\omega)\gamma}%
{\displaystyle\sum\limits_{i=1}^{N-1}} \left\langle
x_{N}^{n}D_{x_{i}}^{m}u\right\rangle _{\omega\gamma
,x_{i},\overline{Q}}^{(\gamma)}+
\]

\begin{equation}
+C\varepsilon^{-(1-\omega)\gamma}\left\langle
D_{t}u\right\rangle _{t,\overline{Q}}^{(\gamma/m)}+C\varepsilon^{m-n}%
\left\langle x_{N}^{n}D_{x_{N}}^{m}u\right\rangle _{\omega\gamma
,x_{N},\overline{Q}}^{(\gamma)}, \label{s1.6.7}%
\end{equation}
where the constants $C$ does not depend on $\varepsilon$, $u$.

\end{theorem}

The proof of this theorem is identical to that of the previous theorem.

\begin{theorem}
\label{Ts6.3}

Let a function $u(x,t)\in$ $C_{n,\omega\gamma}^{m+\gamma,\frac{m+\gamma}{m}%
}(\overline{Q})$ has compact support. Let the support of $u(x,t)$ is
included in $Q_{R}=\{|x|\leq R,|t|\leq R\}$. Then for an integer
$0\leq j<n$ and for an arbitrary $h>0$%

\begin{equation}%
{\displaystyle\sum\limits_{|\alpha|=m-j}}
|x_{N}^{n-j}D_{x}^{\alpha}u(x,t)|_{\overline{Q}}^{(0)}\leq C\left(
h^{(1-\omega)\gamma}%
{\displaystyle\sum\limits_{|\alpha|=m-j}} \left\langle
x_{N}^{n-j}D_{x}^{\alpha}u(x,t)\right\rangle _{\omega
\gamma,x,\overline{Q}}^{(\gamma)}\right.  + \label{s1.6.7.1}%
\end{equation}

\[
+\left.  \frac{(1+R)}{h}%
{\displaystyle\sum\limits_{|\alpha|=m-j-1}}
|x_{N}^{n-j-1}D_{x}^{\alpha}u(x,t)|_{\overline{Q}}^{(0)}\right)
,n-j\geq1,
\]

\begin{equation}%
{\displaystyle\sum\limits_{|\alpha|=m-j}}
|x_{N}^{n-j}D_{x}^{\alpha}u(x,t)|_{\overline{Q}}^{(0)}\leq C\left(
h^{(1-\omega)\gamma}%
{\displaystyle\sum\limits_{|\alpha|=m-j}} \left\langle
x_{N}^{n}D_{x}^{\alpha}u(x,t)\right\rangle _{\omega
\gamma,x,\overline{Q}}^{(\gamma)}\right.  + \label{s1.6.7.1}%
\end{equation}

\[
+\left.  \frac{(1+R^{n-j})}{h}%
{\displaystyle\sum\limits_{|\alpha|=m-j-1}}
|D_{x}^{\alpha}u(x,t)|_{\overline{Q}}^{(0)}\right)  ,n-j<1.
\]

\end{theorem}

\begin{proof}

Let $n-1>0$. Consider first the estimate of a derivative
$D_{x_{N}}D_{x}^{\alpha}u$ , $|\alpha|=m-1$\  with respect to
$x_{N}$. Let $h>0$ and $\varepsilon\in(0,1)$\  be fixed. We have

\[
x_{N}^{n}D_{x_{N}}D_{x}^{\alpha}u(x,t)=\left(  x_{N}^{n}D_{x_{N}}D_{x}%
^{\alpha}u(x,t)-x_{N}^{n}\frac{\Delta_{h,x_{N}}D_{x}^{\alpha}u(x,t)}%
{h}\right)+
\]

\begin{equation}
  +x_{N}^{n}\frac{\Delta_{h,x_{N}}D_{x}^{\alpha}u(x,t)}{h}=
\label{s1.6.8}%
\end{equation}

\[
=-x_{N}^{n}%
{\displaystyle\int\limits_{0}^{1}}
\left[  D_{x_{N}}D_{x}^{\alpha}u(x^{\prime},x_{N}+\theta h,t)-D_{x_{N}}%
D_{x}^{\alpha}u(x,t)\right]  d\theta+x_{N}^{n}\frac{\Delta_{h,x_{N}}%
D_{x}^{\alpha}u(x,t)}{h}=A_{1}+A_{2}%
\]
and evidently

\begin{equation}
|A_{2}|\leq CR\frac{|x_{N}^{n-1}D_{x}^{\alpha}u(x,t)|_{\overline{Q}}^{(0)}}%
{h}. \label{s1.6.9}%
\end{equation}
The expression $A_{1}$ we represent as%

\[
A_{1}=-%
{\displaystyle\int\limits_{0}^{1}} \Delta_{\theta h,x_{N}}\left[
x_{N}^{n}D_{x_{N}}D_{x}^{\alpha}u(x,t)\right] d\theta+
\]

\[
+%
{\displaystyle\int\limits_{0}^{1}} \left[  (x_{N}+\theta
h)^{n}-x_{N}^{n}\right]  D_{x_{N}}D_{x}^{\alpha
}u(x^{\prime},x_{N}+\theta h)d\theta=A_{11}+A_{12}.
\]

The estimate of $A_{11}$ is%

\[
|A_{11}|\leq\left( {\displaystyle\int\limits_{0}^{1}}
\frac{|\Delta_{\theta h,x_{N}}\left[
x_{N}^{n}D_{x_{N}}D_{x}^{\alpha
}u(x,t)\right]  |}{\left(  \theta h\right)  ^{(1-\omega)\gamma}}%
\theta^{(1-\omega)\gamma}d\theta\right)  h^{(1-\omega)\gamma}\leq
\]

\begin{equation}
\leq C\left\langle x_{N}^{n}D_{x_{N}}D_{x}^{\alpha}u(x,t)\right\rangle
_{x_{N},\overline{Q}}^{((1-\omega)\gamma)}h^{(1-\omega)\gamma}\leq
C\left\langle x_{N}^{n}D_{x_{N}}D_{x}^{\alpha}u(x,t)\right\rangle
_{\omega\gamma,x_{N},\overline{Q}}^{(\gamma)}h^{(1-\omega)\gamma}.
\label{s1.6.11}%
\end{equation}
To estimate $A_{12}$ we apply the integration by parts. This gives
($y=\theta h$)

\[
A_{12}=\frac{1}{h}\left[  (x_{N}+h)^{n}-x_{N}^{n}\right]  D_{x}^{\alpha
}u(x^{\prime},x_{N}+h)-
\]

\[
-\frac{n}{h}%
{\displaystyle\int\limits_{0}^{h}}
(x_{N}+y)^{n-1}D_{x}^{\alpha}u(x^{\prime},x_{N}+y)dy
\]
and we obtain

\begin{equation}
|A_{12}|\leq C(1+R)\frac{|x_{N}^{n-1}D_{x}^{\alpha}u(x,t)|_{\overline{Q}%
}^{(0)}}{h}. \label{s1.6.12}%
\end{equation}
Note that in the case $n<1$ the derivative $D_{x}^{\alpha}u(x,t)$
and we have

\[
\left\vert {\displaystyle\int\limits_{0}^{h}}
(x_{N}+y)^{n-1}D_{x}^{\alpha}u(x^{\prime},x_{N}+y)dy\right\vert \leq
|D_{x}^{\alpha}u(x,t)|_{\overline{Q}}^{(0)}%
{\displaystyle\int\limits_{0}^{h}}
(x_{N}+y)^{n-1}dy\leq CR^{n}|D_{x}^{\alpha}u(x,t)|_{\overline{Q}}^{(0)}%
\]
From \eqref{s1.6.8}- \eqref{s1.6.12} it follows that

\[
|x_{N}^{n}D_{x_{N}}D_{x}^{\alpha}u(x,t)|_{\overline{Q}}^{(0)}\leq
Ch^{(1-\omega)\gamma}\left\langle x_{N}^{n}D_{x_{N}}D_{x}^{\alpha
}u(x,t)\right\rangle _{\omega\gamma,x_{N},\overline{Q}}^{(\gamma)}+
\]

\[
+\left\{
\begin{array}
[c]{c}%
(1+R)\frac{|x_{N}^{n-1}D_{x}^{\alpha}u(x,t)|_{\overline{Q}}^{(0)}}{h}%
,n\geq1,\\
(1+R^{n})\frac{|D_{x}^{\alpha}u(x,t)|_{\overline{Q}}^{(0)}}{h},n<1.
\end{array}
\right.
\]
The estimates of other derivatives
$x_{N}^{n}D_{x^{\prime}}D_{x}^{\alpha }u(x,t)$ are completely
similar and give for an arbitrary $h>0$

\[%
{\displaystyle\sum\limits_{|\alpha|=m}}
|x_{N}^{n}D_{x_{N}}D_{x}^{\alpha}u(x,t)|_{\overline{Q}}^{(0)}\leq
C\left(
h^{(1-\omega)\gamma}%
{\displaystyle\sum\limits_{|\alpha|=m}} \left\langle
x_{N}^{n}D_{x}^{\alpha}u(x,t)\right\rangle _{\omega
\gamma,x,\overline{Q}}^{(\gamma)}+\right.
\]

\[
\left.+\frac{(1+R)}{h}%
{\displaystyle\sum\limits_{|\alpha|=m-1}}
|x_{N}^{n-1}D_{x}^{\alpha}u(x,t)|_{\overline{Q}}^{(0)}\right) , \ \
\ n\geq1.
\]

\[%
{\displaystyle\sum\limits_{|\alpha|=m}}
|x_{N}^{n}D_{x_{N}}D_{x}^{\alpha}u(x,t)|_{\overline{Q}}^{(0)}\leq
C\left(
h^{(1-\omega)\gamma}%
{\displaystyle\sum\limits_{|\alpha|=m}} \left\langle
x_{N}^{n}D_{x}^{\alpha}u(x,t)\right\rangle _{\omega
\gamma,x,\overline{Q}}^{(\gamma)}+\right.
\]

\[
\left.+\frac{(1+R^{n})}{h}%
{\displaystyle\sum\limits_{|\alpha|=m-1}}
|D_{x}^{\alpha}u(x,t)|_{\overline{Q}}^{(0)}\right)  , \ \ \ n<1.
\]

The estimates of derivatives $D_{x}^{\alpha}u(x,t)$ of order
$|\alpha|=m-j$, $j<n$, are obtained in the same way.

\end{proof}

\begin{corollary}
The norms \eqref{s1.6} and \eqref{s1.6.2.n} and also the norms \eqref{s1.6.1a} and \eqref{s1.6.3.0001}   are equivalent.
\end{corollary}\label{Cn}

This assertion follows from the previous theorem, from section \ref{ss1.4} and from Theorem \ref{Ts1.1}.

\section{The spaces $C_{n,\omega\gamma}^{m+\gamma,\frac{m+\gamma}{m}%
}(\overline{\Omega}_{T})$,
$C_{n,\omega\gamma}^{m+\gamma}(\overline{\Omega})$ in the case of an
arbitrary smooth domain.} \label{ss1.8}

Let $\Omega$ be a domain in $R^{N}$  (bounded or unbounded) with
boundary $\partial\Omega$ of the class $C^{m+\gamma}$. Let $d(x)$ be
a function of the class $C^{1+\gamma}(\overline{\Omega})$ with the
property

\begin{equation}
\nu\cdot dist(x,\partial\Omega)\leq d(x)\leq\nu^{-1}\cdot dist(x,\partial
\Omega),\quad,dist(x,\partial\Omega)\leq1,\quad\nu>0. \label{s1.8.1}%
\end{equation}
As such a function can serve, for example, the bounded solution of
the problem%

\[
\Delta d(x)=-1,x\in\Omega,\quad d(x)|_{\partial\Omega}=0.
\]
For $x,\overline{x}\in\Omega$ we denote $d(x,\overline{x})=\max
\{d(x),d(\overline{x})\}$ and for a function $v(x)$ denote%

\[
\left\langle v\right\rangle _{\omega\gamma,\overline{\Omega}}^{(\gamma)}%
=\sup_{x,\overline{x}\in\overline{\Omega}}d(x,\overline{x})^{\omega\gamma
}\frac{|u(\overline{x})-u(x)|}{|\overline{x}-x|^{\gamma}}.
\]
Define the space $C_{\omega\gamma}^{\gamma}(\overline{\Omega})$ as
the space of functions $u(x)$ with the finite norm%

\begin{equation}
\left\Vert u\right\Vert _{C_{\omega\gamma}^{\gamma}(\overline{\Omega})}%
\equiv|u|_{\overline{\Omega}}^{(0)}+\left\langle u\right\rangle _{\omega
\gamma,\overline{\Omega}}^{(\gamma)}. \label{s1.8.2}%
\end{equation}
And define the space
$C_{n,\omega\gamma}^{m+\gamma}(\overline{\Omega})$ as the
space of continuous in $\overline{\Omega}$ functions $u(x)$ with the finite norm%

\begin{equation}
\left\Vert u\right\Vert _{C_{n,\omega\gamma}^{m+\gamma}(\overline{\Omega}%
)}\equiv|u|_{n,\omega\gamma,\overline{\Omega}}^{(m+\gamma)}\equiv
|u|_{\overline{\Omega}}^{(0)}+\sum_{|\alpha|=m}\left\langle d(x)^{n}%
D_{x}^{\alpha}u(x)\right\rangle _{\omega\gamma,\overline{\Omega}}^{(\gamma)}.
\label{s1.8.3}%
\end{equation}

For $T>0$ denote $\Omega_{T}=\{(x,t):x\in\Omega,t\in(0,T)\}$ and define the
space $C_{n,\omega\gamma}^{m+\gamma,\frac{m+\gamma}{m}}(\overline{\Omega}%
_{T})$ as the space of continuous in $\overline{\Omega}_{T}$ functions
$u(x,t)$ with the finite norm%

\begin{equation}
\left\Vert u\right\Vert _{C_{n,\omega\gamma}^{m+\gamma,\frac{m+\gamma}{m}%
}(\overline{\Omega}_{T})}\equiv|u|_{n,\omega\gamma,\overline{\Omega}_{T}%
}^{(m+\gamma)}\equiv|u|_{\overline{\Omega}_{T}}^{(0)}+\sum_{|\alpha
|=m}\left\langle d(x)^{n}D_{x}^{\alpha}u(x,t)\right\rangle _{\omega
\gamma,\overline{\Omega}_{T}}^{(\gamma)}+\left\langle D_{t}u\right\rangle
_{t,\overline{\Omega}_{T}}^{(\gamma/m)}. \label{s1.8.4}%
\end{equation}

\begin{theorem}
\label{Ts1.8.1}

Let $\left\Vert u\right\Vert _{C_{n,\omega\gamma}^{m+\gamma
,\frac{m+\gamma}{m}}(\overline{\Omega}_{T})}<\infty$. Then

\[%
{\displaystyle\sum\limits_{j=0}^{j\leq n}}
{\displaystyle\sum\limits_{|\alpha|=m-j}} \left\langle
d(x)^{n-j}D_{x}^{\alpha}u\right\rangle _{\omega\gamma
,\overline{Q}}^{(\gamma,\gamma/m)}+%
{\displaystyle\sum\limits_{j=0}^{j\leq n}}
{\displaystyle\sum\limits_{|\alpha|=m-j}}
\left\langle d(x)^{n-j\omega}D_{x}^{\alpha}u\right\rangle _{t,\overline{Q}%
}^{(\frac{\gamma+j}{m})}+
\]

\begin{equation}
+%
{\displaystyle\sum\limits_{j=1}^{j\leq m-n}}
{\displaystyle\sum\limits_{|\alpha|=j}}
\left\langle D_{x}^{\alpha}u\right\rangle _{t,\overline{Q}}^{(1-\frac{j}%
{m-n}+\frac{\gamma}{m})}+%
{\displaystyle\sum\limits_{|\alpha|<m-n}}
|D_{x}^{\alpha}u|_{\omega\gamma,\Omega_{T}}^{(\gamma)}\leq
C\left\Vert u\right\Vert
_{C_{n,\omega\gamma}^{m+\gamma,\frac{m+\gamma}{m}}(\overline
{\Omega}_{T})}. \label{s1.8.5}%
\end{equation}

 If $n$ is an integer and $d(x)^{n}D_{x}^{\alpha}u|_{\partial\Omega
}\equiv0$ for $|\alpha|=m$ then also%

\begin{equation}%
{\displaystyle\sum\limits_{|\alpha|=m-n}}
|D_{x}^{\alpha}u|_{\omega\gamma,\Omega_{T}}^{(\gamma)}\leq
C\left\Vert u\right\Vert
_{C_{n,\omega\gamma}^{m+\gamma,\frac{m+\gamma}{m}}(\overline
{\Omega}_{T})}. \label{s1.8.6}%
\end{equation}

\end{theorem}

The proof of this theorem follows from the results of sections
\ref{ss1}, \ref{ss1.4} and \ref{ss1.6} by localisation and
considering the functions $u(x,t)\eta (x)$, where $\eta(x)\in
C^{\infty}(\overline{\Omega})$ and has sufficiently small support
near $\partial\Omega$. After corresponding change of the variables
$v(x,t)=u(x,t)\eta(x)$ can be considered in a half-space $Q$. The
proof is pretty standard with the making use of the interpolation
inequalities and therefore we omit it.

\section{Spaces $\ C_{n,\omega\gamma,0}^{m+\gamma,\frac{m+\gamma}{m}%
}(\overline{\Omega}_{T})$.} \label{ss1.9}

 We denote by  $C_{n,\omega\gamma,0}^{m+\gamma,\frac{m+\gamma}{m}%
}(\overline{\Omega}_{T})$ the closed subspace of $C_{n,\omega\gamma}%
^{m+\gamma,\frac{m+\gamma}{m}}(\overline{\Omega}_{T})$ consisting of
functions $u(x,t)$ with the property $u(x,0)\equiv
u_{t}(x,0)\equiv0$ in $\overline{\Omega}$.

\begin{proposition}
\label{P0000001}

Let $u(x,t)\in C_{n,\omega\gamma,0}^{m+\gamma,\frac{m+\gamma}{m}%
}(\overline{\Omega}_{T})$, $T\leq1$. Then for $\ 1\leq j\leq n$,
$|\alpha|=m-j$ and with some $\delta>0$%

\begin{equation}
|d(x)^{n}D_{x}^{\alpha}u|_{\overline{\Omega}_{T}}^{(\gamma,\gamma/m)}\leq
CT^{\delta}\left\Vert u\right\Vert _{C_{n,\omega\gamma,0}^{m+\gamma
,\frac{m+\gamma}{m}}(\overline{\Omega}_{T})}. \label{0000001}
\end{equation}

And for $|\alpha|<m-n$

\begin{equation}
|D_{x}^{\alpha}u|_{\overline{\Omega}_{T}}^{(\gamma,\gamma/m)}\leq CT^{\delta
}\left\Vert u\right\Vert _{C_{n,\omega\gamma,0}^{m+\gamma,\frac{m+\gamma}{m}%
}(\overline{\Omega}_{T})}. \label{0000001.1}%
\end{equation}

\end{proposition}

\begin{proof}

First of all, since $D_{x}^{\alpha}u(x,0)\equiv0$,

\begin{equation}
|d(x)^{n}D_{x}^{\alpha}u|_{\overline{\Omega}_{T}}^{(0)}\leq C\left\langle
d(x)^{n-j}D_{x}^{\alpha}u\right\rangle _{t,\overline{\Omega}_{T}}^{(\gamma
/m)}T^{\gamma/m}\leq CT^{\gamma/m}\left\Vert u\right\Vert _{C_{n,\omega
\gamma,0}^{m+\gamma,\frac{m+\gamma}{m}}(\overline{\Omega}_{T})}. \label{s1.9.2}%
\end{equation}

Further, let $t,\overline{t}\in\lbrack0,T]$. Then%

\[
\frac{|d(x)^{n}D_{x}^{\alpha}u(x,\overline{t})-d(x)^{n}D_{x}^{\alpha}%
u(x,t)|}{|\overline{t}-t|^{\gamma/m}}=
\]

\[
=d^{j\omega}(x)\frac{|d(x)^{n-j\omega
}D_{x}^{\alpha}u(x,\overline{t})-d(x)^{n-j\omega}D_{x}^{\alpha}u(x,t)|}%
{|\overline{t}-t|^{\frac{\gamma+j}{m}}}|\overline{t}-t|^{\frac{j}{m}}\leq
\]

\[
\leq CT^{j/m}\left\langle d(x)^{n-j\omega}D_{x}^{\alpha}u\right\rangle
_{t,\overline{\Omega}_{T}}^{(\frac{\gamma+j}{m})}\leq CT^{j/m}\left\Vert
u\right\Vert _{C_{n,\omega\gamma,0}^{m+\gamma,\frac{m+\gamma}{m}}%
(\overline{\Omega}_{T})}.
\]
This means

\begin{equation}
\left\langle d(x)^{n}D_{x}^{\alpha}u\right\rangle _{t,\overline{\Omega}_{T}%
}^{(\gamma/m)}\leq CT^{j/m}\left\Vert u\right\Vert _{C_{n,\omega\gamma
,0}^{m+\gamma,\frac{m+\gamma}{m}}(\overline{\Omega}_{T})}. \label{s1.9.3}%
\end{equation}
Consider now the smoothness with respect to $x$ - variables. Note
that the function $d(x)^{n}D_{x}^{\alpha}u(x,t)$ has bounded
gradient in $x$ - variables (since $|\alpha|=m-j<m$)

\[
\frac{\partial}{\partial x_{i}}d(x)^{n}D_{x}^{\alpha}u(x,t)=n\frac{\partial
d(x)}{\partial x_{i}}\left[  d(x)^{n-1}D_{x}^{\alpha}u(x,t)\right]
+d(x)^{n}D_{x_{i}}D_{x}^{\alpha}u(x,t),
\]
where both terms are bounded in $\overline{\Omega}_{T}$ by
$C\left\Vert
u\right\Vert _{C_{n,\omega\gamma,0}^{m+\gamma,\frac{m+\gamma}{m}}%
(\overline{\Omega}_{T})}$.

Let now $x,\overline{x}\in\overline{\Omega}$. Consider the difference%

\[
A\equiv d(x,\overline{x})^{\omega\gamma}\frac{|d(\overline{x})^{n}%
D_{x}^{\alpha}u(\overline{x},t)-d(x)^{n}D_{x}^{\alpha}u(x,t)|}{|\overline
{x}-x|^{\gamma}}.
\]
If $|\overline{x}-x|^{\gamma}\geq t^{1/m}$ then%

\[
A\leq C\frac{|d(\overline{x})^{n-j\omega}D_{x}^{\alpha}u(\overline
{x},t)|+|d(x)^{n-j\omega}D_{x}^{\alpha}u(x,t)|}{t^{\frac{\gamma+j}{m}}}%
t^{j/m}\leq
\]

\[
\leq C\left\langle d(x)^{n-j\omega}D_{x}^{\alpha}u\right\rangle _{t,\overline
{\Omega}_{T}}^{(\frac{\gamma+j}{m})}T^{j/m}\leq CT^{j/m}\left\Vert
u\right\Vert _{C_{n,\omega\gamma,0}^{m+\gamma,\frac{m+\gamma}{m}}%
(\overline{\Omega}_{T})}.
\]
Let now $|\overline{x}-x|^{\gamma}<t^{1/m}$. Then%

\[
A=\left(  \frac{|d(\overline{x})^{n}D_{x}^{\alpha}u(\overline{x}%
,t)-d(x)^{n}D_{x}^{\alpha}u(x,t)|}{|\overline{x}-x|}\right)  |\overline
{x}-x|^{1-\gamma}\leq
\]

\[
\leq C|\nabla_{x}\left[  d(x)^{n}D_{x}^{\alpha}u(x,t)\right]  |_{\overline
{\Omega}_{T}}^{(0)}T^{\frac{1-\gamma}{m}}\leq CT^{\frac{1-\gamma}{m}%
}\left\Vert u\right\Vert _{C_{n,\omega\gamma,0}^{m+\gamma,\frac{m+\gamma}{m}%
}(\overline{\Omega}_{T})}.
\]
Thus, we have proved that

\begin{equation}
\left\langle d(x)^{n}D_{x}^{\alpha}u\right\rangle _{\omega\gamma
,x,\overline{\Omega}_{T}}^{(\gamma)}\leq CT^{\delta}\left\Vert u\right\Vert
_{C_{n,\omega\gamma,0}^{m+\gamma,\frac{m+\gamma}{m}}(\overline{\Omega}_{T})},
\label{s1.9.4}
\end{equation}
where $\delta=(1-\gamma)/m$. The estimate \eqref{0000001} follows
from \eqref{s1.9.2}- \eqref{s1.9.4}.

Let now $|\alpha|<m-n$. If $n$ is an integer, then $D_{x}^{\alpha}u$
either has bounded derivatives in $x$ of order not grater then$
m-n$ or such derivatives has the H\"{o}lder property with
arbitrary exponent $\gamma^{\prime}\in(\gamma,0)$ (if
$D_{x_{N}}^{m-n}u$ is not bounded). In both cases the proof of
\eqref{0000001.1} is completely analogous to the proof of
\eqref{0000001}.

\end{proof}

Note that the above property for unweighted space is well known -
\cite{LSU}, \cite{Bizh}.

\end{document}